\documentclass{article}

\usepackage{amsfonts, amssymb, amsmath, amsthm, epsf, cite}

\newtheorem{theorem}{Theorem}[section]
\newtheorem{lemma}{Lemma}[section]
\newtheorem{corollary}{Corollary}[section]
\newtheorem{condition}{Condition}[section]

\theoremstyle{definition}
\newtheorem{definition}{Definition}[section]
\newtheorem{example}{Example}[section]

\theoremstyle{remark}
\newtheorem{remark}{Remark}[section]

\numberwithin{equation}{section} \numberwithin{figure}{section}

\begin{document}

\renewcommand{\Im}{{\rm Im\,}}
\renewcommand{\Re}{{\rm Re\,}}
\newcommand{\loc}{{\rm loc}}
\newcommand{\ind}{{\rm ind\,}}
\newcommand{\supp}{{\rm supp\,}}
\newcommand{\Dom}{{\rm D}}
\renewcommand{\ker}{{\rm ker\,}}
\renewcommand{\dim}{{\rm dim\,}}
\newcommand{\codim}{{\rm codim\,}}
\newcommand{\Orb}{{\rm Orb}}
\newcommand{\rank}{{\rm rank\,}}
\newcommand{\const}{{\rm const}}
\newcommand{\dist}{{\rm dist}}
\newcommand{\n}{{|\!|\!|}}
\newcommand{\Bn}{{\Big|\!\Big|\!\Big|}}

%[Solvability of Nonlocal Elliptic Problems]
\title{On the Fredholm and Unique Solvability of Nonlocal Elliptic Problems in Multidimensional
Domains}

\author{Pavel Gurevich, Alexander Skubachevskii}

\date{}

%\thanks{Supported by the
%Russian Foundation for Basic Research (project No.~04-01-00256)
%and the Russian President Grant (project No.~MK-980.2005.1). The
%first author was also supported by the Alexander von Humboldt
%Foundation.}

% \email{pavel.gurevich@iwr.uni-heidelberg.de}

% author two information
%\author{A.~L. Skubachevskii}
%\address{Peoples' Friendship University of Russia,
%Department of Differential Equations and Mathematical Physics, Miklukho-Maklaya str. 6,
%117198  Moscow, Russia}
%%\curraddr{}
%\email{skub@lector.ru}

% This form of \subjclass will not work if you are using a version
% of amsart.cls older than version 2.0 (December 1999).
%\subjclass[2000]{35J40}

%\keywords{Elliptic problems,  nonlocal boundary conditions, solvability, nonsmooth
%domains}

%\date{}

\maketitle

\abstract{We consider elliptic equations of order $2m$ in a
bounded domain $Q\subset\mathbb R^n$ with nonlocal boundary-value
conditions connecting the values of a solution and its derivatives
on $(n-1)$-dimensional smooth manifolds $\Gamma_i$ with the values
on manifolds $\omega_{i}(\Gamma_i)$, where
$\bigcup_i\overline{\Gamma_i}=\partial Q$ is a boundary of $Q$ and
$\omega_i$ are $C^\infty$ diffeomorphisms. By proving a priori
estimates for solutions and  constructing a right regularizer, we
show the Fredholm solvability in weighted space. For nonlocal
elliptic problems with a parameter,   we prove the unique
solvability.}

\setcounter{tocdepth}{1} \tableofcontents

\section*{Introduction}

We consider elliptic equations of order $2m$ in a bounded domain $Q\subset\mathbb R^n$
with nonlocal boundary-value conditions connecting the values of a solution and its
derivatives on $(n-1)$-dimensional smooth manifolds $\Gamma_i$ with the values on
manifolds $\omega_{i}(\Gamma_i)$, where $\bigcup_i\overline{\Gamma_i}=\partial Q$ is a
boundary of $Q$ and $\omega_i$ are $C^\infty$ diffeomorphisms. The presence of nonlocal
terms leads to appearing power-law singularities of solutions and their derivatives at
the points of the set $\mathcal K_1=\bigcup_i(\overline{\Gamma_i}\setminus\Gamma_i)$
(which is called the {\it set of conjugation points}). Therefore, nonlocal elliptic
problems are naturally studied in weighted spaces $H_a^{l+2m}(Q)$, $a\in\mathbb R$,
$l\ge0$ is an integer (see definition~\eqref{1.5}), originally introduced in the theory
of elliptic problems in nonsmooth domains~\cite{Kondr}. Because of the transformations
$\omega_i$ occurring in nonlocal terms, the points of the set $\mathcal K_1$ turn out to
be connected with the points of the set
$$
\Big\{\bigcup\limits_i \omega_i(\mathcal K_1)\Big\}\cup \Big\{\bigcup\limits_{i,j}
\omega_j(\omega_i(\mathcal K_1)\cap\Gamma_j)\Big\}.
$$
The latter points belong to $Q$ or $\partial Q$. Therefore, we must consider certain
consistency conditions at the points of the set
$$
\mathcal K=\mathcal K_1\cup\Big\{\bigcup\limits_i \omega_i(\mathcal K_1)\Big\}\cup
\Big\{\bigcup\limits_{i,j} \omega_j(\omega_i(\mathcal K_1)\cap\Gamma_j)\Big\}.
$$
The following two approaches are possible. First, one can consider all the points of the
set $\mathcal K$ as singular points in the definition of weighted spaces, which allows
one to study nonlocal elliptic problems for any value of the parameter $a\in\mathbb R$
(see~\cite{SkMs86} for the case $n=2$). Second, one can assume that only the points of
the set $\mathcal K_1$ or the set $\mathcal K\cap\partial Q$ are singular points in the
definition of weighted spaces, which allows one to study nonlocal elliptic problems only
for $a>l+2m-1$ (see~\cite{SkDu91,SkJMAA}).

It is proved in~\cite{MP,NP} that ``local'' elliptic problems in bounded domains have the
Fredholm property (see Definition~\ref{defFredholm}) if some model elliptic operators
(depending on a parameter $\omega$) in plane angles have a trivial kernel and cokernel
for all $\omega\in S^{n-3}$, where
$$
S^{n-3}=\{\omega\in\mathbb R^{n-2}:\ |\omega|=1\}.
$$

Similarly, elliptic operators in $\mathbb R^2\setminus\{0\}$ with the parameter
$\omega\in S^{n-3}$ arise if the points of the set $\mathcal K\cap Q$ are considered as
singular points in the definition of weighted spaces. However, we prove in this paper
that these operators are not isomorphisms, see Sec.~\ref{sec3}. Therefore, unlike the
case of plane domains, if $n\ge3$ and $\mathcal K\cap Q\ne\varnothing$, only the points
of the set $\mathcal K_1$ or the set $\mathcal K\cap\partial Q$ can be considered as
singular points. This leads to the restriction $a>l+2m-1$ (see Sec.~\ref{subsec4.1} for
details).

The paper is organized as follows.

The setting of nonlocal elliptic problems is presented in Sec.~\ref{sec1}. In the same
section, we define model problems in dihedral angles and introduce function spaces. In
Sec.~\ref{sec2}, we consider the solvability of nonlocal problems in dihedral angles. In
particular, we give an example of a nonlocal elliptic problem in a dihedral angle which
is uniquely solvable on weighted spaces for $0\le a\le 2$. In Sec.~\ref{sec3}, we show
that an elliptic operator of order $2m$ acting from $H_a^{l+2m}(\mathbb R^n)$ to
$H_a^{l}(\mathbb R^n)$ is not an isomorphism for any $a\in\mathbb R$ and integer $l\ge0$.
Here we suppose that the points of the set
$$
\mathcal P=\{x=(y,z)\in\mathbb R^n:\ y=0,\ z\in\mathbb R^{n-2}\}
$$
are singular in the definition of weighted spaces $H_a^{l+2m}(\mathbb R^n)$ and
$H_a^{l}(\mathbb R^n)$. In Sec.~\ref{sec4}, we prove a priori estimates for solutions of
nonlocal problems in bounded domains. In Sec.~\ref{sec5}, we construct a right
regularizer for those problems. Thus, we prove a theorem on the Fredholm solvability of
nonlocal elliptic problems in bounded domains. Section~\ref{sec6} is devoted to
generalizations of nonlocal elliptic problems to the case of nonlocal terms supported on
the manifolds $\omega_{is}(\Gamma_i)$ near the set $\mathcal K_1$ and abstract nonlocal
terms supported outside the set $\mathcal K_1$. In Secs.~\ref{sec7} and~\ref{sec8}, we
prove the unique solvability of nonlocal elliptic problems with a parameter.

Note that it was A.~V. Bitzadze and A.~A. Samarskii~\cite{BitzSam} who first considered
an elliptic equation with nonlocal conditions imposed on the shifts of different parts of
the boundary of rectangular. In the general situation, they formulated this problem as an
unsolved problem. Solvability and regularity of solutions for higher-order elliptic
equations with general nonlocal conditions supported near the boundary were studied
in~\cite{SkMs86,SkDu90,SkDu91,SkJMAA}. Second-order elliptic equations with nonlocal
conditions near the boundary were also considered in~\cite{Kishkis,GM,Gusch}. One can
find various applications of nonlocal elliptic problems as well as comprehensive
bibliography of the question in~\cite{SkMs86,SkJMAA,SkBook,GurRJMP03}.

This research was partially supported by the
Russian Foundation for Basic Research (project No.~04-01-00256)
and the Russian President Grant (project No.~MK-980.2005.1). The
first author was also supported by the Alexander von Humboldt
Foundation.

\section{Setting of Nonlocal Elliptic Problems}\label{sec1}

\subsection{}\label{subsec1.1}
Let $Q\subset\mathbb R^n$ ($n\ge3$) be a bounded domain with boundary $\partial
Q=\bigcup\limits_{i=1}^{N_0}\overline{\Gamma_i}$, where $\Gamma_i$ are open connected, in
the topology of $\partial Q$, $(n-1)$-dimensional $C^\infty$ manifolds. Assume that, in a
neighborhood of each point $g\in\overline{\Gamma_i}\setminus\Gamma_i$, the domain $Q$ is
diffeomorphic to an $n$-dimensional dihedral angle
$$
\Theta=\{x=(y,z)\in\mathbb R^n:\ d_1<\varphi<d_2,\ z\in\mathbb R^{n-2}\},
$$
where $r,\varphi$ are the polar coordinates of the point $y\in\mathbb R^2$, $d_j=d_j(g)$,
$j=1,2$.

Introduce the differential operators
$$
A(x,D)=\sum\limits_{|\alpha|\le 2m}a_\alpha(x)D^\alpha,\qquad B_{i\mu
s}(x,D)=\sum\limits_{|\alpha|\le m_{i\mu}}b_{i\mu s\alpha}(x)D^\alpha,
$$
where $a_\alpha, b_{i\mu s\alpha}\in C^\infty(\mathbb R^n)$ are complex-valued functions
($i=1,\dots,N_0$; $\mu=1,\dots,m$; $s=0,\dots,S_i$), $m_{i\mu}\le 2m-1$,
$\alpha=(\alpha_1,\dots,\alpha_n)$, $|\alpha|=\alpha_1+\dots+\alpha_n$,
$D^\alpha=D_1^{\alpha_1}\dots D_n^{\alpha_n}$, $D_j=-i\partial/\partial x_j$. If it is
necessary to indicate the variables with respect to which a function $u$ is
differentiated, we write $D_y^\alpha u$, $D_z^\alpha u$, etc.

Let $\omega_{is}$ ($i=1,\dots,N_0$; $s=1,\dots,S_i$) denote a $C^\infty$ diffeomorphism
mapping some neighborhood $\Omega_i$ of the manifold $\Gamma_i$ onto the set
$\omega_{is}(\Omega_i)$ in such a way that $\omega_{is}(\Gamma_i)\subset Q$. Assume that
the set
\begin{equation}\label{1.1}
\mathcal K=\Big\{\bigcup\limits_i(\overline{\Gamma_i}\setminus\Gamma_i)\Big\}\cup
\Big\{\bigcup\limits_{i,s}\omega_{is}(\overline{\Gamma_i}\setminus\Gamma_i)\Big\}\cup
\Big\{\bigcup\limits_{j,p}\bigcup\limits_{i,s}
\omega_{jp}\big(\omega_{is}(\overline{\Gamma_i}\setminus\Gamma_i)\cap\Gamma_j\big)\Big\}
\end{equation}
can be represented as follows:
\begin{equation}\label{1.2}
\mathcal K=\mathcal K_1\cup\mathcal K_2\cup\mathcal K_3,
\end{equation}
where $\mathcal K_1=\bigcup_\nu\mathcal K_{1\nu}=\partial Q\setminus\bigcup_i\Gamma_i$,
$\mathcal K_2=\bigcup_\nu\mathcal K_{2\nu}\subset\bigcup_i\Gamma_i$, $\mathcal
K_3=\bigcup_\nu\mathcal K_{3\nu}\subset Q$ ($\nu=1,\dots,N_j$; $j=1,2,3$), $\mathcal
K_{j\nu}$ are mutually disjoint $(n-2)$-dimensional connected $C^\infty$ manifolds
without a boundary. In particular, the sets $\mathcal K_2$ and $\mathcal K_3$ may be
empty.

We study the following nonlocal problem:
\begin{equation}\label{1.3}
Au\equiv A(x,D)u(x)=f_0(x),\qquad x\in Q,
\end{equation}
\begin{equation}\label{1.4}
B_{i\mu}u\equiv\sum\limits_{s=0}^{S_i}(B_{i\mu
s}(x,D)u)(\omega_{is}(x))|_{\Gamma_i}=f_{i\mu}(x),\qquad x\in\Gamma_i;\ i=1,\dots,N_0;\
\mu=1,\dots,m,
\end{equation}
where $(B_{i\mu s}(x,D)u)(\omega_{is}(x))|_{\Gamma_i}=B_{i\mu s}(\hat x,D)u(\hat
x)|_{\hat x=\omega_{is}(x),\ x\in\Gamma_i}$, $\omega_{i0}(x)\equiv x$.

We assume throughout that the operators $A(x,D)$ and $B_{i\mu 0}(x,D)$ satisfy the
following conditions.
\begin{condition}\label{cond1.1}
The operator $A(x,D)$ is properly elliptic for each $x\in\overline{Q}$.
\end{condition}
\begin{condition}\label{cond1.2}
The system $\{B_{i\mu 0}(x,D)\}_{\mu=1}^m$ satisfies the Lopatinskii condition with
respect to $A(x,D)$ and is normal for all $i=1,\dots,N_0$ and $x\in\overline{\Gamma_i}$
\rm{(}\it{see} \rm{\cite[Chap.~2, Sec.~1.4]{LM})}.
\end{condition}

Denote by $\omega_{is}^{+1}$ the transformation
$\omega_{is}:\Omega_i\to\omega_{is}(\Omega_i)$ and by
$\omega_{is}^{-1}:\omega_{is}(\Omega_i)\to\Omega_i$ the inverse transformation.

\begin{definition}
The set of all points that can be obtained from a point $g\in\mathcal K_1$ by
consecutively applying to it the transformations $\omega_{is}^{+1}$ or $\omega_{is}^{-1}$
mapping the points of the set $\mathcal K_1$ to $\mathcal K_1$ is called an {\em orbit}
of the point $g$ and denoted by $\mathcal O(g)$.
\end{definition}

In other words, the orbit $\mathcal O(g)$ consists of the point $g\in\mathcal K_1$ and
the points that can be obtained from $g$ in the following way: given a point
$h\in\mathcal O(g)$, the point $\omega_{is}(h)$ belongs to $\mathcal O(g)$ iff
$h\in\overline{\Gamma_i}\cap\mathcal K_1$ and $\omega_{is}(h)\in\mathcal K_1$, while the
point $\omega_{is}^{-1}(h)$ belongs to $\mathcal O(g)$ iff
$h\in\omega_{is}(\overline{\Gamma_i})\cap\mathcal K_1$.

Assume that the following condition holds.
\begin{condition}[finiteness of the orbits]\label{cond1.3}
\begin{enumerate}
\item
For each point $g\in\mathcal K_1$, the orbit $\mathcal O(g)$ consists of finitely many
points $g_j$, $j=1,\dots,N=N(g)$.
\item
There are neighborhoods $\hat V(g_j)\subset V(g_j)\subset\mathbb R^n\setminus(\mathcal
K_2\cup\mathcal K_3)$, $V(g_j)\cap V(g_p)=\varnothing$ $(j\ne p)$, of the points
$g_j\in\mathcal O(g)$ such that, if $g_j\in\overline{\Gamma_i}$ and
$\omega_{is}(g_j)=g_p$, then $\hat V(g_j)\subset\Omega_i$ and $\omega_{is}(\hat
V(g_j))\subset V(g_p)$.
\end{enumerate}
\end{condition}

The following condition means that the support of nonlocal terms intersects the boundary
at the points of the set $\mathcal K_1$ in a nontangential way.
\begin{condition}[nontangential approach]\label{cond1.4}
For each point $g\in\mathcal K_1$ and $j=1,\dots,N(g)$, there exists a smooth
nondegenerate change of variables $x\to x'=x'(g,j)$ such that the neighborhood $V(g_j)$
reduces, under this change of variables, to some neighborhood of the origin $V(0)$ and,
moreover$:$
\begin{enumerate}
\item
The sets $Q\cap V(g_j)$ and $\Gamma_i\cap V(g_j)$ reduce to the intersection of a
dihedral angle $\Theta_j$ with $V(0)$ and to the intersection of a side $\Gamma_{j\rho}$
$(\rho=1$ or $\rho=2)$ of $\Theta_j$ with $V(0)$, respectively$;$
\item
Each transformation $\omega_{is}(x)$, for $x\in\hat V(g_j)$, $g_j\in\overline{\Gamma_i}$,
reduces to the composition of rotation and homothety on the plane $\{y'\}$ in the new
variables $x'=(y',z')$, where $y'\in\mathbb R^2$ and $z'\in\mathbb R^{n-2}$.
\end{enumerate}
\end{condition}

\begin{remark}\label{remN=1}
If $N(g)=1$, then the orbit of $g\in\mathcal K_1$ consists of the unique point $g$. This
is the case iff the following two conditions are fulfilled:
\begin{enumerate}
\item
$\omega_{is}(g)=g$ for all $i$ and $s$ such that $g\in\overline{\Gamma_i}\cap\mathcal
K_1$ and $\omega_{is}(g)\in\mathcal K_1$;
\item
there do not exist indices $i$ and $s$ and a point $h\in\overline{\Gamma_i}\cap\mathcal
K_1$ such that $h\ne g$ and $\omega_{is}(h)=g$.
\end{enumerate}
\end{remark}

\subsection{}
For any domain $\Omega$, denote by $W^k(\Omega)=W_2^k(\Omega)$ ($k\ge0$ is an integer)
the Sobolev space. Denote by $W^{k-1/2}(\Gamma)$ ($k\ge1$ is an integer) the space of
traces on a smooth $(n-1)$-dimensional manifold $\Gamma\subset\overline{\Omega}$, with
the norm
$$
\|\psi\|_{W^{k-1/2}(\Gamma)}=\inf\|v\|_{W^k(\Omega)}\qquad (v\in W^k(\Omega):\
v|_\Gamma=\psi).
$$

If $X$ is a domain in $\mathbb R^n$, $n=1,2,\dots$, we denote by $C_0^\infty(X)$ the set
of functions infinitely differentiable on $\overline{ X}$ and compactly supported on $X$.
If $M$ is a union of finitely many $(n-l)$-dimensional manifolds ($l=1,\dots,n$) lying in
$\overline X$, we denote by $C_0^\infty(\overline X\setminus M)$ the set of functions
infinitely differentiable on $\overline{ X}$ and compactly supported on $\overline
X\setminus M$.

Now we introduce different weighted spaces for different domains $\Omega$. Consider the
following cases:
\begin{enumerate}
\item $\Omega=Q$; denote either $K=\mathcal K_1$  or
$K=\mathcal K_1\cup\mathcal K_2$ (cf. Condition~\ref{cond4.1} in Sec.~\ref{sec4}), and
let $\rho(x)$ be a function such that $\rho\in C^\infty(\mathbb R^n\setminus K)$,
$\rho(x)>0$ for $x\in\mathbb R^n\setminus K$ and it is equivalent, in a neighborhood of
the set $K$, to the distance from a point $x\in\Omega$ to the set $K$;
\item $\Omega\subset\mathbb R^n$ is a bounded domain with
boundary $\partial\Omega\in C^\infty$; denote by $K$ some $(n-2)$-dimensional manifold of
class $C^\infty$ lying in $\overline{\Omega}$, and let $\rho$ be the same function as in
case~1;
\item
$\Omega$ is an $n$-dimensional dihedral angle $\Theta$; denote $K=\mathcal P$, where
$$
\mathcal P=\{x=(y,z)\in\mathbb R^n:\ y=0,\ z\in\mathbb R^{n-2}\},
$$
and let $\rho(x)=|y|$;
\item
$\Omega=\mathbb R^n$; denote $K=\mathcal P$, and let $\rho(x)=|y|$.
\end{enumerate}
Introduce the weighted space $H_a^k(\Omega)=H_a^k(\Omega,K)$ as the completion of the set
$C_0^\infty(\overline{\Omega}\setminus K)$ with respect to the norm
\begin{equation}\label{1.5}
\|u\|_{H_a^k(\Omega)}=\bigg(\sum\limits_{|\alpha|\le k}\,\int\limits_\Omega
\rho^{2(a-k+|\alpha|)}|D^\alpha u|^2\, dx\bigg)^{1/2},
\end{equation}
where $k\ge0$ is an integer and $a\in\mathbb R$.

Denote by $H_a^{k-1/2}(\Gamma)$ ($k\ge1$ is an integer) the space of traces on a smooth
$(n-1)$-dimensional manifold $\Gamma\subset\overline{\Omega}$, with the norm
\begin{equation}\label{1.6}
\|\psi\|_{H_a^{k-1/2}(\Gamma)}=\inf\|v\|_{H_a^k(\Omega)}\qquad (v\in H_a^k(\Omega):\
v|_\Gamma=\psi).
\end{equation}

One can similarly introduce the weighted spaces $H_a^k(\Omega)$ and $H_a^{k-1/2}(\Gamma)$
for $n=2$. In particular, we set $K=\{0\}$ for $\Omega=\theta=\{y\in\mathbb R^2:\
d_1<\varphi<d_2,\ 0<r\}$ or $\Omega=\mathbb R^2$.

In what follows, we assume that $u\in H_a^{l+2m}(Q)$ and $f=\{f_0,f_{i\mu}\}\in\mathcal
H_a^l(Q,\Gamma)$ in problem~\eqref{1.3}, \eqref{1.4}, where
$$
\mathcal H_a^l(Q,\Gamma)=H_a^l(Q)\times \prod\limits_{i,\mu}
H_a^{l+2m-m_{i\mu}-1/2}(\Gamma_i)
$$
and $l\ge0$ is an integer.

\subsection{}
Fix an arbitrary point $g\in\mathcal K_1$. By Condition~\ref{cond1.3}, the orbit
$\mathcal O(g)$ consists of finitely many points $g_j$, $j=1,\dots,N=N(g)$. We now reduce
problem~\eqref{1.3}, \eqref{1.4} to a system of $N$ elliptic equations in dihedral angles
with nonlocal boundary-value conditions. To do this, we suppose that
$$
\supp u\subset\Big(\bigcup\limits_j\hat V(g_j)\Big)\cap\overline{Q}.
$$
Denote by $u_j(x)$ the function $u(x)$ for $x\in Q\cap V(g_j)$. If
$g_j\in\overline{\Gamma_i}$ and $x\in\hat V(g_j)$, then $\omega_{is}(x)\in V(g_p)$ for
some $p$, $1\le p\le N$, by Condition~\ref{cond1.3}. Denote the function
$u(\omega_{is}(x))$ by $u_p(\omega_{is}(x))$. It is clear that
$u(\omega_{i0}(x))=u(x)=u_j(x)$. In the above notation, problem~\eqref{1.3}, \eqref{1.4}
takes the form
\begin{equation}\label{1.7}
A(x,D)u_j(x)=f_0(x),\qquad x\in Q\cap\hat V(g_j),
\end{equation}
\begin{equation}\label{1.8}
\begin{aligned}
\sum\limits_{s\in S_{ij}^g}(B_{i\mu
s}(x,D)u_p)(\omega_{is}(x))|_{\Gamma_i}=f_{i\mu}(x),\qquad
x\in\hat V(g_j)\cap\Gamma_i;\\
i\in\{1\le i\le N_0:\ \hat V(g_j)\cap\Gamma_i\ne\varnothing\};\ j=1,\dots,N;\
\mu=1,\dots,m,
\end{aligned}
\end{equation}
where $S_{ij}^g=\{0\le s\le S_i:\ \omega_{is}(g_j)=g_p\in\mathcal O(g)\ \text{for some}\
p=1,\dots,N\}$.

Using the change of variables $x\to x'(g,j)$ from Sec.~\ref{subsec1.1}, we introduce the
functions $v_j(x')=u_j(x(x'))$. By Condition~\ref{cond1.4}, problem~\eqref{1.7},
\eqref{1.8} takes the following form:
\begin{equation}\label{1.9}
A_j(x',D_{y'},D_{z'})v_j(x')=f_j(x'),\qquad x'\in \Theta_j;\ j=1,\dots,N,
\end{equation}
\begin{equation}\label{1.10}
\begin{aligned}
\sum\limits_{k=1}^N\sum\limits_{s=0}^{S_{j\rho k}}(B_{j\rho\mu
ks}(x',D_{y'},D_{z'})v_k)(\mathcal G_{j\rho
ks}y',z')|_{\Gamma_{j\rho}}=f_{j\rho\mu}(x'),\qquad
x'\in\Gamma_{j\rho};\\
j=1,\dots,N;\ \rho=1,2;\ \mu=1,\dots,m.
\end{aligned}
\end{equation}
Here the operators $A_j$ and $B_{j\rho\mu ks}$ have variable coefficients of class
$C^\infty$;
$$
\begin{aligned}
\Theta_j&=\{x'=(y',z'): 0<d_{j1}<\varphi<d_{j2},\ z'\in\mathbb
R^{n-2}\},\\
 \Gamma_{j\rho}&=\{x'=(y',z'): \varphi=d_{j\rho},\
z'\in\mathbb R^{n-2}\};
\end{aligned}
$$
$\mathcal G_{j\rho ks}$ is the operator of rotation by an angle $\varphi_{j\rho ks}$ and
homothety with a coefficient $\chi_{j\rho ks}$ in the $y'$-plane so that
$d_{k1}<d_{j\rho}+\varphi_{j\rho ks}<d_{k2}$ and $0<\chi_{j\rho ks}$ for $(k,s)\ne(j,0)$,
while $\varphi_{j\rho j0}=0$ and $\chi_{j\rho j0}=1$ (i.e., $\mathcal G_{j\rho
j0}y'\equiv y'$); $v=(v_1,\dots v_N)$.

\begin{remark}
If $g\in\overline{\Gamma_i}$, $N=N(g)=1$ (cf. Remark~\ref{remN=1}), and
$\omega_{is}(g)\ne g$ for all $s=1,\dots,S_i$, then model problem~\eqref{1.9},
\eqref{1.10} contains no nonlocal terms due to the fact that the manifolds $\mathcal
K_{j\nu}$ are mutually disjoint.
\end{remark}

Introduce the following spaces of vector-valued functions:
$$
\mathcal H_a^k(\Theta)=\prod\limits_{j=1}^N H_a^k(\Theta_j),\quad \mathcal
H_a^l(\Theta,\Gamma)=\mathcal
H_a^l(\Theta)\times\prod\limits_{j=1}^N\prod\limits_{\rho=1,2}\prod\limits_{\mu=1}^m
H_a^{l+2m-m_{j\rho\mu}-1/2}(\Gamma_{j\rho}),
$$
where $m_{j\rho\mu}$ is the order of the operator $B_{j\rho\mu ks}(x',D_{y'},D_{z'})$.

Consider the linear bounded operator $\mathcal L_g: \mathcal
H_a^{l+2m}(\Theta)\to\mathcal H_a^l(\Theta,\Gamma)$ given by
\begin{equation}\label{1.11}
\mathcal L_g v=\Big\{A_j(D_{y'},D_{z'})v_j(y',z'),\
\sum\limits_{k=1}^N\sum\limits_{s=0}^{S_{j\rho k}}(B_{j\rho\mu
ks}(D_{y'},D_{z'})v_k)(\mathcal G_{j\rho ks}y',z')|_{\Gamma_{j\rho}}\Big\},
\end{equation}
where $A_j(D_{y'},D_{z'})$ and $B_{j\rho\mu ks}(D_{y'},D_{z'})$ are principal homogeneous
parts of the operators $A_j(0,D_{y'},D_{z'})$ and $B_{j\rho\mu ks}(0,D_{y'},D_{z'})$,
respectively. The subscript $g$ means that the operator $\mathcal L_g$ depends on the
choice of the point $g\in\mathcal K_1$ (and therefore, it depends on the orbit $\mathcal
O(g)$). Clearly, each of the operators $A_j(D_{y'},D_{z'})$ is properly elliptic, while
the system $\{B_{j\rho\mu j0}(D_{y'},D_{z'})\}_{\mu=1}^m$ satisfies the Lopatinskii
condition with respect to $A_j(D_{y'},D_{z'})$ and is normal for all $j=1,\dots,N$ and
$\rho=1,2$.

\begin{example}\label{ex1.1}
Let $Q\subset \mathbb R^3$ be a bounded domain with boundary $\partial Q\in C^\infty$
which is a surface of revolution about the axis $x_3$. Set
$P=\{(0,0,3)\}\cup\{(0,0,-3)\}\cup\{x:\ x_3=0,\ \sqrt{x_1^2+x_2^2}=3\}$ and
$P^{1/4}=\{x:\ \dist (x,P)<1/4\}$. Assume that, outside the set $P^{1/4}$, the boundary
$\partial Q$ coincides with the boundary of the domain
$$
\left\{x:\ x_3<3-\sqrt{x_1^2+x_2^2}\right\}\cap\left\{x:\
x_3>-3+\sqrt{x_1^2+x_2^2}\right\}.
$$
Denote
$$
\Gamma_1=\{x\in\partial Q:\ x_3<-2\},\quad \Gamma_2=\{x\in\partial Q:\ x_3>2\},\quad
\Gamma_3=\partial Q\setminus(\overline{\Gamma_1}\cup\overline{\Gamma_2}).
$$

We consider the following nonlocal boundary-value problem:
\begin{equation}\label{ex1.1-1.12}
-\Delta u=f_0(x),\qquad x\in Q,
\end{equation}
\begin{equation}\label{ex1.1-1.13}
\begin{aligned}[]
[u(x)+\alpha_j u(x+h_j)+\beta_j u(\mathcal G_\pi
x+h_j)]|_{\Gamma_j}&=0,\qquad j=1,2,\\
u(x)|_{\Gamma_3}&=0,
\end{aligned}
\end{equation}
where $\alpha_j,\beta_j\in\mathbb R$, $h_j=(-1)^{j+1}(0,0,4)$, $j=1,2$, and $\mathcal
G_\pi$ is the operator of rotation by the angle $\pi$ about the axis $x_3$. Clearly, we
have (see Fig.~\ref{fig1.1})
\begin{figure}[ht]
{ \hfill\epsfxsize80mm\epsfbox{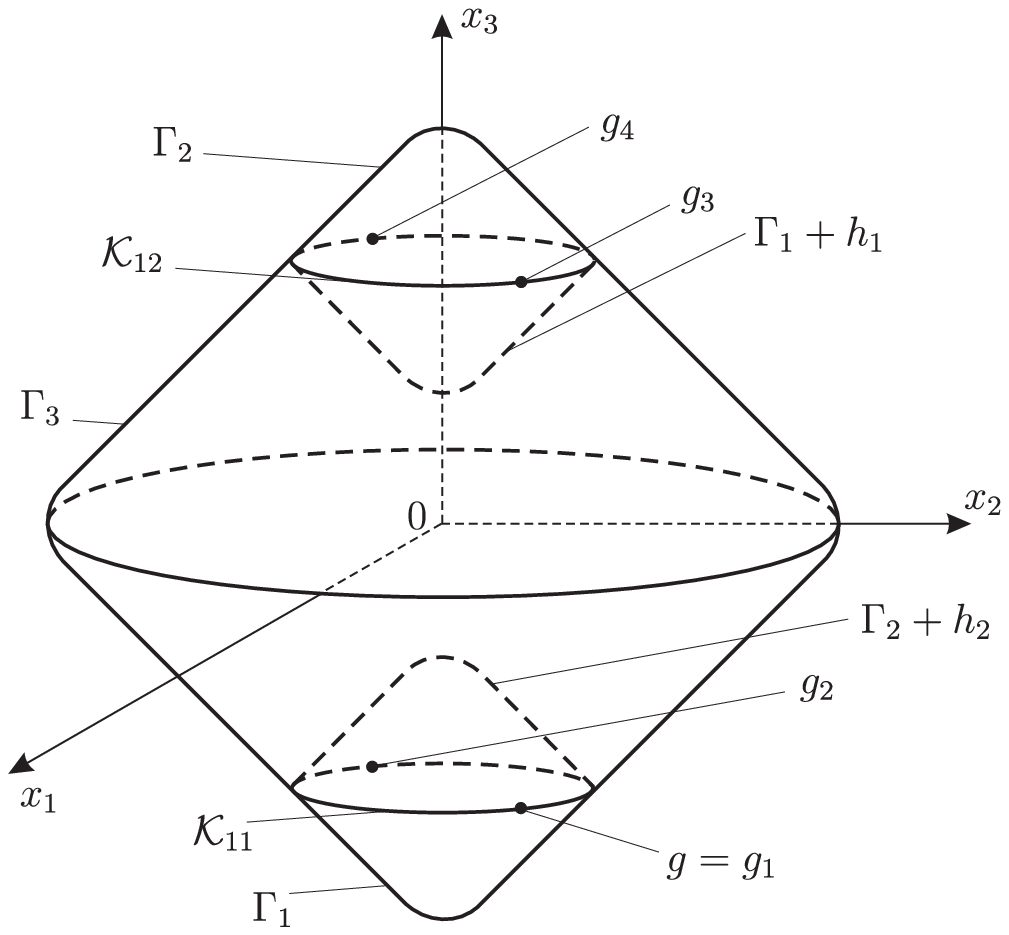}\hfill\ } \caption{Problem~\eqref{ex1.1-1.12},
\eqref{ex1.1-1.13}}\label{fig1.1}
\end{figure}
$$
\mathcal K=\mathcal K_1=\mathcal K_{11}\cup\mathcal K_{12},\qquad \mathcal
K_{1\nu}=\{x\in\partial Q:\ x_3=(-1)^\nu 2\},\ \nu=1,2.
$$

The orbit of each point $g\in\mathcal K_{11}$ consists of the four points: $g_1=g$,
$g_2=\mathcal G_\pi g_1$, $g_3=g_1+h_1$, and $g_4=\mathcal G_\pi g_1+h_1$. Let $\hat
V(g_j)=V(g_j)=\{x:\ |x-g_j|<\varepsilon\}$, where $\varepsilon$ is sufficiently small,
and let $\supp u\subset\left(\bigcup_j V(g_j)\right)\cap\overline{Q}$. For $x\in V(g_j)$,
we introduce the new variables $x'=(y_1',y_2',z')$ by the formulas
$$
y_1'=r-1,\quad y_2'=x_3-g_{j3},\quad z'=\varphi-\varphi_j,
$$
where $r,\varphi,x_3$ and $1,\varphi_j,g_{j3}$ are the cylindrical coordinates of the
points $x$ and $g_j$, respectively. Clearly, the transformation $x\mapsto x'(g,j)$ is
nondegenerate for $x\ne0$ and each open set $V(g_j)$ is taken onto one and the same
neighborhood of the origin $V(0)$ under this transformation. We define the vector-valued
function $v(x')$ such that $v_j(x')=u_j(x(x'))$ for $x'\in V(0)$, where $u_j(x)=u(x)$ for
$x\in V(g_j)\cap Q$. Denote $x'=(y_1',y_2',z')$ by $x=(y_1,y_2,z)$ again. Then the
boundary-value problem~\eqref{ex1.1-1.12}, \eqref{ex1.1-1.13} takes the form (see
Fig.~\ref{fig1.2})
\begin{equation}\label{ex1.1-1.14}
-\dfrac{\partial^2 v_j}{\partial y_1^2}-\dfrac{\partial^2 v_j}{\partial
y_2^2}-\dfrac{1}{(1+y_1)^2}\dfrac{\partial^2 v_j}{\partial
z^2}-\dfrac{1}{1+y_1}\dfrac{\partial v_j}{\partial y_1}=f_j(x),\qquad x\in\Theta_j,\
j=1,\dots,4,
\end{equation}
\begin{equation}\label{ex1.1-1.15}
\begin{aligned}
&v_j|_{\Gamma_{j1}}=0,\qquad j=1,\dots,4,\\
&(v_1+\alpha_1v_3+\beta_1v_4)|_{\Gamma_{12}}=(v_2+\beta_1v_3+\alpha_1v_4)|_{\Gamma_{22}}=0,\\
&(v_3+\alpha_2v_1+\beta_2v_2)|_{\Gamma_{32}}=(v_4+\beta_2v_1+\alpha_2v_2)|_{\Gamma_{42}}=0.
\end{aligned}
\end{equation}
Here
$$
\Theta_1=\Theta_2=\{x\in\mathbb R^3:\ y_2>y_1\},\qquad \Theta_3=\Theta_4=\{x\in\mathbb
R^3:\ y_2<-y_1\},
$$
$$
\Gamma_{11}=\Gamma_{21}=\{x\in\mathbb R^3:\ y_2=y_1,\ y_1>0\},\qquad
\Gamma_{31}=\Gamma_{41}=\{x\in\mathbb R^3:\ y_2=-y_1,\ y_1>0\},
$$
$$
\Gamma_{12}=\Gamma_{22}=\{x\in\mathbb R^3:\ y_2=y_1,\ y_1<0\},\qquad
\Gamma_{32}=\Gamma_{42}=\{x\in\mathbb R^3:\ y_2=-y_1,\ y_1<0\}.
$$

\begin{figure}[ht]
{ \hfill\epsfxsize120mm\epsfbox{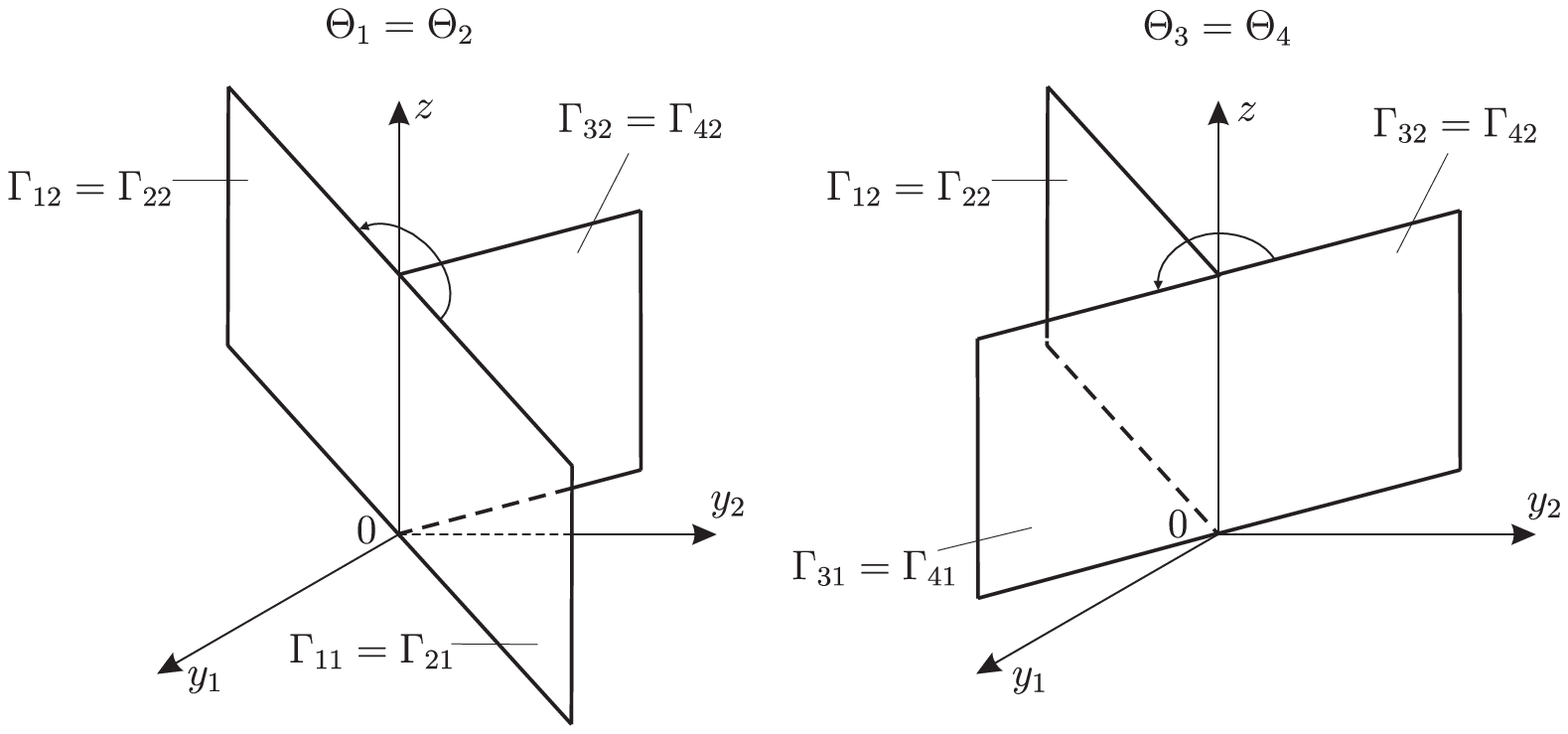}\hfill\ } \caption{Problem~\eqref{ex1.1-1.14},
\eqref{ex1.1-1.15}}\label{fig1.2}
\end{figure}

Clearly, Conditions~\ref{cond1.1}--\ref{cond1.4} hold in this example.

Passing to the principal homogeneous parts in Eqs.~\eqref{ex1.1-1.14} and freezing the
coefficients at the origin, we obtain
$$
-\Delta v_j=f_j(x),\qquad x\in\Theta_j,\ j=1,\dots,4.
$$
Nonlocal boundary conditions~\eqref{ex1.1-1.15} do not change.
\end{example}

\subsection{}\label{subsec1.4}
Now fix an arbitrary point $g\in\mathcal K_2$. Clearly, $g\in\mathcal
K_{2\nu}\cap\Gamma_i$ for some $1\le\nu\le N_2$ and $1\le i\le N_0$. By virtue of the
smoothness of $\Gamma_i$ and $\mathcal K_{2\nu}$, there exists a $C^\infty$
diffeomorphism $x\to x'=x'(g)$ defined on a small neighborhood $V(g)$ of the point $g$,
such that the images of $Q\cap V(g)$ and $\mathcal K_{2\nu}\cap V(g)$ are the
intersection of the half-space $\mathbb R_+^n=\{x:\ |\varphi|<\pi/2,\ z\in\mathbb
R^{n-2}\}$ with some neighborhood $V(0)$ and the intersection of the set $\mathcal P$
with $V(0)$, respectively.

Let $A(D_{y'},D_{z'})$ and $B_{i\mu0}(D_{y'},D_{z'})$ be the principal homogeneous parts
of the operators $A(g,D_{y},D_{z})$ and $B_{i\mu0}(g,D_{y},D_{z})$, respectively, written
in the new coordinates $x'=x'(g)$.

We introduce the linear bounded operator
$$
\begin{aligned}
&\mathcal L_g: H_a^{l+2m}(\mathbb R_+^n)\to\mathcal H_a^l(\mathbb
R_+^n,\Gamma)\\
&\qquad=H_a^l(\mathbb R_+^n)\times H_a^{l+2m-m_{i\mu}-1/2}(\mathbb R_-^{n-1})\times
H_a^{l+2m-m_{i\mu}-1/2}(\mathbb R_+^{n-1})
\end{aligned}
$$
given by
\begin{equation}\label{1.12}
\begin{aligned}
 \mathcal
L_g u & =\big(A(D_{y'},D_{z'})u(y',z'),\\
&\qquad B_{i\mu0}(D_{y'},D_{z'})u(y',z')|_{\varphi=-\pi/2},\
B_{i\mu0}(D_{y'},D_{z'})u(y',z')|_{\varphi=\pi/2}\big),
\end{aligned}
\end{equation}
where $\mathbb R_{\pm}^{n-1}=\{x'=(y',z')\in\mathbb R^n:\ \varphi=\pm\pi/2,\ z'\in\mathbb
R^{n-2}\}$.

\section{Nonlocal Elliptic Problems in Dihedral Angles}\label{sec2}

\subsection{}
When studying nonlocal problems in bounded domains, we will represent the nonlocal
operators as the sum of three operators. The first operator will correspond to nonlocal
terms supported near the set $\mathcal K_1$, the second operator to nonlocal terms
supported outside the set $\mathcal K_1$, and the third one to lower-order terms (compact
perturbations). In this section, we consider a model operator corresponding to the
problem with nonlocal terms supported near the set $\mathcal K_1$.

By using the Fourier transform with respect to $z$, one can reduce the study of the
operator $\mathcal L_g$ in dihedral angles to the study of a model operator $\mathcal
L_g(\omega)$ in plane angles, where $\omega$ is a parameter belonging to the unit sphere
$$
S^{n-3}=\{\omega\in\mathbb R^2:\ |\omega|=1\},
$$
see~\cite{SkDu90,GurGiess}. In this section, we formulate some results (mostly proved
in~\cite{SkDu90,GurGiess}) which we need below and illustrate them by an example. Note
that the Fourier transform approach was earlier proposed for the study of local elliptic
problems in dihedral angles~\cite{MP}.

To introduce the operator $\mathcal L_g(\omega)$, we preliminarily consider weighted
spaces with nonhomogeneous weight. Denote by $E_a^k(\Omega)$ the completion of the set
$C_0^\infty(\overline{\Omega}\setminus\{0\})$ with respect to the norm
$$
\|u\|_{E_a^k(\Omega)}=\bigg(\sum\limits_{|\alpha|\le k}\,\int\limits_\Omega
r^{2a}(r^{2(|\alpha|-k)}+1)|D_y^\alpha u(y)|^2\, dy\bigg)^{1/2},
$$
where either $\Omega=\theta=\{y\in\mathbb R^2:\ d_1<\varphi<d_2\}$ or $\Omega=\mathbb
R^2$; $r,\varphi$ are the polar coordinates of the point $y$; $k\ge0$ is an integer. Let
$\gamma\subset\overline{\Omega}$ be a half-line given by $\gamma=\{y\in\mathbb R^2:\
\varphi=\varphi_0\}$, where $d_1\le\varphi_0\le d_2$ for $\Omega=\theta$. Denote by
$E_a^{k-1/2}(\gamma)$ ($k\ge1$ is an integer) the space of traces on $\gamma$ with the
norm
$$
\|\psi\|_{E_a^{k-1/2}(\gamma)}=\inf\|v\|_{E_a^k(\Omega)}\qquad (v\in E_a^k(\Omega):\
v|_\gamma=\psi).
$$

Introduce the following spaces of vector-valued functions:
$$
\mathcal E_a^k(\theta)=\prod\limits_{j=1}^N E_a^k(\theta_j),\quad \mathcal
E_a^l(\theta,\gamma)=\mathcal E_a^l(\theta)\times\prod\limits_{j=1}^N
\prod\limits_{\rho=1,2}\prod\limits_{\mu=1}^m
E_a^{l+2m-m_{j\rho\mu}-1/2}(\gamma_{j\rho}),
$$
where $\theta_j=\{y\in\mathbb R^2:\ d_{j1}<\varphi<d_{j2}\}$ and
$\gamma_{j\rho}=\{y\in\mathbb R^2:\ \varphi=d_{j\rho}\}$.

For a fixed point $g\in\mathcal K_1$, we consider the linear bounded operator $$\mathcal
L_g(\omega): \mathcal E_a^{l+2m}(\theta)\to\mathcal E_a^l(\theta,\gamma)$$ given by
\begin{equation}\label{2.1}
\mathcal L_g(\omega) V=\Big\{A_j(D_{y},\omega)V_j(y),\ \sum\limits_{k,s}(B_{j\rho\mu
ks}(D_{y},\omega)V_k)(\mathcal G_{j\rho ks}y)|_{\gamma_{j\rho}}\Big\},
\end{equation}
where $\omega\in S^{n-3}$ and $V=(V_1,\dots,V_N)$, cf.~\eqref{1.11}.

\subsection{}
We write the operators $A_j(D_{y},0)$ and $B_{j\rho\mu ks}(D_{y},0)$ in the polar
coordinates:
\begin{equation*}
A_j(D_{y},0)=r^{-2m}\hat A_j(\varphi,D_\varphi,rD_r),\qquad B_{j\rho\mu
ks}(D_{y},0)=r^{-m_{j\rho\mu}}\hat B_{j\rho\mu ks}(\varphi,D_\varphi,rD_r),
\end{equation*}
where $D_\varphi=-i\partial/\partial\varphi$, $D_r=-i\partial/\partial r$.

Introduce the following spaces of vector-valued functions:
$$
\mathcal W^k(d_1,d_2)=\prod\limits_{j=1}^N W^k(d_{j1},d_{j2}),\quad \mathcal
W^l[d_1,d_2]=\mathcal W^l(d_1,d_2)\times\mathbb C^{mN}\times\mathbb C^{mN}.
$$
Consider the analytic operator-valued function $\hat{\mathcal L}_g(\lambda): \mathcal
W^{l+2m}(d_1,d_2)\to\mathcal W^l[d_1,d_2]$ given by
\begin{equation}\label{2.4}
\begin{aligned}
\hat{\mathcal L}_g(\lambda) w & =\Big\{\hat A_j(\varphi,D_\varphi,\lambda)w_j,\\
&\qquad\sum\limits_{k,s}e^{(i\lambda-m_{j\rho\mu})\ln\chi_{j\rho ks}}(\hat B_{j\rho\mu
ks}(\varphi,D_\varphi,\lambda)w_k)(\varphi+\varphi_{j\rho
ks})|_{\varphi=d_{j\rho}}\Big\},
\end{aligned}
\end{equation}
where $w=(w_1,\dots,w_N)$.

By~Lemmas 2.1 and 2.2 in~\cite{SkDu90}, there exists a finite-meromorphic operator-valued
function $\hat{\mathcal R}_g(\lambda): \mathcal W^l[d_1,d_2]\to \mathcal
W^{l+2m}(d_1,d_2)$ such that $\hat{\mathcal L}_g^{-1}(\lambda)=\hat{\mathcal
R}_g(\lambda)$ for any $\lambda$ which is not a pole of $\hat{\mathcal R}_g(\lambda)$.
Moreover, if $\lambda_0=\mu_0+i\nu_0$ is a pole of $\hat{\mathcal R}_g(\lambda)$, then
$\lambda_0$ is an eigenvalue of $\hat{\mathcal L}_g(\lambda)$, and there exists a number
$\delta>0$ such that the set $\{\lambda\in\mathbb C:\ 0<|\Im\lambda-\nu_0|<\delta\}$
contains no eigenvalues of $\hat{\mathcal L}_g(\lambda)$.

\subsection{}\label{subsec2.3}
\begin{definition}\label{defFredholm}
Let $H_1$ and $H_2$ denote Hilbert spaces. A linear bounded operator $L:H_1\to H_2$ is
said to have the {\em Fredholm property} if $\dim\mathcal N(L)<\infty$, $\codim\mathcal
R(L)<\infty$, and $\mathcal R(L)$ is closed, where $\mathcal N(L)$ and $\mathcal R(L)$
are the kernel and the image of the operator $L$, respectively.
\end{definition}

The following theorem shows that spectral properties of the operator-valued function
$\hat{\mathcal L}_g(\lambda)$ affect whether or not the operator $\mathcal L_g(\omega)$
has the Fredholm property.

\begin{theorem}\label{t2.1}
Let Conditions~$\ref{cond1.1}$--$\ref{cond1.4}$ hold. If the line $\Im\lambda=a+1-l-2m$
contains no eigenvalues of the operator-valued function $\hat{\mathcal L}_g(\lambda)$,
then the operator $\mathcal L_g(\omega): \mathcal E_a^{l+2m}(\theta)\to\mathcal
E_a^l(\theta,\gamma)$ has the Fredholm property for all $\omega\in S^{n-3}$.

If the operator $\mathcal L_g(\omega)$ has the Fredholm property for a certain $\omega\in
S^{n-3}$, then the line $\Im\lambda=a+1-l-2m$ contains no eigenvalues of $\hat{\mathcal
L}_g(\lambda)$.
\end{theorem}
Theorem~\ref{t2.1} was proved in~\cite{GurGiess}. This result is a generalization of
Theorem~3.2 in~\cite{SkDu90}, where one additionally assumes that the line
$\Im\lambda=a+1-l-2m$ contains no eigenvalues of the corresponding localized operator
with a parameter $\lambda$.

The following theorem results from Theorems~3.3, 9.2, and~9.3 in~\cite{GurGiess}.
\begin{theorem}\label{t2.2}
Let Conditions~$\ref{cond1.1}$--$\ref{cond1.4}$ hold. Then the operator $\mathcal L_g:
\mathcal H_a^{l+2m}(\Theta)\to\mathcal H_a^l(\Theta,\Gamma)$ is an isomorphism iff the
operator $\mathcal L_g(\omega): \mathcal E_a^{l+2m}(\theta)\to\mathcal
E_a^l(\theta,\gamma)$ is an isomorphism for each $\omega\in S^{n-3}$.
\end{theorem}

Denote
\begin{equation}\label{4.8}
\mathcal L_g' v=\Big\{A_j^0(x,D_{y},D_{z})v_j(y,z),\ \sum\limits_{k,s}(B_{j\rho\mu
ks}^0(x,D_{y},D_{z}) v_k)(\mathcal G_{j\rho ks}y,z)|_{\Gamma_{j\rho}}\Big\},
\end{equation}
where $A_j^0(x,D_{y},D_{z})$ and $B_{j\rho\mu ks}^0(x,D_{y},D_{z})$ are principal
homogeneous parts of the operators $A_j(x,D_{y},D_{z})$ and $B_{j\rho\mu
ks}(x,D_{y},D_{z})$, respectively. Note that $A_j^0(0,D_{y},D_{z})=A_j(D_{y},D_{z})$ and
$B_{j\rho\mu ks}^0(0,D_{y},D_{z})=B_{j\rho\mu ks}(D_{y},D_{z})$.

Let
$$
B_\varepsilon=\{x\in\mathbb R^n:\ |x|<\varepsilon\},\qquad\varepsilon>0,
$$
be a ball of radius $\varepsilon$ centered at the origin.

For each $\delta>0$, we define a function $\eta=\eta_\delta\in C_0^\infty(\mathbb R^n)$
such that $\eta(x)=1$ for $x\in B_{\delta}$, $\supp\eta\subset B_{2\delta}$, and
\begin{equation}\label{2.D_beta_eta}
|D^\beta\eta(x)|\le k_1\delta^{-|\beta|}, \qquad x\in\mathbb R^n,
\end{equation}
where $k_1=k_1(\beta)>0$ does not depend on $\delta$.

The number $\delta$ will be specified in Secs.~\ref{sec4} and~\ref{sec5}, where we prove
a priori estimates and construct a right regularizer for the nonlocal problem in a
bounded domain.

Introduce the linear bounded operator $\mathcal L_g'': \mathcal
H_a^{l+2m}(\Theta)\to\mathcal H_a^l(\Theta,\Gamma)$ by the formula
$$
\mathcal L_g'' v=\mathcal L_g v+\eta(\mathcal L_g'-\mathcal L_g)v.
$$

\begin{corollary}\label{cort2.2'}
Let Conditions~$\ref{cond1.1}$--$\ref{cond1.4}$ hold. Assume that the line
$\Im\lambda=a+1-l-2m$ contains no eigenvalues of $\hat{\mathcal L}_g(\lambda)$ and
$\dim\mathcal N(\mathcal L_g(\omega))=\codim\mathcal R(\mathcal L_g(\omega))=0$ for any
$\omega\in S^{n-3}$. Then the operator $\mathcal L_g'': \mathcal
H_a^{l+2m}(\Theta)\to\mathcal H_a^l(\Theta,\Gamma)$ is an isomorphism for all
sufficiently small $\delta>0$ and $\|(\mathcal L_g'')^{-1}\|\le c_0$, where $c_0>0$ does
not depend on $\delta$.
\end{corollary}
\begin{proof}
Let us show that
\begin{equation}\label{2.L'-Lto0}
\|\eta(\mathcal L_g'-\mathcal L_g)\|\to0\qquad\text{as}\qquad\delta\to0.
\end{equation}
To do so, we first prove that
\begin{equation}\label{2.B'-Bto0}
\big\|\eta_1\big(B_{j\rho\mu ks}^0(x,D_{y},D_{z})u-B_{j\rho\mu
ks}(D_{y},D_{z})u\big)\big\|_{H_a^{l+2m-m_{j\rho\mu}}(\Theta_k)}\le
k_2\delta\|u\|_{H_a^{l+2m}(\Theta_k)}
\end{equation}
for all $u\in H_a^{l+2m}(\Theta_k)$, where $\eta_1(x)=\eta(\mathcal G_{j\rho
ks}^{-1}y,z)$, while $k_2>0$ does not depend on $u$ and $\delta$.

Let
$$
b\eta_1 D^\beta u,\qquad |\beta|= m_{j\rho\mu},
$$
be an arbitrary term of the expression
$$
\eta_1\big(B_{j\rho\mu ks}^0(x,D_{y},D_{z})u-B_{j\rho\mu ks}(D_{y},D_{z})u\big),
$$
where $b\in C^\infty(\mathbb R^n)$ and $b(0)=0$. It follows from~\eqref{2.D_beta_eta} and
from the relation $b(0)=0$ that
\begin{equation}\label{4.10}
\big|r^{|\alpha|} D^\alpha(b\eta_1)\big|\le k_3\delta,\qquad x\in\mathbb R^n,\
|\alpha|\le l+2m-m_{j\rho\mu},
\end{equation}
where $k_3=k_3(\alpha)>0$ does not depend on $\delta$. Using~\eqref{4.10} and the
definition of the weighted spaces, we directly derive~\eqref{2.B'-Bto0}.

Analogous relations for the pairs of the operators $A_j^0(x,D_{y},D_{z})$ and
$A_j(D_{y},D_{z})$ can be proved in the same way. Thus, we have proved~\eqref{2.L'-Lto0}.

It follows from the conditions of this corollary and from Theorem~\ref{t2.1} that the
operator $\mathcal L_g(\omega)$ is an isomorphism for any $\omega\in S^{n-3}$. Therefore,
by Theorem~\ref{t2.2}, the operator $\mathcal L_g$ is an isomorphism. Combining this fact
with relation~\eqref{2.L'-Lto0}, we complete the proof.
\end{proof}

\subsection{}
In this subsection, we give an example of an operator which corresponds to a nonlocal
elliptic problem in a dihedral angle and is an isomorphism.
\begin{example}\label{ex2.1}
Let
$$
\Theta=\{x=(y,z)\in\mathbb R^3:\ 0<\varphi<d,\ 0<r,\ z\in\mathbb R\}
$$
be a three-dimensional dihedral angle, where $r,\varphi$ are the polar coordinates of the
point $y$. Let
$$
\begin{aligned}
\Gamma_1&=\{x=(y,z)\in\mathbb R^3:\ \varphi=0,\ 0<r,\ z\in\mathbb R\},\\
\Gamma_2&=\{x=(y,z)\in\mathbb R^3:\ \varphi=d,\ 0<r,\ z\in\mathbb R\}
\end{aligned}
$$
be the sides of the angle $\Theta$. Consider the nonlocal elliptic problem
\begin{equation}\label{2.5}
-\Delta v(x)=f_0(x),\qquad x\in\Theta,
\end{equation}
\begin{equation}\label{2.6}
\begin{aligned}
v(\varphi,r,z)|_{\Gamma_1}-\alpha_1
v(\varphi+d/2,r,z)|_{\Gamma_1}&=f_1(x),\qquad x\in\Gamma_1,\\
v(\varphi,r,z)|_{\Gamma_2}-\alpha_2 v(\varphi-d/2,r,z)|_{\Gamma_2}&=f_2(x),\qquad
x\in\Gamma_2,
\end{aligned}
\end{equation}
where $\alpha_1,\alpha_2\in\mathbb R$. Thus, the values of the unknown function $v$ on
the sides $\Gamma_1$ and $\Gamma_2$ are connected with the values of $v$ on the
half-plane $\{x=(y,z)\in\mathbb R^3:\ \varphi=d/2,\ 0<r,\ z\in\mathbb R\}$ lying strictly
inside the angle $\Theta$. The nonlocal transformations are rotations in $y$-plane only,
while transformations with respect to the variables $r$ and $z$ are absent.

Introduce the linear bounded operator
$$
\mathcal L: H_a^{2}(\Theta)\to\mathcal H_a^0(\Theta,\Gamma)=H_a^0(\Theta)\times
H_a^{3/2}(\Gamma_1)\times H_a^{3/2}(\Gamma_2)
$$
by the formula
$$
\mathcal L v=\big(-\Delta v,\ v(\varphi,r,z)|_{\Gamma_1}-\alpha_1
v(\varphi+d/2,r,z)|_{\Gamma_1},\ v(\varphi,r,z)|_{\Gamma_2}-\alpha_2
v(\varphi-d/2,r,z)|_{\Gamma_2}\big),
$$
cf.~\eqref{1.11}. Along with the operator $\mathcal L$, we consider the linear bounded
operator
$$
\mathcal L(\omega): E_a^{2}(\theta)\to\mathcal E_a^0(\theta,\gamma)=E_a^0(\theta)\times
E_a^{3/2}(\gamma_1)\times E_a^{3/2}(\gamma_2)
$$
given by
$$
\mathcal L(\omega) V=\big(-\Delta_y V+V,\ V(\varphi,r)|_{\gamma_1}-\alpha_1
V(\varphi+d/2,r)|_{\gamma_1},\ V(\varphi,r)|_{\gamma_2}-\alpha_2
V(\varphi-d/2,r)|_{\gamma_2}\big),
$$
where
$$
\theta=\{y\in\mathbb R^2:\ 0<\varphi<d,\ 0<r\},
$$
$$
\gamma_1=\{y\in\mathbb R^2:\ \varphi=0,\ 0<r\},\qquad \gamma_2=\{y\in\mathbb R^2:\
\varphi=d,\ 0<r\},
$$
$\omega=\pm1$, cf.~\eqref{2.1}. (Actually, one must write $-\Delta_y V+\omega^2 V$
instead of $-\Delta_y V+V$ in the definition of the operator $\mathcal L(\omega)$, but
$\omega^2=1$ for $\omega=\pm1$. Thus, the operator $\mathcal L(\omega)$ does not depend
on $\omega$ in this example.)

The operator-valued function $\hat{\mathcal L}(\lambda):W^2(0,d)\to\mathcal
W^0[0,d]=L_2(0,d)\times\mathbb C^2$ corresponding to the operator $ \mathcal L(\omega)$
has the form
$$
\hat{\mathcal L}(\lambda) u=\big(-u_{\varphi\varphi}+\lambda^2u,\ u|_{\varphi=0}-\alpha_1
u|_{\varphi=d/2},\ u|_{\varphi=d}-\alpha_2 u|_{\varphi=d/2}\big),
$$
cf.~\eqref{2.4}.

We prove that the operator $\mathcal L(\omega)$ is an isomorphism for $0\le a\le2$,
$0<|\alpha_1+\alpha_2|<2$, and $0<d<2\arctan\sqrt{4(\alpha_1+\alpha_2)^{-2}-1}$. In this
case, Theorem~\ref{t2.2} implies that the operator $\mathcal L$ is also an isomorphism.

The proof comprises three parts.
\begin{enumerate}
\item We prove that the equation
\begin{equation}\label{2.7}
\mathcal A_\alpha w=f_0
\end{equation}
has a unique solution for any $f_0\in L_2(\theta)$, where $\mathcal A_\alpha:
\Dom(\mathcal A_\alpha)\subset L_2(\theta)\to L_2(\theta)$ is the linear bounded operator
given by
$$
\mathcal A_\alpha w=-\Delta w+w,\qquad w\in \Dom(\mathcal A_\alpha)=\{w\in
W_\alpha^1(\theta):\ -\Delta w+w\in L_2(\theta)\},
$$
$$
\begin{aligned}
W_\alpha^1(\theta)&=\{w\in W^1(\theta):\
w(\varphi,r)|_{\gamma_1}-\alpha_1 w(\varphi+d/2,r)|_{\gamma_1}=0,\\
&\qquad w(\varphi,r)|_{\gamma_2}-\alpha_2 w(\varphi-d/2,r)|_{\gamma_2}=0\}.
\end{aligned}
$$
To prove the unique solvability of Eq.~\eqref{2.7}, we reduce it to an elliptic
functional differential equation.
\item
We show that each solution of Eq.~\eqref{2.7} belongs to $H_1^2(\theta\cap B_R)$ for all
$R>0$.
\item
We prove that the equation
\begin{equation}\label{2.8}
\mathcal L(\omega)V=f
\end{equation}
has a unique solution for any $f=(f_0,f_1,f_2)\in\mathcal E_a^0(\theta,\gamma)$.
\end{enumerate}

1. Let us prove that Eq.~\eqref{2.7} has a unique solution $w\in \Dom(\mathcal A_\alpha)$
for any $f_0\in L_2(\theta)$. To do this, we reduce Eq.~\eqref{2.7} to a functional
differential equation.

1a. Consider the functional operator $\mathcal R:L_2(\mathbb R^2)\to L_2(\mathbb R^2)$
given by
$$
\mathcal Ru=u(\varphi,r)+\alpha_1 u(\varphi+d/2,r)+\alpha_2 u(\varphi-d/2,r).
$$
Let $I_\theta:L_2(\theta)\to L_2(\mathbb R^2)$ denote the operator which extends a
function defined on $\theta$ to $\mathbb R^2$ by zero and $P_\theta:L_2(\mathbb R^2)\to
L_2(\theta)$ the operator which restricts a function defined on $\mathbb R^2$ to
$\theta$. Set
$$
\mathcal R_\theta=P_\theta\mathcal R I_\theta.
$$
We claim that the operator $\mathcal R_\theta$ has the bounded inverse
$$
\mathcal R_\theta^{-1}=P_\theta\mathcal R' I_\theta,
$$
where
$$
\mathcal R'u=\big(u(\varphi,r)-\alpha_1 u(\varphi+d/2,r)-\alpha_2
u(\varphi-d/2,r)\big)/(1-\alpha_1\alpha_2),
$$
provided that $\alpha_1\alpha_2\ne1$ (which is true because $|\alpha_1+\alpha_2|<2$).
Indeed,
\begin{equation}\label{2.Rtheta}
\begin{aligned}
\mathcal R_\theta u=u(\varphi,r)+\alpha_1 u(\varphi+d/2,r)\quad\text{for}\
y\in\theta_1=\{y\in\mathbb
R^2:\ 0<\varphi<d/2,\ 0<r\},\\
 \mathcal R_\theta u=u(\varphi,r)+\alpha_2
u(\varphi-d/2,r)\quad\text{for}\ y\in\theta_2=\{y\in\mathbb R^2:\ d/2<\varphi<d,\ 0<r\}.
\end{aligned}
\end{equation}
Therefore,
\begin{multline*}
\mathcal R_\theta^{-1}\mathcal R_\theta u=\big(u(\varphi,r)+\alpha_1
u(\varphi+d/2,r)-\alpha_1
u(\varphi+d/2,r)-\alpha_1\alpha_2u(\varphi,r)\big)/(1-\alpha_1\alpha_2)\\=
u(\varphi,r)\qquad\text{for}\qquad y\in\theta_1,
\end{multline*}
\begin{multline*}
\mathcal R_\theta^{-1}\mathcal R_\theta u=\big(u(\varphi,r)+\alpha_2
u(\varphi-d/2,r)-\alpha_2
u(\varphi-d/2,r)-\alpha_1\alpha_2u(\varphi,r)\big)/(1-\alpha_1\alpha_2)\\=
u(\varphi,r)\qquad\text{for}\qquad y\in\theta_2,
\end{multline*}
which implies $\mathcal R_\theta^{-1}\mathcal R_\theta u(y)=u(y)$ for $y\in\theta$.
Similarly, we obtain $\mathcal R_\theta\mathcal R_\theta^{-1} u(y)=u(y)$ for
$y\in\theta$. Moreover, by using the same arguments as in Theorem~8.1 in~\cite[Chap.~2,
Sec.~8]{SkBook}, one can verify that the operators
$$
\mathcal R_\theta:\mathaccent23W^1(\theta)\to W_\alpha^1(\theta),\qquad \mathcal
R_\theta:\mathaccent23W^1(\theta\cap B_R)\to W_\alpha^1(\theta\cap B_R)
$$
are isomorphisms for all $R>0$, where
$$
\mathaccent23W^1(\theta)=\{u\in W^1(\theta):\ u|_{\gamma_1}=0,\ u|_{\gamma_2}=0\},
$$
$$
 \mathaccent23W^1(\theta\cap B_R)=\{u\in
W^1(\theta\cap B_R):\ u|_{\gamma_1}=0,\ u|_{\gamma_2}=0\},
$$
$$
\begin{aligned}
& W_\alpha^1(\theta\cap B_R)=\{w\in W^1(\theta\cap B_R):\\
&\qquad w(\varphi,r)|_{\gamma_1}-\alpha_1 w(\varphi+d/2,r)|_{\gamma_1}=0,\
w(\varphi,r)|_{\gamma_2}-\alpha_2 w(\varphi-d/2,r)|_{\gamma_2}=0\}.
\end{aligned}
$$

1b. It follows from what has been proved in part~1a that Eq.~\eqref{2.7} is equivalent to
the equation
\begin{equation}\label{2.9}
\mathcal A_\mathcal R u=f_0,
\end{equation}
where $\mathcal A_\mathcal R: \Dom(\mathcal A_\mathcal R)\subset L_2(\theta)\to
L_2(\theta)$ is the unbounded operator given by
$$
\mathcal A_\mathcal R u=(-\Delta+I)\mathcal R_\theta u,\qquad u\in \Dom(\mathcal
A_\mathcal R)=\{u\in \mathaccent23W^1(\theta):\ (-\Delta+I)\mathcal R_\theta u\in
L_2(\theta)\},
$$
and $I$ stands for the identity operator in $L_2(\theta)$.

Similarly to Theorem~10.1 in~\cite[Chap.~2, Sec.~10]{SkBook}, one can show that
Eq.~\eqref{2.9} has a unique solution for any $f_0\in L_2(\theta)$. However, for the
reader's convenience, we prefer to give the proof here.

Consider the sesquilinear form $b_{\mathcal R}[u,v]$ with the domain $\Dom(b_{\mathcal
R})=\mathaccent23W^1(\theta)$ given by
\begin{equation}\label{eqSesquilForm}
b_{\mathcal R}[u,v]=\int\limits_\theta\Big(\sum\limits_{i=1,2}(\mathcal R_\theta
u)_{y_i}\overline{v_{y_i}}+\mathcal R_\theta u\overline{v}\Big)dy.
\end{equation}
It is clear that
\begin{equation}\label{2.10}
\mathcal R_\theta u_{y_i}=(\mathcal R_\theta u)_{y_i}\qquad\text{for}\qquad
u\in\mathaccent23W^1(\theta).
\end{equation}
It follows from the Schwarz inequality and from~\eqref{2.10} that
\begin{equation}\label{2.11}
|b_{\mathcal R}[u,v]|\le
k_1\|u\|_{\mathaccent23W^1(\theta)}\|v\|_{\mathaccent23W^1(\theta)},
\end{equation}
where $k_1>0$ does not depend on $u$ and $v$.

Introduce the isomorphism $\mathcal U:L_2(\theta)\to L_2(\theta_1)\times L_2(\theta_1)$
by the formula
$$
(\mathcal U u)_i(y)=u(\varphi+(i-1)d/2,r),\qquad y\in\theta_1,\ i=1,2.
$$

Let $R_1=\begin{pmatrix} 1&\alpha_1\\ \alpha_2&1
\end{pmatrix}$. One can directly verify that
\begin{equation}\label{2.11'}
R_\theta u=\mathcal U^{-1}R_1\mathcal Uu=\mathcal U^* R_1\mathcal Uu.
\end{equation}
The symmetric part of the matrix $R_1$ has the form
$$
(R_1+R_1^*)/2=\begin{pmatrix} 1&(\alpha_1+\alpha_2)/2\\
(\alpha_1+\alpha_2)/2 & 1\end{pmatrix}.
$$
Since $|\alpha_1+\alpha_2|<2$, it follows that the matrix $(R_1+R_1^*)/2$ is positively
definite. Therefore, using~\eqref{2.10} and~\eqref{2.11'}, we obtain
\begin{multline}\label{2.12}
{\rm Re\,}b_{\mathcal R}[u,u]=
\int\limits_{\theta_1}\Big\{\sum\limits_{i}\Big(\dfrac{(R_1+R_1^*)}{2}(\mathcal U
u_{y_i}),\mathcal U u_{y_i}\Big)_{\mathbb C^2} +\Big(\dfrac{(R_1+R_1^*)}{2}\mathcal U
u,\mathcal U
u\Big)_{\mathbb C^2}\Big\}dy\\
\ge k_2\int\limits_{\theta_1}\Big\{\sum\limits_{i}(\mathcal U u_{y_i},\mathcal U
u_{y_i})_{\mathbb C^2}+(\mathcal U u,\mathcal U u)_{\mathbb C^2}\Big\}dy=
k_2\|u\|^2_{\mathaccent23W^1(\theta)},
\end{multline}
where $k_2>0$ does not depend on $u$.

Inequalities~\eqref{2.11} and~\eqref{2.12} imply that $b_{\mathcal R}$ is a closed
sectorial form on $L_2(\theta)$, with the domain $\Dom(b_{\mathcal
R})=\mathaccent23W^1(\theta)$ and vertex $k_2>0$ (see~\cite[Chap.~6]{Kato}). It follows
from the first representation theorem (see~\cite[Chap.~6, Sec.~2]{Kato}) that the
$m$-sectorial operator $\mathcal A_{\mathcal R}$ associated with the form $b_{\mathcal
R}$ has a bounded inverse $\mathcal A_{\mathcal
R}^{-1}:L_2(\theta)\to\mathaccent23W^1(\theta)$.

Thus, we have proved that Eq.~\eqref{2.7} has a unique solution $w=\mathcal
R_\theta\mathcal A_{\mathcal R}^{-1}f_0\in \Dom(\mathcal A_\alpha)$ for any $f_0\in
L_2(\theta)$.

\smallskip

2. We now prove that, if $w\in \Dom(\mathcal A_\alpha)$ is a solution of Eq.~\eqref{2.7},
then $w\in H_1^2(\theta\cap B_R(0))$ for any $R>0$.

2a. Denote $\theta^{sj}=\theta\cap\{2^{s-j}<|y|<2^{s+j}\}$,
$\gamma_\rho^{sj}=\gamma_\rho\cap\{2^{s-j}<|y|<2^{s+j}\}$, where $s=0,\pm1,\pm2,\dots$;
$\rho,j=1,2,3$; $\gamma_3=\{y\in\mathbb R^2:\ \varphi=d/2,\ 0<r\}$.

We prove that $w\in W^2(\theta^{s3})$ for any $s$. By theorem on interior smoothness
(see, e.g., Theorem~3.2 in~\cite[Chap.~2, Sec.~3]{LM}), we have $w|_{\gamma_3^{s3}}\in
W^{3/2}(\gamma_3^{s3})$. Since
$$
w(\varphi,r)|_{\gamma_1}=\alpha_1 w(\varphi+d/2,r)|_{\gamma_1},\qquad
w(\varphi,r)|_{\gamma_2}=\alpha_2 w(\varphi-d/2,r)|_{\gamma_2},
$$
it follows that
$$
w|_{\gamma_1^{s3}}\in W^{3/2}(\gamma_1^{s3}),\qquad w|_{\gamma_2^{s3}}\in
W^{3/2}(\gamma_2^{s3}).
$$
Therefore, using a theorem on smoothness of solutions of local boundary-value problems in
bounded domains (see e.g., Theorem~8.2 in~\cite[Chap.~2, Sec.~8]{LM}), we obtain $w\in
W^2(\theta^{s2})$. Since $\theta^{s3}=\theta^{s-1,2}\cup\theta^{s+1,2}$, it follows that
$w\in W^2(\theta^{s3})$ for any $s$.

2b. Let us prove that
\begin{equation}\label{2.13}
\|w\|_{W^2(\theta^{01})}\le k_3(\|-\Delta
w\|_{L_2(\theta^{03})}+\|w\|_{W^1(\theta^{03})}),
\end{equation}
where $k_3,k_4,{\dots}>0$ do not depend on $w$.

To do this, we denote $\theta_3^{02}=\{y\in\theta^{02}:\ d/4<\varphi<3d/4\}$ and consider
a function $\xi_0\in C_0^\infty(\theta^{03})$ such that $\xi_0(y)=1$ for
$y\in\theta_3^{02}$.

Using the a priori estimate for solutions of local elliptic problems (see, e.g.,
Theorem~8.2 in~\cite[Chap.~2, Sec.~8]{LM} and Theorem~9.1 in~\cite[Chap.~2, Sec.~9]{LM})
and Leibniz' formula, we have
\begin{multline}\label{2.14}
\|w|_{\gamma_3^{02}}\|_{W^{3/2}(\gamma_3^{02})}\le
\|w\|_{W^{2}(\theta_3^{02})}\le\|\xi_0w\|_{W^{2}(\theta^{03})}\\
\le k_4\|-\Delta(\xi_0w)\|_{L_2(\theta^{03})}\le k_5(\|-\Delta
w\|_{L_2(\theta^{03})}+\|w\|_{W^1(\theta^{03})}).
\end{multline}
Introduce a function $\xi_1(r)\in C_0^\infty(0,+\infty)$ such that $\xi_1(r)=1$ for
$2^{-1}\le r\le 2$ and $\supp\xi_1\subset(2^{-2},2^2)$. Applying Theorem~8.2
in~\cite[Chap.~2, Sec.~8]{LM} and Theorem~9.1 in~\cite[Chap.~2, Sec.~9]{LM} again and
using~\eqref{2.14}, we obtain
\begin{multline*}
\|w\|_{W^{2}(\theta^{01})}\le\|\xi_1w\|_{W^{2}(\theta^{02})}\le k_6\Big(\|-\Delta(\xi_1
w)\|_{L_2(\theta^{02})}+\sum\limits_{\rho=1,2}\|(\xi_1
w)|_{\gamma_\rho^{02}}\|_{W^{3/2}(\gamma_\rho^{02})}\Big)\\ \le k_7(\|-\Delta
w\|_{L_2(\theta^{03})}+\|w\|_{W^1(\theta^{03})}).
\end{multline*}
Thus, inequality~\eqref{2.13} is proved.

2c. Now we prove that $w\in H_1^2(\theta\cap B_R)$ for any $R>0$. It follows from part 2a
of the proof that $w\in H_1^2(\theta^{sj})$ (we set $\rho(y)=|y|$ in the definition of
the space $H_1^2(\theta^{sj})$). Set $y'=2^{-s}y$. Clearly, $y'\in\theta^{0j}$ for
$y\in\theta^{sj}$. Therefore, using the fact that $2^{s-1}<r<2^{s+1}$ for
$y\in\theta^{s1}$, letting $w^s(y')=w(2^sy')$ and applying inequality~\eqref{2.13}, we
obtain
\begin{multline}\label{2.16}
\|w\|_{H_1^2(\theta^{s1})}\le
k_8\sum\limits_{|\alpha|\le2}2^{2s(1-2+|\alpha|)}\int\limits_{\theta^{s1}}|D^\alpha_y
w(y)|^2dy=k_8\sum\limits_{|\alpha|\le2}\,\int\limits_{\theta^{01}}|D^\alpha_{y'}
w^s(y')|^2dy'\\
\le k_9\Big(\|-\Delta_{y'}w^s(y')\|^2_{L_2(\theta^{03})}+
\sum\limits_{|\alpha|\le1}\|D_{y'}^\alpha w^s(y')\|^2_{L_2(\theta^{03})}
\Big)\\
=k_{9}\Big(2^{2s}\|-\Delta_y
w(y)\|_{L_2(\theta^{s3})}^2+\sum\limits_{|\alpha|\le1}2^{2s(0-1+|\alpha|)}\|D_y^\alpha
w(y)\|^2_{L_2(\theta^{s3})}\Big),
\end{multline}
where $k_7,k_8,\dots{>0}$ do not depend on $w$ and $s$. It follows from~\eqref{2.16} that
\begin{equation}\label{2.17}
\|w\|_{H_1^2(\theta^{s1})}\le k_{10}(\|-\Delta
w+w\|^2_{L_2(\theta^{s3})}+\|w\|_{H_0^1(\theta^{s3})})
\end{equation}
for $s\le[\log_2 R]$.

Now we claim that
\begin{equation}\label{2.17'}
w\in H_0^1(\theta\cap B_{8R}).
\end{equation}
Indeed,
\begin{equation}\label{2.17''}
w\in W^1_\alpha(\theta\cap B_{8R})
\end{equation}
by assumption. Therefore, $\mathcal R_\theta^{-1}w\in \mathaccent23W^1(\theta\cap
B_{8R})$ because the operator $ \mathcal R_\theta:\mathaccent23W^1(\theta\cap B_{8R})\to
W_\alpha^1(\theta\cap B_{8R})$ is an isomorphism. By Lemma~4.8 in~\cite{Kondr},
$\mathaccent23W^1(\theta\cap B_{8R})\subset H_0^1(\theta\cap B_{8R})$, which implies that
$\mathcal R_\theta^{-1}w\in H_0^1(\theta\cap B_{8R})$. Therefore, using~\eqref{2.Rtheta},
we have
$$
w\in H_0^1(\theta_1\cap B_{8R}),\qquad w\in H_0^1(\theta_2\cap B_{8R}).
$$
Combining these relations with~\eqref{2.17''} yields~\eqref{2.17'}.

Summing inequalities~\eqref{2.17} with respect to $s\le[\log_2 R]$ and taking into
account relation~\eqref{2.17'}, we obtain
$$
\|w\|_{H_1^2(\theta\cap B_R)}\le k_{11}(\|-\Delta w+w\|^2_{L_2(\theta\cap
B_{8R})}+\|w\|_{H_0^1(\theta\cap B_{8R})}).
$$
Thus, we have proved that $w\in H_1^2(\theta\cap B_R)$.

\smallskip

3. We finally prove that Eq.~\eqref{2.8} has a unique solution $V\in E_a^2(\theta)$ for
any $f\in\mathcal E_a^0(\theta,\gamma)$, where $0\le a\le 2$.

3a. Let $w\in \Dom(\mathcal A_\alpha)$ be a solution of Eq.~\eqref{2.7} with right-hand
side $f_0\in C_0^\infty(\overline{\theta}\setminus\{0\})$. It is easy to check that the
strip $-1\le\Im\lambda\le1$ contains no eigenvalues of the operator-valued function
$\hat{\mathcal L}(\lambda)$ for $0<d<2\arctan\sqrt{4(\alpha_1+\alpha_2)^{-2}-1}$
(see~\cite[Sec.~9.1]{GurRJMP04}). On the other hand, $w\in H_1^2(\theta\cap B_1)$ by the
above, and the inequalities $-1\le a+1-2\le1$ hold. Therefore, by Lemma~3.2
in~\cite{SkMs86} concerning the asymptotic behavior of solutions of nonlocal elliptic
problems in plane angles, we have $w\in H_a^2(\theta\cap B_1)$.

3b. Now let us prove that the equation
\begin{equation}\label{2.18}
\mathcal L(\omega)w=(F_0,0,0)
\end{equation}
has a solution $w\in E_a^2(\theta)$ for any $F_0\in E_a^0(\theta)$.

Repeating the arguments from the proof of inequality~(2.4) in~\cite[Chap.~6, Sec.~2]{NP},
one can see that a solution $w\in \Dom(\mathcal A_\alpha)$ of Eq.~\eqref{2.7} with
right-hand side $f_0\in C_0^\infty(\overline{\theta}\setminus\{0\})$ belongs to
$E_a^2(\theta\setminus B_{1/2})$. Combining this fact with part~3a of our proof yields
$w\in E_a^2(\theta)$. Since the line $\Im\lambda=a+1-2$ contains no eigenvalues of the
operator-valued function $\hat{\mathcal L}(\lambda)$, it follows from Theorem~\ref{2.1}
that the set of functions $F_0\in E_a^0(\theta)$ for which Eq.~\eqref{2.18} has a
solution is closed in $E_a^0(\theta)$. On the other hand, the set
$C_0^\infty(\overline{\theta}\setminus\{0\})$ is dense in $E_a^0(\theta)$. Therefore,
Eq.~\eqref{2.18} has a solution $w\in E_a^2(\theta)$ for any $F_0\in E_a^0(\theta)$.

3c. Let us show that $\mathcal R(\mathcal L(\omega))=\mathcal E_a^0(\theta,\gamma)$. Take
functions $U_\rho\in E_a^2(\theta)$ such that $U_\rho|_{\gamma_\rho}=f_\rho$, $\rho=1,2$.
Consider cut-off functions $\eta_\rho(\varphi)\in C^\infty[0,d]$ such that
$\eta_1(\varphi)=1$ for $0\le\varphi\le d/4$, $\eta_1(\varphi)=0$ for $d/3\le\varphi\le
d$ and $\eta_2(\varphi)=1$ for $3d/4\le\varphi\le d$, $\eta_2(\varphi)=0$ for
$0\le\varphi\le 2d/3$. Then Eq.~\eqref{2.8} is equivalent to Eq.~\eqref{2.18}, where
$F_0=\Delta U-U+f_0$, $U=\eta_1U_1+\eta_2U_2\in E_a^2(\theta)$, and $w=V-U$. It is proved
in part 3b that Eq.~\eqref{2.18} has a solution $w\in E_a^2(\theta)$ for any $F_0\in
E_a^0(\theta)$. Therefore, Eq.~\eqref{2.8} has a solution $V=w+U\in E_a^2(\theta)$ for
any $f\in\mathcal E_a^0(\theta,\gamma)$.

3d. It remains to prove that $\mathcal N(\mathcal L(\omega))=\{0\}$. Let $w\in
E_a^2(\theta)$ be a solution of Eq.~\eqref{2.18} with $F_0=0$. Using the same arguments
as in part~3a of this proof, we have $w\in H_1^2(\theta\cap B_1)\subset W^1(\theta\cap
B_1)$. On the other hand, $E_a^2(\theta\setminus \overline{B_{1/2}})\subset
W^1(\theta\setminus \overline{B_{1/2}})$ because $a\ge0$. Therefore, $w\in W^1(\theta)$,
and $w=0$ by part~1 of this proof.
\end{example}

Note that Example~\ref{ex2.1} was earlier studied by another method
in~\cite[Sec.~10]{GurGiess}. The approach proposed in~\cite{GurGiess} is based on the
Green formulas for nonlocal elliptic problems (see~\cite{GurGiess}) and on the
interrelation between nonlocal elliptic problems and boundary-value problems for
functional differential equations (see~\cite{SkBook}). Other examples of nonlocal
elliptic problems generalizing problem~\eqref{2.5}, \eqref{2.6} and being uniquely
solvable in dihedral angles are constructed in~\cite{Skryabin}.

\subsection{}\label{subsec2.5}
Given a point $g\in\mathcal K_2$, we consider the linear bounded operator $$\mathcal
L_g(\omega): E_a^{l+2m}(\mathbb R_+^2)\to\mathcal E_a^l(\mathbb R_+^2,\gamma)$$ given by
\begin{equation}\label{2.19}
\mathcal L_g(\omega) V=\big(A(D_{y},\omega)V(y),\ B_{i\mu0}(D_y,\omega)V(y)|_{\mathbb
R_-},\ B_{i\mu0}(D_y,\omega)V(y)|_{\mathbb R_+}\big),
\end{equation}
where
$$
\mathcal E_a^l(\mathbb R_+^2,\gamma)=E_a^l(\mathbb R_+^2)\times
E_a^{l+2m-m_{i\mu}-1/2}(\mathbb R_-)\times E_a^{l+2m-m_{i\mu}-1/2}(\mathbb R_+),
$$
$\mathbb R_+^2=\{y\in\mathbb R^2:\ |\varphi|<\pi/2\}$,  $\mathbb R_\pm=\{y\in\mathbb
R^2:\ \varphi=\pm\pi/2\}$, $\omega\in S^{n-3}$, cf.~\eqref{1.12}.

We write the operators $A(D_{y},0)$ and $B_{i\mu0}(D_{y},0)$ in the polar coordinates:
\begin{equation*}
A(D_{y},0)=r^{-2m}\hat A(\varphi,D_\varphi,rD_r),\qquad
B_{i\mu0}(D_{y},0)=r^{-m_{i\mu}}\hat B_{i\mu0}(\varphi,D_\varphi,rD_r).
\end{equation*}

Consider the analytic operator-valued function $$\hat{\mathcal L}_g(\lambda):
W^{l+2m}(-\pi/2,\pi/2)\to\mathcal W^l[-\pi/2,\pi/2]$$ given by
\begin{equation}\label{2.22}
\hat{\mathcal L}_g(\lambda) w=\big(\hat A(\varphi,D_\varphi,\lambda)w,\ \hat
B_{i\mu0}(\varphi,D_\varphi,\lambda)w|_{\varphi=-\pi/2},\ \hat
B_{i\mu0}(\varphi,D_\varphi,\lambda)w|_{\varphi=\pi/2}\big),
\end{equation}
where $\mathcal W^l[-\pi/2,\pi/2]=W^l(-\pi/2,\pi/2)\times\mathbb C\times\mathbb C$,
cf.~\eqref{2.4}.

It follows from~\cite{AV,GS} that there exists a finite-meromorphic operator-valued
function $\hat{\mathcal R}_g(\lambda): \mathcal W^l[-\pi/2,\pi/2]\to \mathcal
W^{l+2m}(-\pi/2,\pi/2)$ such that $\hat{\mathcal L}_g^{-1}(\lambda)=\hat{\mathcal
R}_g(\lambda)$ for any $\lambda$ which is not a pole of $\hat{\mathcal R}_g(\lambda)$.
Moreover, if $\lambda_0=\mu_0+i\nu_0$ is a pole of $\hat{\mathcal R}_g(\lambda)$, then
$\lambda_0$ is an eigenvalue of $\hat{\mathcal L}_g(\lambda)$, and there exists a
$\delta>0$ such that the set $\{\lambda\in\mathbb C:\ 0<|\Im\lambda-\nu_0|<\delta\}$
contains no eigenvalues of $\hat{\mathcal L}_g(\lambda)$.

The following theorem establishes a connection between the operators $\mathcal L_g$ and
$\mathcal L_g(\omega)$ (see Theorem~2.1 in~\cite[Chap.~6, Sec.~2]{NP}).
\begin{theorem}\label{t2.R^n_+Isomorphism}
Let Conditions~$\ref{cond1.1}$ and~$\ref{cond1.2}$ hold. Assume that the line
$\Im\lambda=a+1-l-2m$ contains no eigenvalues of $\hat{\mathcal L}_g(\lambda)$ and
$\dim\mathcal N(\mathcal L_g(\omega))=\codim\mathcal R(\mathcal L_g(\omega))=0$ for any
$\omega\in S^{n-3}$. Then the operator $\mathcal L_g: \mathcal H_a^{l+2m}(\mathbb
R^n_+)\to\mathcal H_a^l(\mathbb R^n_+,\Gamma)$ is an isomorphism.
\end{theorem}

\section{Local Elliptic Problems in $\mathbb R^n\setminus\mathcal P$}\label{sec3}

\subsection{}\label{subsec3.1}
In Secs.~\ref{subsec3.1} and~\ref{subsec3.2}, we recall some known results on the
solvability of elliptic problems in $\mathbb R^2\setminus\{0\}$. These results are
adopted from~\cite{SkMs86,GurPetr}; they will be applied to the investigation of local
elliptic problems in $\mathbb R^n\setminus\mathcal P$, $n\ge3$, see
Secs.~\ref{subsec3.3}--\ref{subsec3.5}.

Let $A$ be a properly elliptic homogeneous operator with constant complex coefficients,
given by
$$
A=A(D_y)=\sum\limits_{|\alpha|=2m}a_\alpha D_y^\alpha.
$$
The operator $A: H_a^{l+2m}(\mathbb R^2)\to H_a^l(\mathbb R^2)$ is bounded for any fixed
integer $l\ge0$. We consider the equation
\begin{equation}\label{3.1}
Av=f_0(y),\qquad y\in\mathbb R^2\setminus\{0\},
\end{equation}
where $f_0\in H_a^l(\mathbb R^2)$.

Write the operator $A(D_y)$ in the polar coordinates,
$$
A(D_y)=r^{-2m}\hat A(\varphi,D_\varphi,rD_r)=r^{-2m}\sum\limits_{\alpha_1+\alpha_2\le 2m}
a_{\alpha_1\alpha_2}(\varphi)D_\varphi^{\alpha_1}(rD_r)^{\alpha_2},
$$
where $a_{\alpha_1\alpha_2}\in C_{2\pi}^\infty[0,2\pi]$, $C_{2\pi}^\infty[0,2\pi]$ is the
set of functions defined on the interval $[0,2\pi]$ such that their $2\pi$-periodic
extensions are infinitely differentiable on $\mathbb R$.

Setting $\tau=\ln r$, we infer from~\eqref{3.1}
\begin{equation}\label{3.2}
\hat A(\varphi,D_\varphi,D_\tau)v=F_0(\varphi,\tau),\qquad 0<\varphi<2\pi,\
-\infty<\tau<\infty,
\end{equation}
\begin{equation}\label{3.3}
D_\varphi^j v|_{\varphi=0}=D_\varphi^j v|_{\varphi=2\pi},\qquad -\infty<\tau<\infty,\
j=0,\dots,l+2m-1,
\end{equation}
where $D_\tau=-i\partial/\partial\tau$, $F_0(\varphi,\tau)=e^{2m\tau}f_0(\varphi,\tau)$,
$D_\varphi^j F_0|_{\varphi=0}=D_\varphi^j F_0|_{\varphi=2\pi}$, $j=1,\dots,l-1$.

By using the Fourier transform with respect to $\tau$, we obtain from
relations~\eqref{3.2} and \eqref{3.3}
\begin{equation}\label{3.4}
\hat A(\varphi,D_\varphi,\lambda)\hat v(\varphi,\lambda)=\hat F_0(\varphi,\lambda),\qquad
0<\varphi<2\pi,
\end{equation}
\begin{equation}\label{3.5}
D_\varphi^j \hat v|_{\varphi=0}=D_\varphi^j \hat v|_{\varphi=2\pi},\qquad
j=0,\dots,l+2m-1.
\end{equation}

Denote by $W_{2\pi}^k(0,2\pi)$ the closure of the set $C_{2\pi}^\infty[0,2\pi]$ in the
space $W^k(0,2\pi)$. Consider the operator-valued function $\hat A(\lambda):
W_{2\pi}^{l+2m}(0,2\pi)\to W_{2\pi}^l(0,2\pi)$ given by
$$
\hat A(\lambda)w=\hat A(\varphi,D_\varphi,\lambda)w(\varphi).
$$
It follows from~\cite[Sec.~1]{SkMs86} that there exists a finite-meromorphic
operator-valued function $\hat{\mathcal R}(\lambda): \mathcal W_{2\pi}^l(0,2\pi)\to
W_{2\pi}^{l+2m}(0,2\pi)$ such that $\hat A^{-1}(\lambda)=\hat{\mathcal R}(\lambda)$ for
any $\lambda$ which is not a pole of $\hat{\mathcal R}(\lambda)$. Moreover, if
$\lambda_0=\mu_0+i\nu_0$ is a pole of $\hat{\mathcal R}(\lambda)$, then $\lambda_0$ is an
eigenvalue of $\hat A(\lambda)$, and there exists a $\delta>0$ such that the set
$\{\lambda\in\mathbb C:\ 0<|\Im\lambda-\nu_0|<\delta\}$ contains no eigenvalues of $\hat
A(\lambda)$.

The following result is proved in~\cite[Sec.~1]{SkMs86}.
\begin{lemma}\label{l3.2}
Assume that the line $\Im\lambda=a+1-l-2m$ contains no eigenvalues of the operator-valued
function $\hat A(\lambda)$. Then Eq.~\eqref{3.1} has a unique solution $v\in
H_a^{l+2m}(\mathbb R^2)$ for any $f_0\in H_a^l(\mathbb R^2)$ and
\begin{equation}\label{3.7}
\|v\|_{H_a^{l+2m}(\mathbb R^2)}\le c\|f_0\|_{H_a^l(\mathbb R^2)},
\end{equation}
where $c>0$ does not depend on $f_0$.
\end{lemma}

\subsection{}\label{subsec3.2}
Now we consider the asymptotic behavior of solutions of elliptic problems in $\mathbb
R^2\setminus\{0\}$. Let $l_1, l_2\ge0$ be integers, and let $a_1,a_2\in\mathbb R$ be such
that
$$
h_2=a_2+1-l_2-2m<a_1+1-l_1-2m=h_1.
$$
By the above properties of the operator-valued function $\hat A(\lambda)$, the strip
$h_2<\Im\lambda<h_1$ contains finitely many eigenvalues $\lambda_j$, $j=1,\dots,J$, of
$\hat A(\lambda)$. Let $q_j$ be the geometrical multiplicity of the eigenvalue
$\lambda_j$. Denote by
\begin{equation}\label{3.JordanChain}
\{\psi_j^{0,q}(\varphi),\dots,\psi_j^{p_{jq}-1,q}(\varphi)\},\qquad q=1,\dots,q_j,
\end{equation}
a canonical system of Jordan chains corresponding to the eigenvalue $\lambda_j$, where
$p_{j1}\ge p_{j2}\ge\dots\ge p_{jq_j}$ are the ranks of the eigenvectors
$\psi_j^{0,1}(\varphi),\dots,\psi_j^{0,q_j}(\varphi)$, respectively,
see~\cite[Sec.~1]{GS}. It is known that the Jordan chain~\eqref{3.JordanChain} satisfies
the equations
\begin{equation}\label{3.8}
\sum\limits_{s=0}^p\frac{1}{s!}\partial_\lambda^s\hat
A(\lambda_j)\psi_j^{p-s,q}(\varphi)=0,\qquad p=0,\dots,p_{jq}-1,
\end{equation}
where $\partial_\lambda^s=\partial^s/\partial \lambda^s$.

\begin{lemma}\label{l3.3}
Let $u\in H_{a_1}^{l_1+2m}(\mathbb R^2)$ be a solution of Eq.~\eqref{3.1}, and let
$f_0\in H_{a_1}^{l_1}(\mathbb R^2)\cap H_{a_2}^{l_2}(\mathbb R^2)$. Suppose that the line
$\Im\lambda=h_2$ contains no eigenvalues of the operator-valued function $\hat
A(\lambda)$. Then
\begin{equation}\label{3.9}
v(y)=\sum\limits_{j=1}^J\sum\limits_{q=1}^{q_j}\sum\limits_{k=0}^{p_{jq}-1}
\alpha_j^{kq}v_j^{kq}(y)+w(y),\qquad y\in\mathbb R^2\setminus\{0\};
\end{equation}
here
\begin{equation}\label{3.10}
v_j^{kq}(y)=r^{i\lambda_j}\sum\limits_{n=0}^k\frac{1}{n!}(i\ln
r)^n\psi_j^{k-n,q}(\varphi),
\end{equation}
$\alpha_j^{kq}=\alpha_j^{kq}(f_0)$ are linear bounded functionals on
$H_{a_1}^{l_1}(\mathbb R^2)\cap H_{a_2}^{l_2}(\mathbb R^2)$, the function $w\in
H_{a_2}^{l_2+2m}(\mathbb R^2)$ satisfies the equation $Aw=f_0$ and the inequality
\begin{equation}\label{3.11}
\|w\|_{H_{a_2}^{l_2+2m}(\mathbb R^2)}\le c\|f_0\|_{H_{a_2}^{l_2}(\mathbb R^2)},
\end{equation}
where $c>0$ does not depend on $f_0$.
\end{lemma}
This lemma was obtained in~\cite[Sec.~3]{SkMs86} in a slightly different form. Its proof
is similar to that of Theorem~1.4 in~\cite[Chap.~3, Sec.~1]{NP}; see
also~\cite[Sec.~5]{GurPetr}, where the coefficients $\alpha_j^{kq}$ are explicitly
calculated.
\begin{remark}
It is easy to see that $\psi_j^{s,q}\in C_{2\pi}^\infty[0,2\pi]$, $j=1,\dots,J$;
$q=1,\dots,q_j$; $s=0,\dots,p_{jq}-1$.
\end{remark}
\begin{remark}\label{r3.2}
Lemma~\ref{l3.3} is also valid for $h_2\ge h_1$.
\end{remark}

Using the same arguments as in~\cite[Chap.~3, Sec.~1]{NP}, one can obtain the following
corollaries from Lemma~\ref{l3.3} (see also~\cite[Sec.~5]{GurPetr}).

\begin{corollary}\label{c3.1}
Let the conditions of Lemma~$\ref{l3.3}$ hold, and let the strip $h_2<\Im\lambda<h_1$
contain no eigenvalues of $\hat A(\lambda)$. Then $v\in H_{a_2}^{l_2+2m}(\mathbb R^2)$.
\end{corollary}

\begin{corollary}\label{c3.2}
Let the line $\Im\lambda=a+1-l-2m$ contain an eigenvalue of $\hat A(\lambda)$. Then the
kernel of the operator $A: H_a^{l+2m}(\mathbb R^2)\to H_a^l(\mathbb R^2)$ is trivial,
while the image of $A$ is not closed in $H_a^l(\mathbb R^2)$.
\end{corollary}

\begin{example}\label{ex3.1}
We consider the equation
\begin{equation}\label{3.12}
-\Delta v(y)=f_0(y),\qquad y\in\mathbb R^2\setminus\{0\}.
\end{equation}
Introduce the operator $A: H_a^2(\mathbb R^2)\to H_a^0(\mathbb R^2)$ by the formula
$Av=-\Delta v$.

Passing to the variables $\tau,\varphi$ and using the Fourier transform with respect to
$\tau$, we have
\begin{equation}\label{3.13}
-\hat v_{\varphi\varphi}(\varphi,\lambda)+\lambda^2\hat v(\varphi,\lambda)=\hat
F_0(\varphi,\lambda),\qquad 0<\varphi<2\pi,
\end{equation}
\begin{equation}\label{3.14}
\hat v|_{\varphi=0}=\hat v|_{\varphi=2\pi},\qquad \hat v_\varphi|_{\varphi=0}=\hat
v_\varphi|_{\varphi=2\pi},
\end{equation}
where $F_0(\varphi,\tau)=e^{2\tau}f_0(\varphi,\tau)$, cf.~\eqref{3.4}, \eqref{3.5}.

Let us study the eigenvalue problem for the corresponding operator-valued function $\hat
A(\lambda):W^2_{2\pi}(0,2\pi)\to L_2(0,2\pi)$ given by
$$
\hat A(\lambda)\hat v=-\hat v_{\varphi\varphi}+\lambda^2 \hat v.
$$
The general solution of the equation
\begin{equation}\label{3.15}
-\hat v_{\varphi\varphi}+\lambda^2 \hat v=0
\end{equation}
for $\lambda\ne0$ has the form
\begin{equation}\label{3.16}
\hat v(\varphi)=c_1e^{\lambda\varphi}+c_2e^{-\lambda\varphi}.
\end{equation}
Substituting this solution into~\eqref{3.14}, we have
$$
\begin{aligned}
c_1(1-e^{\lambda2\pi})+c_2(1-e^{-\lambda2\pi})&=0,\\
c_1\lambda(1-e^{\lambda2\pi})-c_2\lambda(1-e^{-\lambda2\pi})&=0.
\end{aligned}
$$
Therefore, the set of nonzero eigenvalues of $\hat A(\lambda)$ coincides with the set of
nonzero roots of the equation
$$
e^{\lambda2\pi}=1.
$$
The nonzero roots of this equation have the form
\begin{equation}\label{3.17}
\lambda_k=ik,\qquad k=\pm1,\pm2,\dots.
\end{equation}

It is evident that $\lambda_0=0$ is also an eigenvalue of $\hat A(\lambda)$.

1. Let us first consider the eigenvalue $\lambda_0=0$. The corresponding eigenvector has
the form $\psi_0^{0,1}(\varphi)=1$ (up to factor). The geometric multiplicity of
$\lambda_0=0$ is equal to 1.

Due to~\eqref{3.8}, an associated vector $\psi_0^{1,1}(\varphi)$ must satisfy the
equation
$$
\hat A(0)\psi_0^{1,1}+\partial_\lambda\hat A(0)\psi_0^{0,1}=0,
$$
which is equivalent to the following problem:
\begin{gather*}
(\psi_0^{1,1})_{\varphi\varphi}=0,\qquad 0<\varphi<2\pi,\\
\psi_0^{1,1}|_{\varphi=0}=\psi_0^{1,1}|_{\varphi=2\pi},\qquad
(\psi_0^{1,1})_\varphi|_{\varphi=0}=(\psi_0^{1,1})_{\varphi}|_{\varphi=2\pi}.
\end{gather*}
Hence, $\psi_0^{1,1}=0$ (note that an associated vector of an operator-valued function,
unlike an eigenvector, can be equal to zero).

Due to~\eqref{3.8}, the second associated vector $\psi_0^{2,1}(\varphi)$ must satisfy the
equation
$$
\hat A(0)\psi_0^{2,1}+\partial_\lambda\hat
A(0)\psi_0^{1,1}+\frac{1}{2}\partial^2_\lambda\hat A(0)\psi_0^{0,1}=0,
$$
which is equivalent to the following problem:
\begin{gather}
(\psi_0^{2,1})_{\varphi\varphi}=1,\qquad 0<\varphi<2\pi,\label{3.18}\\
\psi_0^{2,1}|_{\varphi=0}=\psi_0^{2,1}|_{\varphi=2\pi},\qquad
(\psi_0^{2,1})_\varphi|_{\varphi=0}=(\psi_0^{2,1})_{\varphi}|_{\varphi=2\pi}.\label{3.19}
\end{gather}
Substituting the general solution $\psi_0^{2,1}(\varphi)=c_0+c_1\varphi+\varphi^2/2$ of
Eq.~\eqref{3.18} into~\eqref{3.19}, we obtain
$$
2\pi c_1+4\pi^2/2=0,\qquad 2\pi=0.
$$
This system is incompatible, and hence there does not exist a second associated vector
for $\lambda_0=0$.

\smallskip

2. Consider the eigenvalue $\lambda_k=ik$, $k=\pm1,\pm2,\dots$. There are two linearly
independent eigenvectors corresponding to $\lambda_k$,
$$
\psi_k^{0,1}=\sin k\varphi,\qquad \psi_k^{0,2}=\cos k\varphi
$$
(up to factor). Hence, the geometric multiplicity of $\lambda_k$ equals~2.

Due to~\eqref{3.8}, an associated vector $\psi_k^{1,j}(\varphi)$ ($j=1,2$) must satisfy
the equation
$$
\hat A(ik)\psi_k^{1,j}+\partial_\lambda\hat A(ik)\psi_k^{0,j}=0,
$$
which is equivalent to the following problem:
\begin{gather}
(\psi_k^{1,j})_{\varphi\varphi}+k^2\psi_k^{1,j}=i2k\psi_k^{0,j},\qquad 0<\varphi<2\pi,
\label{3.20}\\
\psi_k^{1,j}|_{\varphi=0}=\psi_k^{1,j}|_{\varphi=2\pi},\qquad
(\psi_k^{1,j})_\varphi|_{\varphi=0}=(\psi_k^{1,j})_{\varphi}|_{\varphi=2\pi}.
\label{3.21}
\end{gather}
Substituting the general solution
$$
\psi_k^{1,1}(\varphi)=a_k^1\cos k\varphi+b_k^1\sin k\varphi-i\varphi\cos
k\varphi\qquad\text{for}\qquad j=1,
$$
$$
\psi_k^{1,2}(\varphi)=a_k^2\cos k\varphi+b_k^2\sin k\varphi+i\varphi\sin
k\varphi\qquad\text{for}\qquad j=2
$$
of Eq.~\eqref{3.20} into~\eqref{3.21}, we have
$$
0=-i2\pi,\quad 0=0\qquad\text{for}\qquad j=1,
$$
$$
0=0,\quad 0=i2\pi k\qquad\text{for}\qquad j=2.
$$
These systems are incompatible for $k=\pm1,\pm2,\dots$, and hence there do not exist
associated vectors for $\lambda_k=ik$, $k=\pm1,\pm2,\dots$.
\end{example}

\begin{example}\label{ex3.2}
Let $v\in H_1^2(\mathbb R^2)$ be a solution of Eq.~\eqref{3.12} with right-hand side
$f_0\in H_1^0(\mathbb R^2)\cap H_{-\varepsilon}^0(\mathbb R^2)$, where $0<\varepsilon<1$.
We study the asymptotic behavior of the solution $v$. Using the results of
Example~\ref{ex3.1} and Lemma~\ref{l3.3} for $a_1=1,l_1=0,h_1=0$ and
$a_2=-\varepsilon,l_2=0,h_2=-1-\varepsilon$, we obtain
\begin{equation}\label{3.22}
v(y)=(\alpha_{1}^{0,1} \sin\varphi+\alpha_{1}^{0,2}
\cos\varphi)r+w(y)=\alpha_{1}^{0,1}y_2+\alpha_{1}^{0,2}y_1+w(y),\qquad y\in\mathbb
R^2\setminus\{0\},
\end{equation}
where $w\in H_{-\varepsilon}^2(\mathbb R^2)$, while
$\alpha_{1}^{0,j}=\alpha_{1}^{0,j}(f_0)$ are linear bounded functionals on $H_1^0(\mathbb
R^2)\cap H_{-\varepsilon}^0(\mathbb R^2)$. Note that these functionals can be found in an
explicit form (see~\cite[Sec.~5]{GurPetr}).
\end{example}

\subsection{}\label{subsec3.3}
We now proceed with the study of elliptic problems in $\mathbb R^n\setminus\mathcal P$
for $n\ge3$, where
$$
\mathcal P=\{x=(y,z)\in\mathbb R^n:\ y=0,\ z\in\mathbb R^{n-2}\}.
$$
 Let
$$
A=A(D_y,D_z)=\sum\limits_{|\alpha|+|\beta|=2m}a_{\alpha\beta} D_y^\alpha D_z^\beta,
$$
be a homogeneous properly elliptic operator with constant complex coefficients. We
consider the equation
\begin{equation}\label{3.23}
Av=f_0(x),\qquad x\in\mathbb R^n\setminus\mathcal P,
\end{equation}
where $f_0\in H_a^l(\mathbb R^n)$. It is easy to see that the operator
$A:H_a^{l+2m}(\mathbb R^n)\to H_a^{l}(\mathbb R^n)$ is bounded for any $a\in\mathbb R$
and any integer $l\ge0$.

The main result of this section is as follows.
\begin{theorem}\label{t3.1}
Let $a\in\mathbb R$, and let $l\ge0$ be an integer. Then the operator
$A:H_a^{l+2m}(\mathbb R^n)\to H_a^{l}(\mathbb R^n)$ is not an isomorphism.
\end{theorem}

To prove Theorem~\ref{t3.1}, we first apply the Fourier transform $F_{z\to\eta}$ with
respect to $z\in\mathbb R^{n-2}$. Then Eq.~\eqref{3.23} takes the form
\begin{equation}\label{3.24}
A(D_y,\eta)\tilde v(y,\eta)=\tilde f_0(y,\eta),\qquad y\in\mathbb R^2\setminus\{0\},\
\eta\in\mathbb R^{n-2},
\end{equation}
where
$$
\tilde v(y,\eta)=F_{z\to\eta}v=(2\pi)^{-(n-2)/2}\int\limits_{\mathbb R^{n-2}}
v(y,z)e^{-i(\eta,z)}dz.
$$
Denote $Y=|\eta|y$, $\omega=\eta/|\eta|$, $V(Y,\eta)=|\eta|^{2m}\tilde v(y,\eta)$,
$F_0(Y,\eta)=\tilde f_0(y,\eta)$. Then Eq.~\eqref{3.24} takes the following form:
\begin{equation}\label{3.26}
A(D_Y,\omega)V(Y,\eta)=F_0(Y,\eta),\qquad y\in\mathbb R^2\setminus\{0\},\ \omega\in
S^{n-3}.
\end{equation}

Consider the linear bounded operator $A(\omega):E_a^{l+2m}(\mathbb R^2)\to
E_a^{l}(\mathbb R^2)$ given by
\begin{equation}\label{3.27}
A(\omega)V=A(D_Y,\omega)V(Y),\qquad \omega\in S^{n-3}.
\end{equation}

Denote by $A(0):H_a^{l+2m}(\mathbb R^2)\to H_a^{l}(\mathbb R^2)$ the linear bounded
operator given by
$$
A(0)v=A(D_y,0)v(y).
$$
Clearly, the operator $A(0)$ is properly elliptic. Write the operator $A(0)$ in the polar
coordinates,
$$
A(0)=r^{-2m}\hat A(\varphi,D_\varphi,rD_r).
$$
Consider the operator-valued function $\hat A(\lambda):W_{2\pi}^{l+2m}(0,2\pi)\to
W_{2\pi}^{l}(0,2\pi)$ given by
$$
\hat A(\lambda)w=\hat A(\varphi,D_\varphi,\lambda)w(\varphi).
$$
Spectral properties of the operator-valued function $ \hat A(\lambda)$ are described in
Sec.~\ref{subsec3.1}.

The proof of the following lemma is similar to that of Theorem~2.3 in~\cite[Chap.~6,
Sec.~2]{NP} (see also Theorem~4.2 in~\cite{SkDu90}).
\begin{lemma}\label{l3.8}
The operator $A(\omega):E_a^{l+2m}(\mathbb R^2)\to E_a^{l}(\mathbb R^2)$ has the Fredholm
property for each $\omega\in S^{n-3}$ iff the line $\Im\lambda=a+1-l-2m$ contains no
eigenvalues of the operator-valued function $\hat A(\lambda)$.
\end{lemma}

However, we prove below (see Lemma~\ref{l3.14}) that the operator
$A(\omega):E_a^{l+2m}(\mathbb R^2)\to E_a^{l}(\mathbb R^2)$ is not an isomorphism for
$a\in\mathbb R$,  $l\ge0$, and $\omega\in S^{n-3}$.

The following result is valid (see~\cite{SkDu90}).
\begin{lemma}\label{l3.4}
The operator $A:H_a^{l+2m}(\mathbb R^n)\to H_a^{l}(\mathbb R^n)$ is an isomorphism iff
the operator $A(\omega):E_a^{l+2m}(\mathbb R^2)\to E_a^{l}(\mathbb R^2)$ is an
isomorphism.
\end{lemma}

Combining Lemma~\ref{l3.4} with the fact that $A(\omega)$ is not an isomorphism allows us
to prove Theorem~\ref{t3.1}. Thus, it remains to show that the operator $A(\omega)$ is
not an isomorphism for $a\in\mathbb R$, $l\ge0$, and $\omega\in S^{n-3}$. To do so, we
preliminarily establish a priori estimates for solutions of Eq.~\eqref{3.26} and study
the adjoint operators.

The proof of a priori estimates in the spaces $E_a^l(\mathbb R^2)$ is based on the
well-known a priori estimate in Sobolev spaces (see, e.g., Theorem~15.3 in~\cite{ADN} and
the comment following it).
\begin{lemma}\label{l3.10}
Let $Q_1,Q_2\subset \mathbb R^n$ be bounded domains such that $\overline{Q_1}\subset
Q_2$. Assume that an operator
$$
\mathcal A=\sum\limits_{|\alpha|\le2m}a_\alpha(x)D_x^\alpha
$$
with infinitely differentiable coefficients $a_\alpha(x)$ is properly elliptic on
$\overline{Q_2}$. Then the following   estimate holds for all $u\in W^{l+2m}(Q_2)$\rm{:}
\begin{equation}\label{3.33}
\|u\|_{W^{l+2m}(Q_1)}\le c(\|\mathcal Au\|_{W^{l}(Q_2)}+\|u\|_{L_2(Q_2)}),
\end{equation}
where $c>0$ depends on $Q_1,Q_2$, and $M$,
$$
M=\max\limits_{|\beta|\le l_0}\max\limits_{|\alpha|\le
2m}\max\limits_{x\in\overline{Q_2}}|D^\beta a_\alpha(x)|, \qquad l_0=\max(l,1),
$$
and does not depend on $u$.
\end{lemma}
\begin{remark}
Theorem~3.1 in~\cite[Chap.~2, Sec.~3]{LM} ensures the validity of estimate~\eqref{3.33}
with the term $\|u\|_{W^{l+2m-1}(Q_2)}$ instead of $\|u\|_{L_2(Q_2)}$ on the right-hand
side. To obtain estimate~\eqref{3.33}, one must additionally apply the technique close to
that in~\cite[Chap.~5]{Miranda}.
\end{remark}

Denote by $W^k_\loc(\mathbb R^2\setminus\{0\})$ the space of distributions $v$ on
$\mathbb R^2\setminus\{0\}$ such that $\psi v\in W^k(\mathbb R^2)$ for all $\psi\in
C_0^\infty(\mathbb R^2\setminus\{0\})$.

\begin{lemma}\label{l3.9}
Let $v\in W^{l+2m}_\loc(\mathbb R^2\setminus\{0\})\cap E_{a-l-2m}^0(\{|y|<1\})$ and
$A(\omega)v\in E_a^l(\mathbb R^2)$ for some $\omega\in S^{n-3}$. Then $v\in
E_a^{l+2m}(\mathbb R^2)$ and
\begin{equation}\label{3.31}
\|v\|_{E_a^{l+2m}(\mathbb R^2)}\le c(\|A(\omega)v\|_{E_a^l(\mathbb
R^2)}+\|v\|_{E_{a-l-2m}^0(\{|y|<R\})}),
\end{equation}
where $R,c>0$ do not depend on $v$ and $\omega$.
\end{lemma}
\begin{proof}
1. Denote $Q_1^s=\{y\in\mathbb R^2:\ 2^s<|y|<2^{s+1}\}$ and $Q_2^s=\{y\in\mathbb R^2:\
2^{s-1}<|y|<2^{s+2}\}$, $s=0,\pm1,\pm2,\dots$. Evidently, $\overline{Q_1^s}\subset
Q_2^s$. It follows from the belonging $v\in W^{l+2m}_\loc(\mathbb R^2\setminus\{0\})$
that $v\in H_a^{l+2m}(Q_2^s)\cap E_a^{l+2m}(Q_2^s)$ for any $s$ (we set $\rho(y)=|y|$ in
the definition of the spaces $H_a^{l+2m}(Q_2^s)$ and $E_a^{l+2m}(Q_2^s)$). First, we
prove that
\begin{equation}\label{3.32}
\|v\|_{E_a^{l+2m}(Q_1^s)}^2\le
k_1(\|A(\omega)v\|_{E_a^l(Q_2^s)}^2+\|v\|_{E_{a-l-2m}^0(Q_2^s)}^2),\qquad s\le0,
\end{equation}
where $k_1,k_2,\dots{>0}$ do not depend on $v$, $\omega,$ and $s$.

Set $y'=2^{-s}y$. Clearly, $y'\in Q_j^0$ for $y\in Q_j^s$, $j=1,2$;
$s=0,\pm1,\pm2,\dots$. Therefore, setting $v^s(y')=v(2^sy')$ and applying
Lemma~\ref{l3.10}, we obtain, for $s\le 0$,
\begin{multline*}
\|v\|_{H_a^{l+2m}(Q_1^s)}^2\le k_2\sum\limits_{|\alpha|\le
l+2m}2^{2s(a-l-2m+|\alpha|)}\|D^\alpha_y v(y)\|_{L_2(Q_1^s)}^2\\=
k_2\sum\limits_{|\alpha|\le l+2m}2^{2s(a-l-2m+1)}\|D^\alpha_{y'}
v^s(y')\|_{L_2(Q_1^0)}^2\\
\le k_3 2^{2s(a-l-2m+1)}\Big(\sum\limits_{|\alpha|\le l}\|D_{y'}^\alpha
A(D_{y'},2^s\omega)v^s(y')\|_{L_2(Q_2^0)}^2+\|v^s(y')\|_{L_2(Q_2^0)}^2\Big)\\
= k_3 \Big(\sum\limits_{|\alpha|\le l}2^{2s(a-l+|\alpha|)}\|D_{y}^\alpha
A(D_{y},\omega)v(y)\|_{L_2(Q_2^s)}^2+2^{2s(a-l-2m)}\|v(y)\|_{L_2(Q_2^s)}^2\Big)\\
\le k_4 (\|A(D_y,\omega)v\|_{H_a^l(Q_2^s)}^2+\|v\|_{H_{a-l-2m}^0(Q_2^s)}^2).
\end{multline*}
The latter estimate is equivalent to~\eqref{3.32}.

2. To complete the proof, it remains to show that the estimate in~\eqref{3.32} is also
valid for $s>0$. To do so, we apply Lemma~\ref{l3.10} for $\mathcal A=A(D_{y'},D_z)$,
$u(y',z)=\exp(i2^s(\omega,z))v^s(y')$, and
$$
Q_j=Q_j^0\times\{z\in \mathbb R^{n-2}:\ |z_k|<j,\ k=1,\dots,n-2\},\qquad j=1,2.
$$
 Then we obtain
\begin{multline}\label{3.35}
\sum\limits_{\nu=0}^{l+2m}2^{2s\nu}\|v^s(y')\|_{W^{l+2m-\nu}(Q_1^0)}^2\\
\le k_5\Big(\sum\limits_{\nu=0}^l
2^{2s\nu}\|A(D_{y'},2^s\omega)v^s(y')\|_{W^{l-\nu}(Q_2^0)}^2+
\|v^s(y')\|_{L_2(Q_2^0)}^2\Big).
\end{multline}
Since $s>0$, it follows that inequality~\eqref{3.35} is equivalent to the following:
$$
2^{2s(l+2m-1)}\|v(y)\|_{W^{l+2m}(Q_1^s)}^2\le k_6(
2^{2s(l+2m-1)}\|A(D_{y},\omega)v(y)\|_{W^{l}(Q_2^s)}^2+2^{-2s}\|v(y)\|_{L_2(Q_2^s)}^2\Big).
$$
Multiplying both sides of this inequality by $2^{2s(a-l-2m+1)}$ yields
\begin{equation}\label{3.36}
\|v\|_{E_a^{l+2m}(Q_1^s)}^2\le
k_1(\|A(\omega)v\|_{E_a^l(Q_2^s)}^2+2^{2s(a-l-2m)}\|v\|_{L_2(Q_2^s)}^2),\qquad s>0.
\end{equation}

Summing~\eqref{3.32} and \eqref{3.36} with respect to $s=0,-1,-2,\dots$ and
$s=1,2,\dots$, respectively, and choosing a sufficiently large $R>0$, we
obtain~\eqref{3.31}.
\end{proof}

Using Lemmas~\ref{l3.3} and~\ref{l3.9}, we can prove the following result on regularity
of solutions of the equation
\begin{equation}\label{3.40}
A(\omega)v=f_0,\qquad \omega\in S^{n-3}.
\end{equation}
\begin{lemma}\label{l3.15}
Let the closed strip bounded by the lines $\Im\lambda=a_1+1-l_1-2m$ and
$\Im\lambda=a_2+1-l_2-2m$ contain no eigenvalues of the operator-valued function $\hat
A(\lambda)$. Suppose that $v\in E_{a_1}^{l_1+2m}(\mathbb R^2)$ is a solution of
Eq.~\eqref{3.40} for some $\omega\in S^{n-3}$, with right-hand side $f_0\in
E_{a_1}^{l_1}(\mathbb R^2)\cap E_{a_2}^{l_2}(\mathbb R^2)$. Then $v\in
E_{a_2}^{l_2+2m}(\mathbb R^2)$.
\end{lemma}
\begin{proof}
1. Consider a function $\eta\in C^\infty(\mathbb R)$ such that $\eta(r)=0$ for $r\le1$
and $\eta(r)=1$ for $r\ge2$. Denote by $[A(\omega),\eta]$ the commutator of $A(\omega)$
and $\eta$. It is clear that $\supp[A(\omega),\eta]v\subset\{y\in\mathbb R^2:\
1\le|y|\le2\}$. Therefore,
\begin{equation}\label{3.AEtav}
A(\omega)(\eta v)=\eta f_0+[A(\omega),\eta]v\in E_{a_1}^{l_1}(\mathbb R^2)\cap
E_{a_2}^{l_2}(\mathbb R^2)
\end{equation}
because $v\in W^{l+2m}_\loc(\mathbb R^2\setminus\{0\})$, where $l=\max(l_1,l_2)$, due to
Theorem~3.2 in~\cite[Chap.~2, Sec.~3]{LM}.

On the other hand, $\eta v$ vanishes near the origin, and hence $\eta v \in
E_{a_2-l_2-2m}^0(\{|y|<R\})$ for any $R>0$. Using this fact, relation~\eqref{3.AEtav},
and Lemma~\ref{l3.9}, we conclude that $\eta v\in E_{a_2}^{l_2+2m}(\mathbb R^2)$.

2. Since $\supp A(\omega)((1-\eta) v)\subset\{y\in\mathbb R^2:\ |y|\le2\}$, we obtain
(similarly to~\eqref{3.AEtav}) that
\begin{equation}\label{3.A1-Etav}
A(\omega)((1-\eta) v)\in H_{a_1}^{l_1}(\mathbb R^2)\cap H_{a_2}^{l_2}(\mathbb R^2).
\end{equation}
Therefore, by using Lemma~\ref{l3.3} and Remark~\ref{r3.2}, we conclude
that\footnote{Since the operator $A(\omega)$ contains lower-order terms, one must
consecutively apply Lemma~\ref{l3.3} finitely many times, cf.~\cite{Kondr,KondrOl}.}
$(1-\eta) v\in H_{a_2}^{l_2+2m}(\mathbb R^2)$, and hence $(1-\eta) v\in
E_{a_2}^{l_2+2m}(\mathbb R^2)$.

Thus, $v\in E_{a_2}^{l_2+2m}(\mathbb R^2)$.
\end{proof}

\subsection{}\label{subsec3.4}
In this subsection, we consider adjoint operators. Introduce the operator
$$
A'(\omega)=A'(D_y,\omega)=\sum\limits_{|\alpha|+|\beta|=2m}\overline{a_{\alpha\beta}}
\,\omega^\beta D_y^\alpha.
$$
The operator $A'(D_y,\omega)$ is formally adjoint to $A(D_y,\omega)$ with respect to the
Green formula, i.e.,
\begin{equation}\label{3.37}
\int\limits_{\mathbb R^2}A(D_y,\omega)u\overline{v}\,dy= \int\limits_{\mathbb
R^2}u\overline{A'(D_y,\omega)v}\,dy,\qquad \omega\in\mathbb R^{n-2},
\end{equation}
for all $u,v\in C_0^\infty(\mathbb R^2\setminus\{0\})$.

Consider the unbounded operators
\begin{gather*}
\mathcal A(\omega): \Dom(\mathcal A(\omega))\subset
E_{b-2m}^0(\mathbb R^2)\to E_{b}^0(\mathbb R^2),\\
\mathcal A(\omega)v=A(D_y,\omega)v,\qquad v\in \Dom(\mathcal A(\omega))=E_b^{2m}(\mathbb
R^2)
\end{gather*}
and
\begin{gather*}
\mathcal A'(\omega): \Dom(\mathcal A'(\omega))\subset
E_{-b}^0(\mathbb R^2)\to E_{2m-b}^0(\mathbb R^2),\\
\mathcal A'(\omega)v=A'(D_y,\omega)v,\qquad v\in \Dom(\mathcal
A'(\omega))=E_{2m-b}^{2m}(\mathbb R^2).
\end{gather*}

\begin{lemma}\label{l3.11}
The operator $\mathcal A'(\omega)$ is adjoint to $\mathcal A(\omega)$ with respect to the
inner product in $L_2(\mathbb R^2)$ for any $\omega\in S^{n-3}$.
\end{lemma}
\begin{proof}
Denote by $\mathcal A^*(\omega)$ the adjoint operator for $\mathcal A(\omega)$. Since
$C_0^\infty(\mathbb R^2\setminus\{0\})$ is dense in the spaces $E_b^{2m}(\mathbb R^2)$
and $E_{2m-b}^{2m}(\mathbb R^2)$, it follows that identity~\eqref{3.37} is valid for all
$u\in E_b^{2m}(\mathbb R^2)$ and $v\in E_{2m-b}^{2m}(\mathbb R^2)$. Therefore, $\mathcal
A'(\omega)\subset \mathcal A^*(\omega)$.

It remains to prove the inverse inclusion. Let $v\in \Dom(\mathcal A^*(\omega))\subset
E_{-b}^0(\mathbb R^2)$. Since $\mathcal A^*(\omega)v\in E_{2m-b}^0(\mathbb R^2)\subset
L_{2,\loc}(\mathbb R^2\setminus\{0\})$, it follows from Theorem~3.2 in~\cite[Chap.~2,
Sec.~3]{LM} that $v\in W^{2m}_\loc(\mathbb R^2\setminus\{0\})$. Therefore, by
Lemma~\ref{l3.9}, $v\in E_{2m-b}^{2m}(\mathbb R^2)$, and hence $\mathcal
A^*(\omega)\subset\mathcal A'(\omega)$.
\end{proof}

Consider the identity~\eqref{3.37} for $\omega=0$, substitute $u=u_1$ and $v=r^{2m-2}v_1$
into it, and set $\tau=\ln r$. Then we have
\begin{equation}\label{3.38}
\int\limits_{-\infty}^\infty d\tau\int\limits_0^{2\pi} \Big(\hat
A(\varphi,D_\varphi,D_\tau)u_1\overline{v_1}-u_1 \overline{\hat
A'(\varphi,D_\varphi,D_\tau-2i(m-1))v_1}\,\Big)d\varphi=0
\end{equation}
for all $u_1,v_1\in\{u\in C_0^\infty([0,2\pi]\times \mathbb R):\
D_\varphi^ju|_{\varphi=0}=D_\varphi^ju|_{\varphi=2\pi},\ j=0,1,\dots\}$, where $\hat
A'(\varphi,D_\varphi,D_\tau)$ is defined similarly to $\hat A(\varphi,D_\varphi,D_\tau)$.

Consider functions $\psi,\hat\psi\in C_0^\infty(\mathbb R)$ such that
$$
\psi(\tau)=0\quad\text{for}\quad |\tau|>1,\qquad
\int\limits_{-\infty}^\infty\psi(\tau)d\tau=1,
$$
$$
\hat\psi(\tau)=1\quad\text{for}\quad |\tau|<1,\qquad \hat\psi(\tau)=0\quad\text{for}\quad
|\tau|>2.
$$
Substituting $u_1=e^{i\lambda\tau}\psi(\tau)u_2(\varphi)$ and
$v_1=e^{i\overline{\lambda}\tau}\hat\psi(\tau)v_2(\varphi)$ into~\eqref{3.38}, we obtain
\begin{equation}\label{3.39}
\int\limits_0^{2\pi} \Big(\hat A(\varphi,D_\varphi,\lambda)u_2\overline{v_2}-u_2
\overline{\hat A'(\varphi,D_\varphi,\overline{\lambda}-2i(m-1))v_2}\,\Big)d\varphi=0
\end{equation}
for all $u_2,v_2\in C_{2\pi}^\infty[0,2\pi]$ and $\lambda\in\mathbb C$.

Along with $\hat A(\lambda)$, we consider the operator-valued function $\hat
A'(\lambda):W_{2\pi}^{2m}(0,2\pi)\to L_2(0,2\pi)$ given by
$$
\hat A'(\lambda)w=\hat A'(\varphi,D_\varphi,\lambda)w.
$$

We also introduce the unbounded operators
$$
\hat {\mathcal A}(\lambda),\hat {\mathcal A}'(\lambda):\Dom(\hat {\mathcal
A}(\lambda))=\Dom(\hat {\mathcal A}'(\lambda))\subset L_2(0,2\pi)\to L_2(0,2\pi)
$$
given by
$$
\hat {\mathcal A}(\lambda)w=\hat A(\varphi,D_\varphi,\lambda)w,\qquad w\in \Dom(\hat
{\mathcal A}(\lambda))=W_{2\pi}^{2m}(0,2\pi),
$$
$$
\hat {\mathcal A}'(\lambda)w=\hat A'(\varphi,D_\varphi,\lambda)w,\qquad w\in \Dom(\hat
{\mathcal A}'(\lambda)).
$$

Similarly to Lemma~\ref{l3.11}, we conclude from~\eqref{3.39} that the operator
$\hat{\mathcal A}'(\overline{\lambda}-2i(m-1))$ is adjoint to $\hat{\mathcal A}(\lambda)$
with respect to the inner product in $L_2(0,2\pi)$ for any $\lambda\in\mathbb C$. This
fact and the fact that $\hat A(\lambda)$ is a Fredholm operator with $\ind\hat
A(\lambda)=0$ imply:
\begin{lemma}\label{l3.12}
A number $\lambda$ is an eigenvalue of $\hat A(\lambda)$ iff $\overline{\lambda}-2i(m-1)$
is an eigenvalue of $\hat A'(\lambda)$.
\end{lemma}

\subsection{}\label{subsec3.5}
In this subsection, using the results of Secs.~\ref{subsec3.1}--\ref{subsec3.4}, we prove
the following result.
\begin{lemma}\label{l3.14}
Let $a\in\mathbb R$, $l\ge0$ be an integer, and $\omega\in S^{n-3}$. Then the operator
$A(\omega):E_a^{l+2m}(\mathbb R^2)\to E_a^{l}(\mathbb R^2)$ is not an isomorphism.
\end{lemma}

We preliminarily prove two lemmas on the property of the operator $A(\omega)$ to be an
isomorphism. These lemmas, together with properties of the adjoint operator, will enable
us to prove Lemma~\ref{l3.14} and hence Theorem~\ref{t3.1}.

\begin{lemma}\label{l3.16}
Assume that the strip $a_2+1-l-2m<\Im\lambda<a_1+1-l-2m$ contains no eigenvalues of the
operator-valued function $\hat A(\lambda)$. If the operator $A(\omega)$, $\omega\in
S^{n-3}$, is an isomorphism for some $a=a_0\in(a_2,a_1)$, then it is an isomorphism for
all $a\in(a_2,a_1)$.
\end{lemma}
\begin{proof}
1. It suffices to prove that  $\mathcal N(A(\omega))=\{0\}$ and $\mathcal
R(A(\omega))=E_b^l(\mathbb R^2)$ for each $a=b\in(a_2,a_1)$. It follows from
Lemma~\ref{l3.15} that, if $v\in\mathcal N(A(\omega))$ for $a=b$, $a_2<b<a_1$, then $v\in
\mathcal N(A(\omega))$ for $a=a_0$. Hence $v=0$, i.e., $\mathcal N(A(\omega))=\{0\}$ for
$a=b$.

2. Consider Eq.~\eqref{3.40} for $f_0\in C_0^\infty(\mathbb R^2\setminus\{0\})$ and
$a=a_0$. By assumption, this equation has a unique solution $v\in E_{a_0}^{l+2m}(\mathbb
R^2)$. By virtue of Lemma~\ref{l3.15}, $v\in E_b^{l+2m}(\mathbb R^2)$. Lemma~\ref{l3.8}
implies that $\mathcal R(A(\omega))$ (for $a=b$) is closed in $E_b^l(\mathbb R^2)$.
Combining this with the fact that $C_0^\infty(\mathbb R^2\setminus\{0\})$ is dense in
$E_b^l(\mathbb R^2)$ yields $\mathcal R(A(\omega))=E_b^l(\mathbb R^2)$.
\end{proof}

\begin{lemma}\label{l3.17}
Assume that each of the lines $\Im\lambda=a_2+1-l-2m$ and $\Im\lambda=a_1+1-l-2m$
contains an eigenvalue of the operator-valued function $\hat A(\lambda)$, and let the
strip $a_2+1-l-2m<\Im\lambda<a_1+1-l-2m$ contain no eigenvalues of the operator-valued
function $\hat A(\lambda)$. If the operator $A(\omega)$, $\omega\in S^{n-3}$, is an
isomorphism for some $a=a_0\in(a_2,a_1)$, then it is not an isomorphism for
$a\notin(a_2,a_1)$.
\end{lemma}
\begin{proof}
1. Let us prove that $\dim\mathcal N(A(\omega))>0$ for $a>a_1$. Set
$$
u=r^{i\lambda_0}\psi^0(\varphi),
$$
where $\lambda_0$ is an eigenvalue  of the operator-valued function $\hat A(\lambda)$
such that $\Im\lambda_0=a_1+1-l-2m$ and $\psi^0(\varphi)$ is the corresponding
eigenvector. In this case,
\begin{equation}\label{3.42}
A(0)u=0.
\end{equation}
Therefore,
$$
A(\omega)((1-\eta)u)=[A(0),(1-\eta)]u+(A(\omega)-A(0))((1-\eta)u)\equiv F,
$$
where $\eta$ is the same function as in the proof of Lemma~\ref{l3.15}.

Note that
\begin{equation}\label{3.(1-eta)u_notin}
(1-\eta)u\notin E_b^{l+2m}(\mathbb R^2)\qquad\text{for any}\qquad b\le a_1,
\end{equation}
\begin{equation}\label{3.(1-eta)u_in}
(1-\eta)u\in E_a^{l+2m}(\mathbb R^2)\qquad\text{for any}\qquad a>a_1.
\end{equation}

Clearly,
\begin{equation}\label{3.FinE}
F\in E_b^l(\mathbb R^2) \qquad\text{for any}\qquad b\in(a_1-1,+\infty).
\end{equation}
 By Lemma~\ref{l3.16}, Eq.~\eqref{3.40} with the right-hand side
$f_0=F$ has a unique solution $v\in E_b^{l+2m}(\mathbb R^2)$, where
$b\in(a_2,a_1)\cap(a_1-1,+\infty)$. In particular, this implies that  $w=(1-\eta)u-v$ is
not the zero function due to~\eqref{3.(1-eta)u_notin}.

Further, using the relation $v\in E_b^{l+2m}(\mathbb R^2)\subset E_b^{0}(\mathbb R^2)$
and taking~\eqref{3.FinE} into account, we deduce from Lemma~\ref{l3.9} that $v\in
E_{b+l+2m}^{l+2m}(\mathbb R^2)$. Repeating these arguments finitely many times, we obtain
that $v\in E_{a}^{l+2m}(\mathbb R^2)$ for any $a>a_1$. Combining this relation
with~\eqref{3.(1-eta)u_in}, we see that $w\in E_{a}^{l+2m}(\mathbb R^2)$ for $a>a_1$, and
hence $w\in\mathcal N(A(\omega))$ for $a>a_1$.

2. Now let $a<a_2$. If the line $\Im\lambda=a+1-l-2m$ contains an eigenvalue of $\hat
A(\lambda)$, then the conclusion of this lemma follows from Lemma~\ref{l3.8}. Therefore,
we assume that the line $\Im\lambda=a+1-l-2m$ contains no eigenvalues of $\hat
A(\lambda)$. In this case, the set $\mathcal R(A(\omega))$ is closed both for $A(\omega):
E_{a-l}^{2m}(\mathbb R^2)\to E_{a-l}^{0}(\mathbb R^2)$ and for $A(\omega):
E_{a}^{l+2m}(\mathbb R^2)\to E_{a}^{l}(\mathbb R^2)$.

2a. First, we prove that $d=\dim\mathcal R(A(\omega))^\bot>0$ for
$$A(\omega): E_{a-l}^{2m}(\mathbb R^2)\to E_{a-l}^{0}(\mathbb
R^2).$$ By virtue of Lemma~\ref{l3.12}, the lines $\Im\lambda=l+2m-a_2+1-2m$ and
$\Im\lambda=l+2m-a_1+1-2m$ contain eigenvalues of $\hat A'(\lambda)$, while the strip
$l+2m-a_1+1-2m<\Im\lambda<l+2m-a_2+1-2m$ contains no eigenvalues of $\hat A'(\lambda)$.
By assumption, the operator $A(\omega): E_{a_0}^{l+2m}(\mathbb R^2)\to
E_{a_0}^{l}(\mathbb R^2)$, $\omega\in S^{n-3}$, is an isomorphism. Therefore, by
Lemmas~\ref{l3.9} and~\ref{l3.8}, the operator $A(\omega): E_{a_0-l}^{2m}(\mathbb R^2)\to
E_{a_0-l}^{0}(\mathbb R^2)$ is also an isomorphism. Now it follows from Lemma~\ref{l3.11}
that the operator $$A'(\omega): E_{l+2m-a_0}^{2m}(\mathbb R^2)\to
E_{l+2m-a_0}^{0}(\mathbb R^2)$$ is an isomorphism. Applying part~1 of this proof to the
operator $A'(\omega)$, we conclude that $\dim\mathcal N(A'(\omega))>0$ for $A'(\omega):
E_{b}^{2m}(\mathbb R^2)\to E_{b}^{0}(\mathbb R^2)$, where $b>l+2m-a_2$. Therefore, by
virtue of Lemma~\ref{l3.11}, $d>0$ for $A(\omega): E_{a-l}^{2m}(\mathbb R^2)\to
E_{a-l}^{0}(\mathbb R^2)$, where $2m-(a-l)>l+2m-a_2$, i.e., $a<a_2$.

2b. It remains to prove that $\dim\mathcal R(A(\omega))^\bot=d$ for $$A(\omega):
E_{a}^{l+2m}(\mathbb R^2)\to E_{a}^{l}(\mathbb R^2).$$  Due to part~1b of the proof,
Eq.~\eqref{3.40} with right-hand side $f_0\in E_{a-l}^{0}(\mathbb R^2)$ has a solution
$v\in E_{a-l}^{2m}(\mathbb R^2)$ iff
\begin{equation}\label{3.43}
(f_0,f_j)_{E_{a-l}^0(\mathbb R^2)}=0,\qquad j=1,\dots,d,
\end{equation}
where $f_1,\dots,f_d\in E_{a-l}^0(\mathbb R^2)$ are linearly independent functions. It
follows from Lem\-ma~\ref{l3.9} that conditions~\eqref{3.43} are necessary and sufficient
for Eq.~\eqref{3.40} with right-hand side $f_0\in E_{a}^{l}(\mathbb R^2)$ to have a
solution $v\in E_{a}^{l+2m}(\mathbb R^2)$. It follows from the Schwarz inequality and
from the boundedness of the embedding $E_a^l(\mathbb R^2)\subset E_{a-l}^0(\mathbb R^2)$
that
$$
|(f_0,f_j)_{E_{a-l}^0(\mathbb R^2)}|\le\|f_0\|_{E_{a-l}^0(\mathbb
R^2)}\|f_j\|_{E_{a-l}^0(\mathbb R^2)}\le c\|f_0\|_{E_{a}^l(\mathbb
R^2)}\|f_j\|_{E_{a-l}^0(\mathbb R^2)},
$$
where $c>0$ does not depend on $f_0$. Hence, by Riesz' theorem, there are functions
$F_j\in E_{a}^l(\mathbb R^2)$, $j=1,\dots,d$, such that
$$
(f_0,f_j)_{E_{a-l}^0(\mathbb R^2)}=(f_0,F_j)_{E_{a}^l(\mathbb R^2)}\qquad\text{for
all}\qquad f_0\in E_{a}^l(\mathbb R^2),
$$
and the functions $F_j$ are linearly independent. Thus, $\dim\mathcal R(A(\omega))^\bot$
in $E_a^l(\mathbb R^2)$ is equal to $d$.
\end{proof}

\begin{proof}[Proof of Lemma~$\ref{l3.14}$]
1. First, we show that the operator
$$A(\omega):E_{l+m}^{l+2m}(\mathbb R^2)\to E_{l+m}^{l}(\mathbb
R^2)$$ is not an isomorphism for any $l\ge0$. To do so, we prove that $\lambda_0=i(1-m)$
is an eigenvalue of $\hat A(\lambda)$. Consider a homogeneous polynomial $q(y)$ of order
$m-1$ and write it in the polar coordinates, $q(y)=r^{m-1}\tilde q(\varphi)$, where
$\tilde q\in C_{2\pi}^\infty[0,2\pi]$. We have
$$
0=A(D_y,0)q(y)=r^{-2m}\hat A(\varphi,D_\varphi,rD_r)(r^{m-1}\tilde
q(\varphi))=r^{-m-1}\hat A(\varphi,D_\varphi,i(1-m))\tilde q(\varphi).
$$
Hence, $\lambda_0=i(1-m)$ is an eigenvalue and $\tilde q(\varphi)$ is the corresponding
eigenvector. Since the line $\Im\lambda=1-m=m+1-2m$ contains the eigenvalue $\lambda_0$,
it follows from Lemma~\ref{l3.8} that the operator $A(\omega):E_{l+m}^{l+2m}(\mathbb
R^2)\to E_{l+m}^{l}(\mathbb R^2)$ does not have the Fredholm property. Therefore, it is
not an isomorphism.

2. Now we prove that the operator $A(\omega):E_{a}^{2m}(\mathbb R^2)\to E_{a}^{0}(\mathbb
R^2)$ is not an isomorphism for any $a$, $a\ne m$. Assume, to the contrary, that
$A(\omega)$ is an isomorphism for some $a\ne m$. Then, by Lemma~\ref{l3.11},
\begin{equation}\label{3.44}
\text{the operator}\quad A'(\omega):E_{2m-a}^{2m}(\mathbb R^2)\to E_{2m-a}^{0}(\mathbb
R^2)\quad\text{is an isomorphism.}
\end{equation}

Note that
$$
\overline{A(D_y,\omega)u(y)}\equiv[A'(D_{y'},\omega)w(y')]|_{y'=-y},
$$
where $w(y')=\overline{u(-y')}$. It follows from this relation and from~\eqref{3.44} that
the operator $A(\omega):E_{2m-a}^{2m}(\mathbb R^2)\to E_{2m-a}^{0}(\mathbb R^2)$ is also
an isomorphism. This contradicts Lemma~\ref{l3.17} because the strip bounded by the lines
$\Im\lambda=a+1-2m$ and $\Im\lambda=(2m-a)+1-2m$ contains the eigenvalue
$\lambda_0=i(1-m)$ of the operator-valued function $\hat A(\lambda)$.

3. Finally, we prove that the operator $A(\omega):E_{a}^{l+2m}(\mathbb R^2)\to
E_{a}^{l}(\mathbb R^2)$ is not an isomorphism for all $\omega\in S^{n-3}$, $l>0$, and
$a\ne l+m$. Assume, to the contrary, that it is an isomorphism for some $\omega\in
S^{n-3}$, $l>0$, and $a\ne l+m$. Then, by Lemma~\ref{l3.8}, the line
$\Im\lambda=a+1-l-2m$ contains no eigenvalues of the operator-valued function $\hat
A(\lambda)$. Therefore, according to part~2 of this proof, either $\dim\mathcal
N(A(\omega))>0$ or $\dim\mathcal R(A(\omega))^\bot>0$ for the operator
$A(\omega):E_{a-l}^{2m}(\mathbb R^2)\to E_{a-l}^0(\mathbb R^2)$.

Let $\dim\mathcal N(A(\omega))>0$ for the above operator $A(\omega)$. Hence, there exists
a function $v\in E_{a-l}^{2m}(\mathbb R^2)$ such that $v\ne0$ and $A(\omega)v=0$. By
Lemma~\ref{l3.9}, $v\in E_{a}^{l+2m}(\mathbb R^2)$, and hence $\dim\mathcal
N(A(\omega))>0$ for $A(\omega):E_{a}^{l+2m}(\mathbb R^2)\to E_{a}^l(\mathbb R^2)$. This
contradicts our assumption.

Let $\dim\mathcal R(A(\omega))^\bot>0$ for $A(\omega):E_{a-l}^{2m}(\mathbb R^2)\to
E_{a-l}^0(\mathbb R^2)$. Since $A(\omega):E_{a}^{l+2m}(\mathbb R^2)\to E_{a}^l(\mathbb
R^2)$ is an isomorphism, it follows that  the equation
$$
A(\omega)v=f_0
$$
has a solution $v\in E_a^{l+2m}(\mathbb R^2)\subset E_{a-l}^{2m}(\mathbb R^2)$ for each
$f_0\in E_a^{l}(\mathbb R^2)$. On the other hand, $E_a^{l}(\mathbb R^2)$ is dense in
$E_{a-l}^{0}(\mathbb R^2)$, while $\mathcal R(A(\omega))$ is closed in
$E_{a-l}^{0}(\mathbb R^2)$ for $A(\omega):E_{a-l}^{2m}(\mathbb R^2)\to E_{a-l}^0(\mathbb
R^2)$. Therefore, $\mathcal R(A(\omega))=E_{a-l}^0(\mathbb R^2)$, which contradicts the
assumption that $\dim\mathcal R(A(\omega))^\bot>0$.
\end{proof}

Theorem~\ref{t3.1} follows from Lemmas~\ref{l3.4} and~\ref{l3.14}.

\section{A Priori Estimates of Solutions in Bounded Domains}\label{sec4}

\subsection{}\label{subsec4.1}

In this section, we obtain a priori estimates for solutions of nonlocal elliptic problems
in weighted spaces. Combining these estimates with the existence of a right regularizer,
which we construct in the next section, we prove the Fredholm property for the
corresponding nonlocal operator. Let us discuss the choice of weighted spaces. For each
set $\mathcal K_j$, $j=1,2,3$, we may assume that either the set $K$ in the definition of
the spaces $H_a^k(Q)=H_a^k(Q,K)$ and $H_a^{k-1/2}(\Gamma)=H_a^{k-1/2}(\Gamma,K)$ contains
the set $\mathcal K_j$ or it does not. This is equivalent to whether or not right-hand
sides and solutions of nonlocal problems in bounded domains have singularities near the
set $\mathcal K_j$. If $\mathcal K_j\subset K$, then the model operators corresponding to
the points of the set $\mathcal K_j$ and playing a fundamental role in obtaining a priori
estimates and constructing a right regularizer must be considered in weighted spaces;
otherwise, in Sobolev spaces.

Consider the set $\mathcal K_1=\partial Q\setminus \bigcup_i\Gamma_i$ of conjugation
points. It is shown in~\cite{SkMs86,SkRJMP} (see also~\cite{GurDAN04}) that generalized
solutions of nonlocal problems can have power-law singularities near the set $\mathcal
K_1$. Therefore, we always assume that $\mathcal K_1\subset K$, while the corresponding
model operators act on weighted spaces in dihedral angles.

Consider the set $\mathcal K_3\subset Q$. It follows from Theorem~\ref{t3.1} that the
model operator on weighted spaces in $\mathbb R^n$ is not an isomorphism (moreover, one
can show that it does not have even the Fredholm property on weighted spaces, cf.
Remark~2.2 in~\cite[Chap.~6, Sec.~2]{NP}). Therefore, we assume that right-hand sides and
solutions of nonlocal problems in bounded domains have no singularities inside the domain
$Q$, while the corresponding model operator acts on Sobolev spaces. If $\mathcal
K_3=\varnothing$, this assumption leads to no difficulties. However, if $\mathcal
K_3\ne\varnothing$, the following difficulty arises. Take a point
$g\in\overline{\Gamma_i}\cap\mathcal K$ such that $\omega_{is}(g)\in\mathcal K_3$, and
let a function $u$ belong to the weighted space $H_a^{l+2m}$ near the point $g$ and to
the Sobolev space $W^{l+2m}$ near the point $\omega_{is}(g)$. Since $\omega_{is}$ is a
smooth nondegenerate transformation, it follows that the function $(B_{i\mu
s}(x,D)u)(\omega_{is}(x))|_{\Gamma_i}$ occurring in nonlocal conditions~\eqref{1.4}
belongs to $W^{l+2m-m_{i\mu}-1/2}$ near the point $g$; however, in general, it does not
belong to $H_a^{l+2m-m_{i\mu}-1/2}$ near the point $g$. Therefore, the corresponding
nonlocal operator appears to be unbounded on weighted spaces. To eliminate this obstacle,
we additionally assume that $a>l+2m-1$ in the case $\mathcal K_3\ne\varnothing$, which
ensures the inclusion $W^{l+2m-m_{i\mu}-1/2}\subset H_a^{l+2m-m_{i\mu}-1/2}$ in a
neighborhood of $g\in\overline{\Gamma_i}\cap\mathcal K$ (cf. Lemma~\ref{l4.3} below).

Consider the set $\mathcal K_2\subset\bigcup_i\Gamma_i$. We may either include or not
include the set $\mathcal K_2$ in the set $K$. In the first case, we consider model local
operators on weighted spaces. In the second case, we consider model local operators on
Sobolev spaces. The advantage of the ``weighted case'' is that we solve a nonlocal
problem in the whole scale of spaces (depending on the weight parameter $a\in\mathbb R$).
However, the disadvantage is that we must impose some assumptions on the location of
eigenvalues of an auxiliary problem with the parameter $\lambda$ and require that an
auxiliary operator with the parameter $\omega\in S^{n-3}$ be an isomorphism (the latter
is often hard to verify), see Theorem~\ref{t2.R^n_+Isomorphism}. The advantage of the
case of Sobolev spaces is that the model operators are isomorphisms without any
additional assumptions. The disadvantage is that, if $\mathcal K_2\ne\varnothing$, we
must suppose $a>l+2m-1$ even if $\mathcal K_3=\varnothing$ (the reason is similar to that
in the above case $\mathcal K_3\ne\varnothing$).

The following consistency condition integrates all the above cases.
\begin{condition}[consistency condition]\label{cond4.1}
\begin{enumerate}
\item
If $\mathcal K_3=\varnothing$, then either
\begin{enumerate}
\item $a\in\mathbb R$ and $K=\mathcal K_1\cup\mathcal K_2$, or
\item  $a>l+2m-1$ and $K=\mathcal K_1$.
\end{enumerate}
\item
If $\mathcal K_3\ne\varnothing$, then $a>l+2m-1$ and either
\begin{enumerate}
\item  $K=\mathcal K_1\cup\mathcal K_2$, or
\item $K=\mathcal K_1$.
\end{enumerate}
\end{enumerate}
\end{condition}

To conclude this subsection, we prove two auxiliary results. Denote
$$
\mathcal M^\delta=\{x\in\mathbb R^n:\ \dist(x,\mathcal M)<\delta\}
$$
for any set $\mathcal M\subset \mathbb R^n$ and $\delta>0$.

\begin{lemma}\label{lzeta0v}
Let $\zeta\in C^\infty(\mathbb R^n)$ be a function such that $\zeta(x)=0$ for
$x\in\mathcal K_1$. Then
\begin{equation}\label{eqzeta0v0}
\|\zeta v\|_{H_a^l(Q)}\le c\delta\|v\|_{H_a^l(Q)}
\end{equation}
for all $v\in H_a^l(Q)$ such that $\supp v\subset\overline{Q}\cap\mathcal K_1^\delta$,
where $c>0$ does not depend on $\delta$ and $v$.
\end{lemma}
\begin{proof}
Since $\zeta(x)=0$ for $x\in\mathcal K_1$, it follows from the Taylor formula that
\begin{equation}\label{eqzeta0v1}
|\zeta(x)|\le k_1\delta,\qquad x\in \mathcal K_1^{\delta}.
\end{equation}
Therefore, using~\eqref{eqzeta0v1}, we obtain
\begin{multline*}
\|\zeta v\|_{H_a^l(Q)}^2=\sum\limits_{|\beta|\le l}\,\int\limits_{Q\cap\mathcal
K_1^\delta}\rho^{2(a-l+|\beta|)}|\zeta D^\beta v|^2
dx\\
+\sum\limits_{|\alpha|=1}^l\sum\limits_{|\beta|\le
l-|\alpha|}\,\int\limits_{Q\cap\mathcal
K_1^\delta}\rho^{2|\alpha|}\rho^{2(a-l+|\beta|)}|D^\alpha\zeta
D^\beta v|^2 dx\\
\le  \sum\limits_{|\beta|\le l}k_1^2\delta^2\int\limits_{Q\cap\mathcal
K_1^\delta}\rho^{2(a-l+|\beta|)}|D^\beta v|^2
dx+\sum\limits_{|\alpha|=1}^l\sum\limits_{|\beta|\le
l-|\alpha|}k_2\delta^{2|\alpha|}\int\limits_{Q\cap\mathcal
K_1^\delta}\rho^{2(a-l+|\beta|)}|D^\beta v|^2 dx,
\end{multline*}
which implies~\eqref{eqzeta0v0}.
\end{proof}

\begin{lemma}\label{lZetaDeltaU}
Let $\zeta_\delta\in C^\infty(\mathbb R^n)$ be a family of functions such that
$\supp\zeta_\delta\subset \mathcal K_1^{\delta}$ and
\begin{equation}\label{eqZetaDeltaU}
|D^\beta\zeta_\delta(x)|\le c_1\delta^{-|\beta|}, \qquad x\in Q,\quad |\beta|\le l,
\end{equation}
where $c_1>0$ does not depend on $\delta$. Then
\begin{equation}\label{ZetaDeltaU'}
\|\zeta_\delta u\|_{H_a^l(Q)}\le c_2\|u\|_{H_a^l(Q)}
\end{equation}
for all $u\in H_a^l(Q)$, where $c_2>0$ does not depend on $\delta$ and $u$.
\end{lemma}
\begin{proof}
Using~\eqref{eqZetaDeltaU}, we obtain
\begin{multline*}
\|\zeta_\delta u\|_{H_a^l(Q)}^2=\sum\limits_{|\alpha|+|\beta|\le
l}\,\int\limits_{Q\cap\mathcal
K_1^\delta}\rho^{2|\alpha|}\rho^{2(a-l+|\beta|)}|D^\alpha\zeta_\delta D^\beta u|^2 dx\\
\le k \sum\limits_{|\beta|\le l}\,\int\limits_{Q\cap\mathcal
K_1^\delta}\rho^{2(a-l+|\beta|)}|D^\beta u|^2 dx, \end{multline*} which
implies~\eqref{ZetaDeltaU'}.
\end{proof}
\begin{remark}\label{rzeta0v}
Lemmas~\ref{lzeta0v} and~\ref{lZetaDeltaU} are true for the spaces $H_a^l(Q)$ replaced by
$H_a^{l-1/2}(\Gamma)$, where $\Gamma$ is a smooth $(n-1)$-dimensional manifold such that
$\overline\Gamma\subset\overline Q$ and $l-1/2\ge 1/2$. To prove this, it suffices to use
the corresponding bounded operator of extension acting from $H_a^{l-1/2}(\Gamma)$ to
$H_a^{l}(Q)$.
\end{remark}

\subsection{}\label{subsec4.2}

We introduce the linear operator
$$
\mathbf L=\{A,B_{i\mu}\}
$$
corresponding to problem~\eqref{1.3}, \eqref{1.4}. It follows from Lemma~\ref{l4.6} (see
below) that the operator $\mathbf L: H_a^{l+2m}(Q)\to\mathcal H_a^l(Q,\Gamma)$ is
bounded.

The main result of this section is as follows.
\begin{theorem}\label{t4.1}
Let Conditions~$\ref{cond1.1}$--$\ref{cond1.4}$ and~$\ref{cond4.1}$ hold. Assume that the
line $\Im\lambda=a+1-l-2m$ contains no eigenvalues of $\hat{\mathcal L}_g(\lambda)$ for
any $g\in K$ and $\dim\mathcal N(\mathcal L_g(\omega))=\codim\mathcal R(\mathcal
L_g(\omega))=0$ for any $g\in K$ and $\omega\in S^{n-3}$. Then the following estimate
holds for all $u\in H_a^{l+2m}(Q)$$:$
\begin{equation}\label{4.1}
\|u\|_{H_a^{l+2m}(Q)}\le c(\|\mathbf L u\|_{\mathcal
H_a^l(Q,\Gamma)}+\|u\|_{H_{a}^{l+2m-1}(Q)}),
\end{equation}
where $c>0$ does not depend on $u$.
\end{theorem}

Let $A^0$ and $B_{i\mu s}^0$ denote the principal homogeneous parts of the operators
$A(x,D)$ and $B_{i\mu s}(x,D)$, respectively. Set
$$
B_{i\mu}^0u=B_{i\mu0}^0u|_{\Gamma_i}.
$$

For each $\varepsilon>0$, we introduce a function $\xi=\xi_\varepsilon\in
C_0^\infty(\mathbb R^n)$ such that $\xi(x)=1$ for $x\in\mathcal K_1^{\varepsilon/2}$,
$\supp\xi(x)\subset\mathcal K_1^{\varepsilon}$, and
\begin{equation}\label{4.1'}
|D^\beta\xi(x)|\le k_1\varepsilon^{-|\beta|}, \qquad x\in Q,
\end{equation}
where $k_1=k_1(\beta)>0$ does not depend on $\varepsilon$. Since $\omega_{is}$ are
$C^\infty$ diffeomorphisms, it follows that
\begin{equation}\label{4.xiOmega=0}
\supp\xi(\omega_{is}(x))\subset\mathcal K_1^{\varepsilon''},
\end{equation}
where $\varepsilon''=\varepsilon''(i,s,\varepsilon)\to0$ as $\varepsilon\to 0$
$(i=1,\dots N_1;\ s=0,\dots,S_i)$.

We assume that $\varepsilon>0$ is so small that
\begin{equation}\label{eqStar}
0<\varepsilon''<\dist(\mathcal K_1,\mathcal K_2\cup\mathcal K_3)/4.
\end{equation}
Later on, we will make additional assumptions concerning $\varepsilon$ (see the proofs of
Lemmas~\ref{l4.4} and~\ref{l5.1}).

Consider the operators
$$
B_{i\mu}^1 u=\sum\limits_{s=1}^{S_i}(B_{i\mu s}^0(x,D)(\xi
u))(\omega_{is}(x))|_{\Gamma_i},
$$
$$
B_{i\mu}^2 u=\sum\limits_{s=1}^{S_i}(B_{i\mu s}^0(x,D)((1-\xi)
u))(\omega_{is}(x))|_{\Gamma_i},
$$
$$
B_{i\mu}^3=B_{i\mu}-B_{i\mu}^1-B_{i\mu}^2,\qquad A^1=A-A^0.
$$
The operators $B_{i\mu}^1$ correspond to nonlocal terms supported near the set $\mathcal
K_1$ and the operators $B_{i\mu}^2$ to nonlocal terms supported outside the set $\mathcal
K_1$, while $B_{i\mu}^3$ and $A^1$ correspond to lower-order terms (compact
perturbations).

Denote $B^k=\{B_{i\mu}^k\}_{i,\mu}$, $k=0,\dots,3$, $B=B^0+\dots+B^3$, and $C=B^0+B^1$.

Along with the operator $\mathbf L=(A,B)$ we consider the bounded operators
$$
\mathbf L^0=(A^0,B^0):H_a^{l+2m}(Q)\to\mathcal H_a^l(Q,\Gamma),\qquad \mathbf
L^1=(A^0,C):H_a^{l+2m}(Q)\to\mathcal H_a^l(Q,\Gamma).
$$

We first obtain an a priori estimate (similar to~\eqref{4.1}) for the operator $\mathbf
L^1$ with sufficiently small $\varepsilon$. Then we prove a fundamental property of the
operators $B_{i\mu}^2$ related to the fact that the operators $B_{i\mu}^2$ correspond to
nonlocal terms supported outside the set $\mathcal K_1$. Combining these results will
allow us to prove Theorem~\ref{t4.1}.

\begin{lemma}\label{l4.4}
Let the conditions of Theorem~$\ref{t4.1}$ be fulfilled. Then there is an $\varepsilon>0$
such that the following estimate holds for all $u\in H_a^{l+2m}(Q)$$:$
\begin{equation}\label{4.6}
\|u\|_{H_a^{l+2m}(Q)}\le c(\|\mathbf L^1 u\|_{\mathcal
H_a^l(Q,\Gamma)}+\|u\|_{H_{a}^{l+2m-1}(Q)}),
\end{equation}
where $c>0$ does not depend on $u$.
\end{lemma}
\begin{proof}
1. For any point $g\in\mathcal K_1$, denote by $\mathcal O(g)$ the orbit of $g$, see
Sec.~\ref{subsec1.1}. By Condition~\ref{cond1.3}, each orbit $\mathcal O(g)$ consists of
finitely many points $g_j$, $j=1,\dots,N(g)$. Set
\begin{equation*}%\label{2.chi_m}
\chi_m=\chi_m(g)=\min\limits_{j,\rho, k,s}\chi_{j\rho ks}\quad (j,k=1,\dots,N(g);\
\rho=1,2;\ s=0,1,\dots, S_{j\rho k}).
\end{equation*}
Clearly, $\chi_m\le 1$. Let $x'\to x(g,j)$ be the change of variables inverse to the
change of variables $x\to x'(g,j)$ from Sec.~\ref{subsec1.1}. The transformation $x'\to
x(g,j)$ takes each ball $B_{\chi_m\delta}$ onto some neighborhood ${\hat B}_\delta(g_j)$
of the point $g_j$ in such a way that the diameter of ${\hat B}_\delta(g_j)$ tends to
zero as $\delta\to0$. (Note that ${\hat B}_\delta(g_j)$ need not be a ball.) For each
orbit $\mathcal O(g)$, we take a sufficiently small number $\delta=\delta(g)>0$ such that
${\hat B}_\delta(g_j)\subset\hat V(g_j)$, $j=1,\dots,N(g)$, and the operator $\mathcal
L_g''$ has a bounded inverse for $\delta=\delta(g)$ (see Corollary~\ref{cort2.2'}).

It is clear that the union
$$
\bigcup_{g\in\mathcal K_1}\bigcup_{j=1}^{N(g)}{\hat B}_{\delta(g)}(g_j)
$$
covers the set $\mathcal K_1$. We choose finitely many points $g^t\in\mathcal K_1$,
$t=1,\dots T$, such that
$$
\mathcal K_1\subset\bigcup\limits_{t,j}{\hat B}_{\delta(g^t)}(g_j^t).
$$

Let functions $\varphi_j^t\in C_0^\infty(\mathbb R^n)$ form a partition of unity for the
set $\mathcal K_1$ subordinated to the covering $\{{\hat B}_{\delta(g^t)}(g_j^t)\}$. Now
we will pass from the partition of unity $\{\varphi_j^t\}$ to another partition of unity
$\{\xi_j^t\}$ such that each function $\xi_j^t$, being written in the local coordinates
$x'=(y',z')$, does not depend on $y'$ in a neighborhood of the edge $\mathcal P$. To do
so, we denote the function $\varphi_j^t(x)$ written in the variables $x'=(y',z')$ by
$\varphi_j^t(y',z')$. Clearly, there is a number $a'<\chi_m\delta(g^t)$ such that
$$
\varphi_j^t(0,z')=0
$$
for $a'<|z'|<\chi_m\delta(g^t)$. We can assume without loss of generality that the
function $\varphi_j^t(0,z')$ is extended by zero for $|z'|\ge \chi_m\delta(g^t)$, and it
remains to be infinitely differentiable.

Denote by $\psi^t\in C_0^\infty(\mathbb R^2)$ a function such that $\psi^t(y')=1$ for
$|y'|<\varepsilon_1'$ and $\psi_j^t(y')=0$ for $|y'|>2\varepsilon_1'$, where
$\varepsilon_1'>0$ is so small that
\begin{equation}\label{eqVstavka4.5}
\left\{(2\varepsilon_1')^2+(a')^2\right\}^{1/2}<\chi_m\delta(g^t),
\end{equation}
and $\varepsilon_1'$ does not depend on $\varepsilon$.

Set
$$
\xi_j^t(y',z')=\psi^t(y')\varphi_j^t(0,z').
$$
By virtue of~\eqref{eqVstavka4.5}, we have
\begin{equation}\label{eqVstavka4.6}
\supp \xi_j^t(y',z')\subset B_{\chi_m\delta(g^t)}.
\end{equation}
It is also clear that
\begin{equation}\label{eqVstavka4.7}
\xi_j^t(y',z')=\varphi_j^t(0,z'),\qquad |y'|<\varepsilon_1'.
\end{equation}

Denote by $\xi_j^t(x)$ the functions $\xi_j^t(y',z')$ written in the variables
$x=x(g^t,j)$. Since $\supp \xi_j^t(x)\subset {\hat B}_{\delta(g^t)}(g_j^t)$, we can
extend each function $\xi_j^t(x)$ by zero outside the neighborhood ${\hat
B}_{\delta(g^t)}(g_j^t)$ to obtain the function infinitely differentiable on $\mathbb
R^n$.

Obviously, we have
\begin{equation}\label{eqVstavka4.8}
\sum\limits_{j,t}\xi_j^t(x)=\sum\limits_{j,t}\varphi_j^t(x)=1,\qquad x\in\mathcal K_1.
\end{equation}

2. Take an arbitrary function $u\in H_a^{l+2m}(Q)$. If $g_j^t\in\overline{\Gamma_i}$ and
$x\in\hat V(g_j^t)$, then it follows from Condition~\ref{cond1.3} that $\omega_{is}(x)\in
V(g_p^t)$ for some $1\le p\le N(g^t)$. Denote $u_p^t(x)=u(x)$ for $x\in Q\cap V(g_p^t)$.
Then $u_p^t(\omega_{is}(x))=u(\omega_{is}(x))$ for $x\in Q\cap \hat V(g_j^t)$. Let $x\to
x'(g^t,j)$ be the change of variables from Sec.~\ref{subsec1.1}, corresponding to the
orbit $\mathcal O(g^t)$. Denote the functions $\xi_j^t$ and $u_j^t$ written in the new
variables $x'$ by the same symbols (which leads to no confusion) and let
$u^t=(u_1^t,\dots, u_{N(g^t)}^t)$. Applying Corollary~\ref{cort2.2'}, we obtain
$$
\|\xi_q^tu\|_{H_a^{l+2m}(Q)}\le k_1\|\xi_q^t u_q^t\|_{H_a^{l+2m}(\Theta_q)}\le
k_1\|\xi_q^t u^t\|_{\mathcal H_a^{l+2m}(\Theta)}\le k_2\|\mathcal
L_{g^t}''(\xi_q^tu^t)\|_{\mathcal H_a^l(\Theta,\Gamma)},
$$
where $q=1,\dots,N(g^t)$, while $k_1,k_2,{\dots}>0$ do not depend on $u$.

It follows from~\eqref{eqVstavka4.6} that
$$
\xi_q^t(\mathcal G_{j\rho ks}y',z')=0,\qquad |x'|>\delta(g^t).
$$
 Therefore, $\mathcal L_{g^t}''(\xi_q^tu^t)=\mathcal
L_{g^t}'(\xi_q^t u^t)$, and, by using Leibniz' formula, we have
\begin{multline}\label{4.AprL^1_1}
\|\xi_q^tu\|_{H_a^{l+2m}(Q)}\le k_2\|\mathcal L_{g^t}'(\xi_q^tu^t)\|_{\mathcal
H_a^l(\Theta,\Gamma)} \le k_3 \Big(\|\mathbf L^1 u\|_{\mathcal
H_a^l(Q,\Gamma)}+\|u\|_{H_a^{l+2m-1}(Q)}\\
+ \sum\limits_{h=1,2}\sum\limits_{j,\rho,\mu}\sum\limits_{(k,s)\ne(j,0)}\|\Psi_{j\rho\mu
ks}^h\|_{H_a^{l+2m-m_{j\rho\mu}-1/2}(\Gamma_{j\rho})}\Big),
\end{multline}
where
$$
\Psi_{j\rho\mu ks}^1=\big(B_{j\rho\mu ks}^0(x',D_{y'},D_{z'})((1-\xi)
\xi_q^tu_k^t)\big)\big(\mathcal G_{j\rho ks}y',z'\big)|_{\Gamma_{j\rho}},
$$
$$
\Psi_{j\rho\mu ks}^2=\big(\xi_q^t(\mathcal G_{j\rho
ks}y',z')-\xi_q^t(y',z')\big)(B_{j\rho\mu ks}^0(x',D_{y'},D_{z'})(\xi u_k^t))(\mathcal
G_{j\rho ks}y',z')|_{\Gamma_{j\rho}}.
$$

Denote by $\xi_{qk}^t(x)$ the function $\xi_q^t(y',z')$ written in the variables
$x=x(g^t,k)$. Clearly, $\supp \xi_{qk}^t\subset \hat B_{\delta(g^t)}(g_k^t)$. Passing to
the variables $\hat x =\omega_{is}(x)$, we estimate the norm of $\Psi_{j\rho\mu ks}^1$ in
the following way:
\begin{multline*}
\|\Psi_{j\rho\mu ks}^1\|_{H_a^{l+2m-m_{j\rho\mu}-1/2}(\Gamma_{j\rho})}\\
\le k_4\|B_{i\mu s}^0(\hat x,D)((1-\xi)\xi_{qk}^t u)(\hat x
)|_{\omega_{is}(\Gamma_i)}\|_{H_a^{l+2m-m_{i\mu}-1/2}(\omega_{is}(\Gamma_i))}.
\end{multline*}
Denote
$$
Q_b=\{x\in Q:\ \dist (x,\partial Q)>b\},
$$
where $b>0$. Since $\hat B_{\delta(g^t)}(g_k^t)\subset\hat V(g_k^t)$, it follows from
Condition~\ref{cond1.3} that the set
$$
\Omega_0=\big(\overline{\omega_{is}(\Gamma_i)}\cap \hat
B_{\delta(g^t)}(g_k^t)\big)\setminus\mathcal K_1^{\varepsilon/2}
$$
intersects neither $\mathcal K_1$ nor $\mathcal K_2$. Therefore, there exists a number
$b=b(\varepsilon)>0$ such that $\Omega_0\subset Q_b$. Since
$$
\supp\big((1-\xi)\xi_{qk}^t\big)\big|_{\omega_{is}(\Gamma_i)}\subset\Omega_0\subset Q_b,
$$
using the last inequality, the equivalence of the norms in the spaces $H_a^{l+2m}(Q_b)$
and $W^{l+2m}(Q_b)$, and Lemma~\ref{l3.10}, we obtain
\begin{multline}\label{eqVstavka4.10}
\|\Psi_{j\rho\mu ks}^1\|_{H_a^{l+2m-m_{j\rho\mu}-1/2}(\Gamma_{j\rho})}\le
k_5\|u\|_{H_a^{l+2m}(Q_b)}\le k_6\|u\|_{W^{l+2m}(Q_b)}\\
\le k_7(\|A^0u\|_{W^l(Q_{b/2})}+\|u\|_{L_2(Q_{b/2})})\le
 k_8(\|A^0u\|_{H_a^l(Q)}+\|u\|_{H_a^{l+2m-1}(Q)}).
\end{multline}

Now let us estimate the norm of $\Psi_{j\rho\mu ks}^2$. By virtue
of~\eqref{eqVstavka4.7},
$$
\xi_q^t(\mathcal G_{j\rho ks}y',z')-\xi_q^t(y',z')=0,\qquad
 |y'|<\varepsilon_1'/\max(1,\chi_{j\rho ks}).
$$
Therefore, similarly to~\eqref{eqVstavka4.10}, we obtain
\begin{equation}\label{eqVstavka4.10'}
\|\Psi_{j\rho\mu ks}^2\|_{H_a^{l+2m-m_{j\rho\mu}-1/2}(\Gamma_{j\rho})}\le
 k_9(\|A^0u\|_{H_a^l(Q)}+\|u\|_{H_a^{l+2m-1}(Q)}).
\end{equation}

It follows from~\eqref{4.AprL^1_1}, \eqref{eqVstavka4.10}, and~\eqref{eqVstavka4.10'}
that
\begin{equation}\label{4.AprL^1_2}
\|\xi_q^tu\|_{H_a^{l+2m}(Q)}\le k_{10}(\|\mathbf L^1 u\|_{\mathcal
H_a^l(Q,\Gamma)}+\|u\|_{H_a^{l+2m-1}(Q)}).
\end{equation}

Setting
\begin{equation}\label{4.AprL^1_2'}
\xi_0(x)=\sum\limits_{t,q}\xi_q^t(x)
\end{equation}
and using inequality~\eqref{4.AprL^1_2}, we have
\begin{equation}\label{eqVstavka4.12}
\|\xi_0u\|_{H_a^{l+2m}(Q)}\le k_{11}(\|\mathbf L^1 u\|_{\mathcal
H_a^l(Q,\Gamma)}+\|u\|_{H_a^{l+2m-1}(Q)})
\end{equation}
(note that $k_{11}$ depends on $\varepsilon$).

3. Using a partition of unity, Theorem~\ref{t2.R^n_+Isomorphism}, Leibniz' formula, and a
priori estimates of solutions for elliptic problems in the interior of $Q$ and near a
smooth part of the boundary, we obtain
\begin{multline}\label{eqVstavka4.13}
\|(1-\xi_0)u\|_{H_a^{l+2m}(Q)}\le c_1(\|(1-\xi_0)\mathbf
L^0u\|_{\mathcal H_a^l(Q,\Gamma)}+\|u\|_{H_a^{l+2m-1}(Q)})\\
\le c_2\Big(\|(1-\xi_0)\mathbf L^1u\|_{\mathcal
H_a^l(Q,\Gamma)}+\sum\limits_{i,\mu}\|(1-\xi_0)B_{i\mu}^1u\|_{H_a^{l+2m-m_{i\mu}-1/2}(\Gamma_i)}+\|u\|_{H_a^{l+2m-1}(Q)}\Big),
\end{multline}
where $c_1,c_2,{\dots}>0$ do not depend on $u$ and $\varepsilon$ (we recall that the
function $\xi_0$ does not depend on $\varepsilon$).

It follows from~\eqref{eqVstavka4.8} and~\eqref{4.AprL^1_2'} that $1-\xi_0(x)=0$ for
$x\in\mathcal K_1$. On the other hand, $\supp\xi(\omega_{is}(x))\subset\mathcal
K_1^{\varepsilon''}$ due to~\eqref{4.xiOmega=0}. Therefore, applying Lemma~\ref{lzeta0v}
and Remark~\ref{rzeta0v} and taking into account that $\omega_{is}$ are $C^\infty$
diffeomorphisms, we have
$$
\|(1-\xi_0)B_{i\mu}^1u\|_{H_a^{l+2m-m_{i\mu}-1/2}(\Gamma_i)}\le
c_3\varepsilon''\|B_{i\mu}^1u\|_{H_a^{l+2m-m_{i\mu}-1/2}(\Gamma_i)}\le
c_4\varepsilon''\|\xi u\|_{H_a^{l+2m}(Q)}.
$$
Further, using the relation $\supp\xi\subset\mathcal K_1^{\varepsilon}$, inequalities
\eqref{4.1'}, and Lemma~\ref{lZetaDeltaU}, we obtain from the last estimate that
$$
\|(1-\xi_0)B_{i\mu}^1u\|_{H_a^{l+2m-m_{i\mu}-1/2}(\Gamma_i)} \le
c_5\varepsilon''\|u\|_{H_a^{l+2m}(Q)}.
$$
Combining this estimate with~\eqref{eqVstavka4.13} yields
\begin{equation}\label{eqVstavka4.15}
\|(1-\xi_0)u\|_{H_a^{l+2m}(Q)} \le c_6(\|\mathbf L^1u\|_{\mathcal
H_a^l(Q,\Gamma)}+\varepsilon''\|u\|_{H_a^{l+2m}(Q)}+\|u\|_{H_a^{l+2m-1}(Q)}).
\end{equation}
Choosing $\varepsilon$ in the definition of the function $\xi$ so small that
$c_6\varepsilon''=1/2$ and using inequalities~\eqref{eqVstavka4.12}
and~\eqref{eqVstavka4.15}, we complete the proof.
\end{proof}

\subsection{}\label{subsec4.3}
In this subsection, we prove Theorem~\ref{t4.1}. First, we formulate some results on
properties of weighted spaces, which are needed below.

\begin{lemma}\label{l4.2}
Let $Q_1\subset\mathbb R^n$ be a bounded domain such that $\overline{Q}\subset Q_1$.
Assume that the set $K_1$ in the definition of the space $H_a^k(Q_1)=H_a^k(Q_1,K_1)$
coincides with the set $K$ in the definition of the space $H_a^k(Q)=H_a^k(Q,K)$. Then,
for any function $v\in H_a^k(Q)$, there exists a function $v_1\in H_a^k(Q_1)$ such that
$v_1(x)=v(x)$ for $x\in Q$ and
\begin{equation*}
\|v_1\|_{H_a^k(Q_1)}\le c\|v\|_{H_a^k(Q)},
\end{equation*}
where $c>0$ does not depend on $v$.
\end{lemma}
Lemma~\ref{l4.2} is proved in~\cite[Sec.~3]{SkJMAA}.

\begin{lemma}\label{l4.3}
Let $a>l+2m-1$. Assume that $\delta>0$ satisfies the following conditions$:$
\begin{enumerate}
\item
$\mathcal K_1^\delta\cap(\mathcal K_2\cup\mathcal K_3)=\varnothing$ if $\mathcal
K_2\cup\mathcal K_3\ne\varnothing$,
\item
$\mathcal K_2^{\delta}\cap(\mathcal K\setminus\mathcal K_2)=\varnothing$ if $\mathcal
K_2\ne\varnothing$,
\item
$\mathcal K_3^{\delta}\subset Q$ and $\mathcal K_3^{\delta}\cap(\mathcal
K\setminus\mathcal K_3)=\varnothing$ if $\mathcal K_3\ne\varnothing$,
\item
$\delta>0$ is arbitrary if $\mathcal K_2\cup\mathcal K_3=\varnothing$.
\end{enumerate}
Then
\begin{equation}\label{4.5}
\|u\|_{H_a^{l+2m}(\mathcal K_j^{\delta}\cap Q)}\le c_1\|u\|_{W^{l+2m}(\mathcal
K_j^{\delta}\cap Q)}
\end{equation}
for all $u\in W^{l+2m}(\mathcal K_j^{\delta}\cap Q)$ if $\mathcal K_j\ne\varnothing$
$(j=1,2)$ and
\begin{equation}\label{4.4}
\|u\|_{H_a^{l+2m}(\mathcal K_3^{\delta})}\le c_2\|u\|_{W^{l+2m}(\mathcal K_3^{\delta})}
\end{equation}
for all $u\in W^{l+2m}(\mathcal K_3^{\delta})$ if $\mathcal K_3\ne\varnothing$, where
$c_1,c_2>0$ do not depend on $u$.
\end{lemma}
Lemma~\ref{l4.3} is proved in~\cite{SkJMAA} (see also Lemma~5.2 in~\cite{KovSk}).

The following result is also obtained in~\cite[Sec.~3]{SkJMAA}. It means that the
operators $B_{i\mu}^2$ correspond to nonlocal terms supported outside the set $\mathcal
K_1$. For the reader's convenience, we give the proof of this result.

\begin{lemma}\label{l4.6}
Let Condition~$\ref{cond4.1}$ hold. Then there exists a number
$\varkappa=\varkappa(\varepsilon)>0$ such that
\begin{equation}\label{4.15}
\|B_{i\mu}^2u\|_{H_a^{l+2m-m_{i\mu}-1/2}(\Gamma_i)}\le
c_1\|u\|_{H_a^{l+2m}(Q\setminus\overline{\mathcal K_1^{2\varkappa}})}
\end{equation}
for all $u\in H_a^{l+2m}(Q\setminus\overline{\mathcal K_1^{2\varkappa}});$ furthermore,
there exists a number $\sigma=\sigma(\varkappa)$ such that
\begin{equation}\label{4.16}
\|B_{i\mu}^2u\|_{H_a^{l+2m-m_{i\mu}-1/2}(\Gamma_i\setminus\overline{\mathcal
K_1^{\varkappa}})}\le c_2\|u\|_{H_a^{l+2m}(Q_\sigma)}
\end{equation}
for all $u\in H_a^{l+2m}(Q_\sigma);$ here $i=1,\dots,N_0;$ $\mu=1,\dots,m;$ $c_1,c_2>0$
do not depend on $u$.
\end{lemma}
\begin{proof}
1. It suffices to show that inequalities~\eqref{4.15} and~\eqref{4.16} are valid for the
function
$$
\varphi_{i\mu s}=(B_{i\mu s}^0(x,D)((1-\xi) u))(\omega_{is}(x))|_{\Gamma_i}
$$
substituted for $B_{i\mu}^2u$.

2. Let $\overline{\omega_{is}(\Gamma_i)}\cap\mathcal K_2\ne\varnothing$. We assume
without loss of generality that $\overline{\omega_{is}(\Omega_i)}\cap\mathcal
K_2=\overline{\omega_{is}(\Gamma_i)}\cap\mathcal K_2$, see Fig.~\ref{fig4.1}.
\begin{figure}[ht]
{ \hfill\epsfxsize120mm\epsfbox{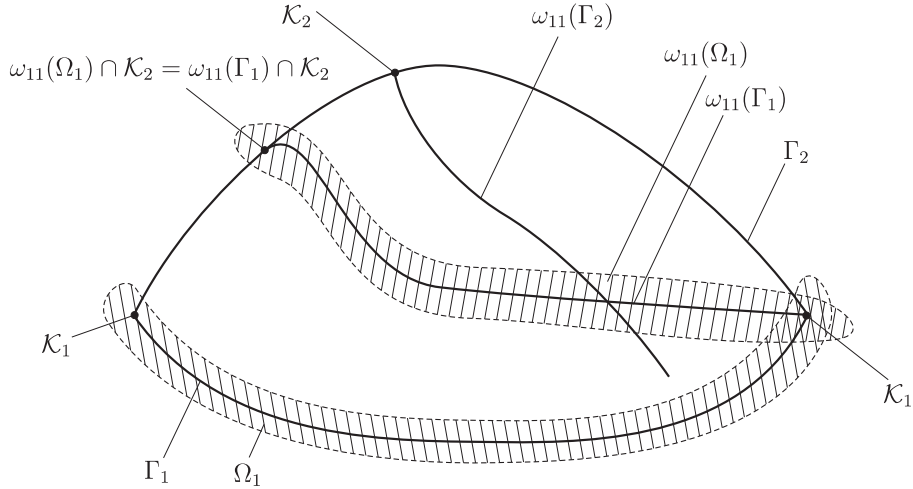}\hfill\ } \caption{The domain
$Q$}\label{fig4.1}
\end{figure}
Let $U$ be an extension of the function $(1-\xi)u$ to $Q\cup\omega_{is}(\Omega_i)$,
defined by Lemma~\ref{l4.2} and satisfying the inequality
\begin{equation}\label{eqVstavka4.16}
\|U\|_{H_a^{l+2m}(Q\cup\omega_{is}(\Omega_i))}\le k_1\|(1-\xi)u\|_{H_a^{l+2m}(Q)},
\end{equation}
where $k_1,k_2,{\dots}>0$ do not depend on  $u$.

Set
$$
\Phi_{i\mu s}(x)=(B_{i\mu s}^0(x,D)U)(\omega_{is}(x)),\qquad x\in\Omega_i.
$$
Clearly,
$$
\varphi_{i\mu s}=\Phi_{i\mu s}|_{\Gamma_i}.
$$
Introducing the new variable $\hat x=\omega_{is}(x)$, applying Lemma~\ref{l4.3} if
$\mathcal K_j\ne\varnothing$ and $\mathcal K_j\not\subset K$, $j=2,3$, and using
inequality~\eqref{eqVstavka4.16}, we obtain
\begin{multline}\label{4.18}
\|\varphi_{i\mu s}\|_{H_a^{l+2m-m_{i\mu}-1/2}(\Gamma_i)}\le\|\Phi_{i\mu
s}\|_{H_a^{l+2m-m_{i\mu}}(\Omega_i)}\\
\le k_2\|B_{i\mu s}^0(\hat x,D)U(\hat
x)\|_{H_a^{l+2m-m_{i\mu}}(\omega_{is}(\Omega_i)))}\le k_3\|(1-\xi)u\|_{H_a^{l+2m}(Q)}.
\end{multline}
Thus, setting $2\varkappa=\varepsilon/2$, we see that~\eqref{4.18} implies~\eqref{4.15}.

Since the transformation $\omega_{is}$ is continuous and $\omega_{is}(\Gamma_i)\subset
Q$, it follows that $\omega_{is}(\Gamma_i\setminus\overline{\mathcal
K_1^\varkappa})\subset Q_{2\sigma}$ for sufficiently small $\sigma>0$. Introduce a
function $\eta\in C_0^\infty(\mathbb R^n)$ such that $\eta(x)=1$ for $x\in Q_{2\sigma}$
and $\eta(x)=0$ for $x\notin Q_{\sigma}$.

Suppose that the function $\eta u$ is extended  by zero outside $Q$. Set
$$
\Psi_{i\mu s}(x)=(B_{i\mu s}^0(x,D)(\eta(1-\xi) u))(\omega_{is}(x)),\qquad x\in\Omega_i.
$$
It is clear that
$$
\varphi_{i\mu s}|_{\Gamma_i\setminus\overline{\mathcal
K_1^\varkappa}}=\Psi|_{\Gamma_i\setminus\overline{\mathcal K_1^\varkappa}}.
$$
Hence, applying Lemma~\ref{l4.3} if $\mathcal K_j\ne\varnothing$ and $\mathcal
K_j\not\subset K$, $j=2,3$,  we obtain
\begin{multline*}
\|\varphi_{i\mu s}\|_{H_a^{l+2m-m_{i\mu}-1/2}(\Gamma_i\setminus\overline{\mathcal
K_1^\varkappa})}\le\|\Psi_{i\mu
s}\|_{H_a^{l+2m-m_{i\mu}}(\Omega_i)}\\
\le k_4\|B_{i\mu s}^0(\hat x,D)(\eta(1-\xi) u)(\hat
x)\|_{H_a^{l+2m-m_{i\mu}}(\omega_{is}(\Omega_i)))}\le k_5\|u\|_{H_a^{l+2m}(Q_\sigma)}.
\end{multline*}

3. If $\overline{\omega_{is}(\Gamma_i)}\cap\mathcal K_2=\varnothing$, then
$\overline{\omega_{is}(\Gamma_i)}\setminus\mathcal K_1^{\varepsilon/2}\subset Q$.
Therefore, similarly to the above, we obtain
$$
\|\varphi_{i\mu s}\|_{H_a^{l+2m-m_{i\mu}-1/2}(\Gamma_i)}\le k_5
\|u\|_{H_a^{l+2m}(Q_\sigma)}.
$$
This proves inequalities~\eqref{4.15} and~\eqref{4.16}.
\end{proof}

\begin{proof}[Proof of Theorem~$\ref{t4.1}$]
1. Take an arbitrary function $u\in H_a^{l+2m}(Q)$. It follows from Lemma~\ref{l4.4} that
\begin{equation}\label{4.22}
\|u\|_{H_a^{l+2m}(Q)}\le k_1\Big(\|\mathbf Lu\|_{\mathcal
H_a^l(Q,\Gamma)}+\|A^1u\|_{H_a^l(Q)}+\sum\limits_{i,\mu}\sum\limits_{k=2,3}\|B_{i\mu}^k
u\|_{H_a^{l+2m-m_{i\mu}-1/2}(\Gamma_i)} \Big),
\end{equation}
where $k_1,k_2,\dots{>0}$ do not depend on $u$.

It follows from the boundedness of the domain $Q$ and from Lemma~\ref{l4.3} that
\begin{equation}\label{4.23}
\|A^1u\|_{H_a^l(Q)}+\sum\limits_{i,\mu}\|B_{i\mu}^3
u\|_{H_a^{l+2m-m_{i\mu}-1/2}(\Gamma_i)}\le k_2\|u\|_{H_a^{l+2m-1}(Q)}.
\end{equation}

2. Consider a function $\eta\in C^\infty(\mathbb R^n)$ such that
$$
\eta(x)=1\quad\text{for}\quad x\in\mathbb R^n\setminus\overline{\mathcal
K_1^{2\varkappa}},\qquad \eta(x)=0\quad\text{for}\quad x\in\mathcal K_1^{\varkappa},
$$
where $\varkappa>0$ is the constant occurring in Lemma~\ref{l4.6}.

It follows from inequality~\eqref{4.15}, from Lemma~\ref{l4.4}, and from Leibniz' formula
that
\begin{multline}\label{4.26}
\|B_{i\mu}^2 u\|_{H_a^{l+2m-m_{i\mu}-1/2}(\Gamma_i)}\le k_3\|\eta
u\|_{H_a^{l+2m}(Q)}\\
\le k_4\Big(\|\eta \mathbf L^1u\|_{\mathcal H_a^l(Q,\Gamma)}+\|u\|_{H_a^{l+2m-1}(Q)}+
\sum\limits_{i,\mu}\sum\limits_{s\ne0}\|\Psi_{i\mu
s}\|_{H_a^{l+2m-m_{i\mu}-1/2}(\omega_{is}(\Gamma_i))}\Big),
\end{multline}
where
$$
\Psi_{i\mu s}=\big(\eta(x)-\eta(\omega_{is}^{-1}(x))\big)(B_{i\mu s}^0(x,D)(\xi
u))(x)|_{\omega_{is}(\Gamma_i)}.
$$
It is clear that, if $\overline{\omega_{is}(\Gamma_i)}\cap\mathcal
K_1^\varepsilon=\varnothing$, then $\Psi_{i\mu s}=0$. Let
$\overline{\omega_{is}(\Gamma_i)}\cap\mathcal K_1^\varepsilon\ne\varnothing$. We claim
that
\begin{equation}\label{4.Psi_iMus=0}
\supp\Psi_{i\mu s}\subset Q_b
\end{equation}
for some $b>0$. Indeed,
\begin{equation}\label{4.omegaANDxi}
\omega_{is}(\Gamma_i)\subset Q\qquad\text{and}\qquad \supp\xi\subset\overline{\mathcal
K_1^\varepsilon}.
\end{equation}
Therefore, by virtue of~\eqref{eqStar}, it suffices to show that $\Psi_{i\mu s}(x)=0$ for
$x$ in some neighborhood of $\mathcal K_1$. Since
\begin{equation}\label{4.eta=0}
\eta(x)=0,\qquad x\in\overline{\omega_{is}(\Gamma_i)}\cap\mathcal K_1^\varkappa,
\end{equation}
it remains to prove that
\begin{equation}\label{4.etaOmega-1=0}
\eta(\omega_{is}^{-1}(x))=0,\qquad x\in\overline{\omega_{is}(\Gamma_i)}\cap\mathcal
K_1^d,
\end{equation}
for a sufficiently small $d>0$. Note that, if
$\overline{\omega_{is}(\Gamma_i)}\cap\mathcal K_1=\varnothing$, then~\eqref{4.Psi_iMus=0}
follows from~\eqref{eqStar} and~\eqref{4.omegaANDxi}. If
$\overline{\omega_{is}(\Gamma_i)}\cap\mathcal K_1\ne\varnothing$ for some $i$ and $s$,
then $\mathcal K_{1\nu}\subset\overline{\omega_{is}(\Gamma_i)}$ for some $\nu$ and
$\omega_{is}^{-1}(\mathcal K_{1\nu})\subset\mathcal K_1$. Hence, there exists a
sufficiently small $d>0$ such that $\mathcal K_{1\nu}^d\subset\omega_{is}(\Omega_i)$ and
$\omega_{is}^{-1}(\mathcal K_{1\nu}^d)\subset\mathcal K_1^\varkappa$ (because the
transformations $\omega_{is}^{-1}$ are smooth). Clearly, \eqref{4.etaOmega-1=0} holds in
this case. Thus, we obtain~\eqref{4.Psi_iMus=0}.

It follows from~\eqref{4.Psi_iMus=0} and Lemma~\ref{l3.10} that
$$
\|\Psi_{i\mu s}\|_{H_a^{l+2m-m_{i\mu}-1/2}(\omega_{is}(\Gamma_i))}\le
k_5(\|Au\|_{H_a^l(Q)}+\|u\|_{H_a^{l+2m-1}(Q)}).
$$
Using this inequality, we infer from~\eqref{4.26} that
\begin{multline}\label{4.26'}
\|B_{i\mu}^2 u\|_{H_a^{l+2m-m_{i\mu}-1/2}(\Gamma_i)}\\
\le k_6\Big(\| \mathbf L u\|_{\mathcal H_a^l(Q,\Gamma)}+\|u\|_{H_a^{l+2m-1}(Q)}+
\sum\limits_{i,\mu}\|\eta B_{i\mu}^2u\|_{H_a^{l+2m-m_{i\mu}-1/2}(\Gamma_i)}\Big).
\end{multline}

By virtue of~\eqref{4.16}, Lemma~\ref{l3.10}, and Leibniz' formula, we have
\begin{equation}\label{4.27}
\|\eta B_{i\mu}^2u\|_{H_a^{l+2m-m_{i\mu}-1/2}(\Gamma_i)}\le
k_7\|u\|_{H_a^{l+2m}(Q_\sigma)} \le k_8(\|Au\|_{H_a^l(Q)}+\|u\|_{H_a^{l+2m-1}(Q)}).
\end{equation}

Combining estimates~\eqref{4.22}, \eqref{4.23}, \eqref{4.26'}, and \eqref{4.27}, we
obtain the desired estimate~\eqref{4.1}.
\end{proof}

\section{The Fredholm Property of Nonlocal Elliptic Problems}\label{sec5}

\subsection{}

In this section, we prove the main result of the paper concerning the Fredholm property
of nonlocal elliptic problems in weighted spaces. This result can be formulated as
follows.

\begin{theorem}\label{t5.1}
Let Conditions~$\ref{cond1.1}$--$\ref{cond1.4}$ and~$\ref{cond4.1}$ hold. Assume that the
line $\Im\lambda=a+1-l-2m$ contains no eigenvalues of $\hat{\mathcal L}_g(\lambda)$ for
any $g\in K$ and $\dim\mathcal N(\mathcal L_g(\omega))=\codim\mathcal R(\mathcal
L_g(\omega))=0$ for any $g\in K$ and $\omega\in S^{n-3}$. Then the operator $\mathbf L:
H_a^{l+2m}(Q)\to \mathcal H_a^l(Q,\Gamma)$ has the Fredholm property.
\end{theorem}

Due to Theorem 16.4 in~\cite{Kr} (about compact perturbations of Fredholm operators), it
suffices to prove Theorem~\ref{t5.1} and the other assertions of this section for $A^1=0$
and $B^3=0$. Therefore, we assume that $A^1=0$ and $B^3=0$ throughout this section.
\begin{corollary}\label{c5.1}
Let the conditions of Theorem~$\ref{t5.1}$ be fulfilled. Then $\ind \mathbf L=\ind
\mathbf L^1$.
\end{corollary}
\begin{proof}
We introduce the operator
$$
 L_tu=\{A^0u,\ {C}u+(1-t){B}^2u\}.
$$
We have $L_0={\bf L}$ (because $A^1=0$ and $B^3=0$) and $L_1={\bf L}^1$.

By Theorem~\ref{t5.1}, the operators $L_t$ have the Fredholm property for all $t$.
Furthermore, for any $t_0$ and $t$, the following estimate holds:
$$
\|L_tu- L_{t_0}u\|_{{\mathcal H}_a^l(Q,\Gamma)}\le k_{t_0}|t-t_0|\cdot
\|u\|_{H_a^{l+2m}(Q)},
$$
where $k_{t_0}>0$ does not depend on $t$. Therefore, by Theorem~16.2 in~\cite{Kr}, we
have $\ind L_t=\ind L_{t_0}$ for any $t$ in a sufficiently small neighborhood of the
point $t_0$. Since $t_0$ is arbitrary, these neighborhoods cover the segment $[0, 1]$.
Choosing a finite subcovering, we obtain the relations $ \ind{\bf L}=\ind L_0=\ind
L_1=\ind{\bf L}^1. $
\end{proof}

The proof of Theorem~$\ref{t5.1}$ is based on the existence of a right regularizer for
the operator $\mathbf L$.
\begin{theorem}\label{t5.2}
Let the conditions of Theorem~$\ref{t5.1}$ be fulfilled. Then there exists a linear
bounded operator $\mathbf R:\mathcal H_a^l(Q,\Gamma)\to H_a^{l+2m}(Q)$ such that
$$
\mathbf L\mathbf R=\mathbf I+\mathbf T
$$
where $\mathbf I$ and $\mathbf T$ are  the identity operator and a compact operator on
$\mathcal H_a^l(Q,\Gamma)$, respectively.
\end{theorem}

\begin{proof}[Proof of Theorem~$\ref{t5.1}$]
Assume that Theorem~\ref{t5.2} is true. By Lemma~3.5 in~\cite{Kondr}, the embedding of
$H_a^{l+2m}(Q)$ into $H_{a}^{l+2m-1}(Q)$ is compact. Therefore, by Theorem~7.1
in~\cite{Kr} and Theorem~\ref{t4.1}, $\dim\mathcal N(\mathbf L)<\infty$ and the range
$\mathcal R(\mathbf L)$ is closed in $\mathcal H_a^l(Q,\Gamma)$. On the other hand,
Theorem~15.2 in~\cite{Kr} and Theorem~\ref{t5.2} imply that $\codim\mathcal R(\mathbf
L)<\infty$.
\end{proof}

Thus, it remains to prove Theorem~\ref{t5.2}.

\subsection{}

First, we prove the following auxiliary result.
\begin{lemma}\label{lI+T+M}
Let $H$ be a Hilbert space and $I$ the identity operator on $H$. Let $M_\varepsilon$ and
$S_\varepsilon$, $\varepsilon>0$, be families of bounded operators on $H$ such that
\begin{equation}\label{eqI+T+M}
\|M_\varepsilon\|\le c_1\varepsilon,\qquad \|S_\varepsilon\|\le c_2,
\end{equation}
where $c_1,c_2>0$ do not depend on $\varepsilon$, and the operators $S_\varepsilon^2$ are
compact. Then the operators
$$
L_\varepsilon=I+M_\varepsilon+S_\varepsilon
$$
 have the Fredholm
property for sufficiently small $\varepsilon>0$.
\end{lemma}
\begin{proof}
To prove the lemma, we will construct a right and a left regularizers for
$L_\varepsilon$. We have
$$
L_\varepsilon(I-(M_\varepsilon+S_\varepsilon))=I-M_\varepsilon^2-M_\varepsilon
S_\varepsilon-S_\varepsilon M_\varepsilon-S_\varepsilon^2.
$$
It follows from~\eqref{eqI+T+M} that
$$
\|M_\varepsilon^2+M_\varepsilon S_\varepsilon+S_\varepsilon M_\varepsilon\|\le
c_3\varepsilon,
$$
where $c_3>0$ does not depend on $\varepsilon$. Therefore, the operators
$I-M_\varepsilon^2-M_\varepsilon S_\varepsilon-S_\varepsilon M_\varepsilon$ have the
Fredholm property by Theorem~16.2 in~\cite{Kr}, provided that $\varepsilon>0$ is
sufficiently small. Further, using the fact that the operators $S_\varepsilon^2$ are
compact and applying Theorem~16.4 in~\cite{Kr}, we see that the operators
$L_\varepsilon(I-(M_\varepsilon+S_\varepsilon))$ also have the Fredholm property. Now it
follows from Theorem~15.2 in~\cite{Kr} that there exist bounded operators
$R_{1\varepsilon}$ and compact operators $T_{1\varepsilon}$ such that
\begin{equation}\label{eqI+T+M1}
L_\varepsilon(I-(M_\varepsilon+S_\varepsilon))R_{1\varepsilon}=I+T_{1\varepsilon}.
\end{equation}

Similarly, one can prove that there exist bounded operators $R_{2\varepsilon}$ and
compact operators $T_{2\varepsilon}$ such that
\begin{equation}\label{eqI+T+M2}
R_{2\varepsilon}(I-(M_\varepsilon+S_\varepsilon))L_\varepsilon=I+T_{2\varepsilon}.
\end{equation}

The conclusion of the lemma follows from relations~\eqref{eqI+T+M1} and~\eqref{eqI+T+M2}
and from Theorems~15.2 and~14.3 in~\cite{Kr}.
\end{proof}

To prove Theorem~\ref{t5.2}, we preliminarily consider the operator $\mathbf L^1$, i.e.,
assume that nonlocal terms are supported near the set $\mathcal K_1$.

\begin{lemma}\label{l5.1}
Let the conditions of Theorem~$\ref{t5.1}$ be fulfilled and the number $\varepsilon$ be
sufficiently small. Then there exist a linear bounded operator $\mathbf R_1:\mathcal
H_a^l(Q,\Gamma)\to H_a^{l+2m}(Q)$ and a compact operator $\mathbf T_1: \mathcal
H_a^l(Q,\Gamma)\to \mathcal H_a^l(Q,\Gamma)$ such that
$$
\mathbf L^1\mathbf R_1=\mathbf I+\mathbf T_1.
$$
\end{lemma}
\begin{proof}
1. To construct a right regularizer, we consider a partition of unity $\{\xi_j^t\}$
different from that in the proof of Lemma~\ref{l4.4}.

For each orbit $\mathcal O(g)$, $g\in\mathcal K_1$, we denote by $\hat B_\delta(g_j)$ the
same neighborhoods as in the proof of Lemma~\ref{l4.4}, and let $\{\varphi_j^t\}$ be the
same partition of unity for $\mathcal K_1$. We denote the function $\varphi_j^t(x)$
written in the variables $x'=(y',z')$ by $\varphi_j^t(y',z')$. Clearly, there is a number
$a'<\chi_m\delta(g^t)$ such that
$$
\varphi_j^t(0,z')=0
$$
for $a'<|z'|<\chi_m\delta(g^t)$. As before, we assume that $\varphi_j^t(0,z')$ is
extended by zero for $|z'|\ge \chi_m\delta(g^t)$, and it remains to be infinitely
differentiable.

Let $\hat\varphi^t\in C_0^\infty(\mathbb R^{n-2})$ be a function such that
$$
\hat\varphi^t(z')=1,\qquad |z'|< a',
$$
$$
\hat\varphi^t(z')=0,\qquad |z'|> {\hat a}',
$$
where $a'<{\hat a}'<\chi_m\delta(g^t)$.

Denote by $\psi^t,\hat\psi^t\in C_0^\infty(\mathbb R^2)$ functions such that
$\psi^t(y')=1$ for $|y'|\le\varepsilon_1'$ and $\psi^t(y')=0$ for $|y'|\ge
3\varepsilon_1'/2$, $\hat\psi^t(y')=1$ for $|y'|\le3\varepsilon_1'/2$, and
$\hat\psi^t(y')=0$ for $|y'|\ge 2\varepsilon_1'$, where $\varepsilon_1'>0$ is so small
that
\begin{equation}\label{eqVstavka5.5}
\left\{(2\varepsilon_1')^2+({\hat a}')^2\right\}^{1/2}<\chi_m\delta(g^t),
\end{equation}
and $\varepsilon_1'$ does not depend on $\varepsilon$.

Set
$$
\xi_j^t(y',z')=\psi^t(y')\varphi_j^t(0,z'),\qquad
\hat\xi^t(y',z')=\hat\psi^t(y')\hat\varphi^t(z').
$$
By virtue of~\eqref{eqVstavka5.5}, we have
\begin{equation}\label{eqVstavka5.6}
\supp \xi_j^t(y',z')\subset B_{\chi_m\delta(g^t)},\qquad \supp \hat\xi^t(y',z')\subset
B_{\chi_m\delta(g^t)}.
\end{equation}
It is also clear that
\begin{equation}\label{eqVstavka5.7}
\xi_j^t(y',z')=\varphi_j^t(0,z'),\quad \hat\xi^t(y',z')=\hat\varphi^t(z'),\qquad
|y'|<\varepsilon_1',
\end{equation}
\begin{equation}\label{eqVstavka5.8}
\hat\xi^t(y',z')\xi_j^t(y',z')=\xi_j^t(y',z'),\qquad (y',z')\in\mathbb R^n.
\end{equation}

Denote by $\xi_j^t(x)$ and $\hat\xi_j^t(x)$ the functions $\xi_j^t(y',z')$ and
$\hat\xi^t(y',z')$, respectively, written in the variables $x=x(g^t,j)$. Since $\supp
\xi_j^t(x)\subset {\hat B}_{\delta(g^t)}(g_j^t)$ and $\supp \hat\xi_j^t(x)\subset {\hat
B}_{\delta(g^t)}(g_j^t)$, we can extend these functions by zero outside the neighborhood
${\hat B}_{\delta(g^t)}(g_j^t)$ to obtain the functions infinitely differentiable on
$\mathbb R^n$.

Obviously, we have
\begin{equation}\label{eqVstavka5.9}
\sum\limits_{j,t}\xi_j^t(x)=\sum\limits_{j,t}\varphi_j^t(x)=1,\qquad x\in\mathcal K_1.
\end{equation}

2. We set
$$
\begin{aligned}
M_H^t&=\left\{u\in H_a^{l+2m}(Q):\ \supp u\subset\bigcup_{j=1}^{N(g^t)}
V(g_j^t)\right\},\\
\mathcal M_H^t&=\{v\in\mathcal H_a^{l+2m}(\Theta):\ \supp v\subset V(0)\}.
\end{aligned}
$$
For all $u\in M_H^t$, denote $u_j^t(x)=u(x)$, $x\in V(g_j^t)$. We define the isomorphism
$U^t: M_H^t\to\mathcal M_H^t$ by the formulas
$$
(U^tu)_j(x')=u_j^t(x(x')),\ x'\in\Theta_j\cap V(0);\qquad (U^tu)_j(x')=0,\
x'\in\Theta_j\setminus V(0);
$$
$j=1,\dots,N(g^t)$.

We set
$$
\begin{aligned}
G_H^t&=\left\{f\in \mathcal H_a^{l}(Q,\Gamma):\ \supp f\subset\bigcup_{j=1}^{N(g^t)}
V(g_j^t)\right\},\\ \mathcal
 G_H^t&=\{\Phi\in\mathcal H_a^{l}(\Theta,\Gamma):\ \supp
\Phi\subset V(0)\}.
\end{aligned}
$$
If $f=\{f_0,f_{i\mu}\}\in G_H^t$, then we denote by $f_j^t(x)$ and $f_{j\rho\mu}^t(x)$
the functions $f_0(x)$ and $f_{i\mu}(x)$ for $x\in V(g_j^t)$ and $x\in\Gamma_i\cap
V(g_j^t)$, respectively, where $\rho$ and $j$ are such that the transformation $x\mapsto
x'(g^t,j)$ maps $\Gamma_i\cap V(g_j^t)$ onto $\Gamma_{j\rho}\cap V(0)$. Further, we
define the isomorphism $F^t: G_H^t\to\mathcal G_H^t$ by the formula
$$
(F^tf)(x')=\{(F^tf)_j(x'), (F^tf)_{j\rho\mu}(x')\}.
$$
Here
$$
(F^tf)_j(x')=f_j^t(x(x')),\ x'\in\Theta_j\cap V(0);\qquad (F^tf)_j(x')=0,\
x'\in\Theta_j\setminus V(0);
$$
$$
(F^tf)_{j\rho\mu}(x')=f_{j\rho\mu}^t(x(x')),\ x'\in\Gamma_{j\rho}\cap V(0);\qquad
(F^tf)_{j\rho\mu}(x')=0,\ x'\in\Gamma_{j\rho}\setminus V(0);
$$
$j=1,\dots,N(g^t)$; $\rho=1,2$; $\mu=1,\dots,m$.

3. We set
\begin{equation}\label{eqVstavka5.10}
R_{\mathcal K_1}f=\sum\limits_t(U^t)^{-1}\Big(\hat\xi^t(\mathcal
L_{g^t}'')^{-1}F^t\Big(\sum\limits_q\xi_q^tf\Big)\Big).
\end{equation}
Since the functions $\hat\xi^t$ and $\xi_q^t$ do not depend on $\varepsilon$, it follows
from Corollary~\ref{cort2.2'} that
\begin{equation}\label{eqVstavka5.10'}
\|R_{\mathcal K_1}f\|_{H_a^{l+2m}(Q)}\le c_1\|f\|_{\mathcal H_a^l(Q,\Gamma)},
\end{equation}
where $c_1,c_2,{\dots}>0$ depend neither on $\varepsilon$ nor on a function occurring in
the right-hand side.

Clearly, we have
\begin{equation}\label{eqVstavka5.11}
\mathbf L^1 R_{\mathcal K_1}f=\mathbf L R_{\mathcal K_1}f+T_1f,
\end{equation}
where $T_1:\mathcal H_a^l(Q,\Gamma)\to\mathcal H_a^l(Q,\Gamma)$ is a linear bounded
operator given by
$$
T_1 f=\Big\{0,-\sum\limits_{s=1}^{S_i}\big(B_{i\mu s}^0(x,D)((1-\xi)R_{\mathcal
K_1}f)\big)(\omega_{is}(x))|_{\Gamma_i}\Big\}
$$
(recall that $A^1=0$ and $B^3=0$). Moreover, using Lemmas~\ref{l4.6}
and~\ref{lZetaDeltaU} and estimate~\eqref{eqVstavka5.10'}, we have
\begin{equation}\label{eqVstavka5.11'}
\|T_1f\|_{\mathcal H_a^l(Q,\Gamma)}\le c_2(\|R_{\mathcal K_1}f\|_{H_a^{l+2m}(Q)}+\|\xi
R_{\mathcal K_1}f\|_{H_a^{l+2m}(Q)})\le c_3\|f\|_{\mathcal H_a^l(Q,\Gamma)}.
\end{equation}

Further, by virtue of~\eqref{eqVstavka5.6}, we have
$$
\supp\hat\xi^t(y',z')\subset B_{\delta(g^t)}.
$$
Therefore,
\begin{equation}\label{eqVstavka5.12}
\hat\xi^t\mathcal L_{g^t}'v=\hat\xi^t\mathcal L_{g^t}''v,\qquad v\in\mathcal
H_a^{l+2m}(\Theta).
\end{equation}

It follows from Leibniz' formula that
\begin{equation}\label{eqVstavka5.13}
\mathcal L_{g^t}'(\hat\xi^t v)=\hat\xi^t\mathcal L_{g^t}'v+\tilde{\mathcal
L}_{g^t}v+\{0,\mathcal T_{j\rho\mu}^t v\},
\end{equation}
where $\tilde{\mathcal L}_{g^t}:\mathcal H_a^{l+2m-1}(\Theta)\to \mathcal
H_a^l(\Theta,\Gamma)$ is a bounded operator such that $\supp\tilde{\mathcal
L}_{g^t}v\subset B_{\delta(g^t)}$, while
$$
\mathcal T_{j\rho\mu}^t v=\sum\limits_{(k,s)\ne(j,0)}\big(\hat\xi^t(\mathcal G_{j\rho
ks}y',z')-\hat\xi^t(y',z')\big)\big(B_{j\rho\mu ks}^0(x',D)v_k\big)(\mathcal G_{j\rho
ks}y',z')|_{\Gamma_{j\rho}}
$$
and $\supp \mathcal T_{j\rho\mu}^t v\subset B_{\delta(g^t)}$. Clearly, $$\mathcal
T_{j\rho\mu}^t:\mathcal H_a^{l+2m}(\Theta)\to
H_a^{l+2m-m_{j\rho\mu}-1/2}(\Gamma_{j\rho})$$ is a bounded operator. Moreover, since the
function $\hat\xi^t$ does not depend on $\varepsilon$, it follows that
\begin{equation}\label{eqVstavka5.13''}
\|\tilde{\mathcal L}_{g^t} v\|_{\mathcal H_a^l(\Theta,\Gamma)}\le c_4\|v\|_{\mathcal
H_a^{l+2m-1}(\Theta)},
\end{equation}
\begin{equation}\label{eqVstavka5.13'}
\|\mathcal T_{j\rho\mu}^t v\|_{H_a^{l+2m-m_{j\rho\mu}-1/2}(\Gamma_{j\rho})}\le
c_5\|v\|_{\mathcal H_a^{l+2m}(\Theta)}.
\end{equation}

Using definition~\eqref{eqVstavka5.10} of the operator $R_{\mathcal K_1}$, the
isomorphisms $U^t$ and $F^t$, and relations~\eqref{eqVstavka5.13}, \eqref{eqVstavka5.12},
and~\eqref{eqVstavka5.8}, we obtain
\begin{equation}\label{eqVstavka5.14}
\begin{aligned}
\mathbf L R_{\mathcal K_1}f&=\sum\limits_t(F^t)^{-1}\mathcal L_{g^t}'\hat\xi^t(\mathcal
L_{g^t}'')^{-1}F^t\Big(\sum\limits_q\xi_q^tf\Big)\\
&=\sum\limits_t(F^t)^{-1}\hat\xi^t
F^t\Big(\sum\limits_q\xi_q^tf\Big)+T_2f+T_3f\\
&=\sum\limits_t(F^t)^{-1}
F^t\Big(\sum\limits_q\hat\xi^t_q\xi_q^tf\Big)+T_2f+T_3f\\
&=\sum\limits_{t,q}\xi_q^t f+T_2f+T_3f,
\end{aligned}
\end{equation}
where
$$
T_2f=\sum\limits_t(F^t)^{-1}\tilde{\mathcal L}_{g^t}(\mathcal
L_{g^t}'')^{-1}F^t\Big(\sum\limits_q\xi_q^tf\Big),
$$
$$
T_3f=\sum\limits_t(F^t)^{-1}\Big\{0,\mathcal T_{j\rho\mu}^t(\mathcal
L_{g^t}'')^{-1}F^t\Big(\sum\limits_q\xi_q^tf\Big)\Big\}.
$$

Since the operator $\tilde{\mathcal L}_{g^t}:\mathcal H_a^{l+2m-1}(\Theta)\to \mathcal
H_a^l(\Theta,\Gamma)$ is bounded, $\supp\tilde{\mathcal L}_{g^t}v\subset
B_{\delta(g^t)}$, and $\mathcal H_a^{l+2m}(\Theta\cap B_{\delta(g^t)})$ is compactly
embedded into $\mathcal H_a^{l+2m-1}(\Theta\cap B_{\delta(g^t)})$, it follows that the
operator
$$
T_2:\mathcal H_a^l(Q,\Gamma)\to \mathcal H_a^l(Q,\Gamma)
$$
is compact. Furthermore, by~\eqref{eqVstavka5.13''}, \eqref{eqVstavka5.13'}, and
Corollary~\ref{cort2.2'}, we have
\begin{equation}\label{eqVstavka5.14'}
\|T_i f\|_{\mathcal H_a^l(Q,\Gamma)}\le c_6\|f\|_{\mathcal H_a^l(Q,\Gamma)},\qquad i=2,3
\end{equation}
(we have also used the fact that the operator $\mathcal L_{g^t}''$ and the functions
$\xi_q^t$ do not depend on $\varepsilon$).

Now let us prove that the square of the operator $T_3$ is compact. Indeed,
\begin{equation}\label{eqVstavka5.15}
\|(T_3)^2 f\|_{\mathcal H_a^l(Q,\Gamma)}\le c_7\sum\limits_{t,j,\rho,\mu}\Big\|\mathcal
T_{j\rho\mu}^t(\mathcal L_{g^t}'')^{-1}F^t\Big(\sum\limits_q\xi_q^tT_3
f\Big)\Big\|_{H_a^{l+2m-m_{j\rho\mu}-1/2}(\Gamma_{j\rho})}.
\end{equation}
It follows from~\eqref{eqVstavka5.6} and~\eqref{eqVstavka5.7} that
$$\supp \big(\hat\xi^t(\mathcal G_{j\rho
ks}y',z')-\hat\xi^t(y',z')\big)\subset B_{\delta(g^t)}$$ and $$\hat\xi^t(\mathcal
G_{j\rho ks}y',z')-\hat\xi^t(y',z')=0$$ for $|y'|\le\varepsilon_1'/\chi_M$, where
$\chi_M=\max\limits_{j,\rho,k,s}\chi_{j\rho ks}$. Therefore, the condition
$d_{k1}<d_{j\rho}+\varphi_{j\rho ks}<d_{k2}$ for $(k,s)\ne (j,0)$ implies that
\begin{equation}\label{eqVstavka5.16}
\|\mathcal T_{j\rho\mu}^t v \|_{H_a^{l+2m-m_{j\rho\mu}-1/2}(\Gamma_{j\rho})}\le
c_8\sum\limits_{k=1}^{N(g^t)}\|v_k\|_{H_a^{l+2m}(\Omega_k^t)},
\end{equation}
where
$$
\Omega_k^t=\{x'=(y',z'):\ d_{k1}+d_0<\varphi<d_{k2}-d_0,\ |y'|>\varepsilon_1'/\chi_M,\
|x'|<\delta(g^t)\}
$$
and
\begin{equation}\label{eqd0}
d_0=\min_{j,\rho,k,s}\big(d_{j\rho}+\varphi_{j\rho ks}-d_{k1},\
d_{k2}-(d_{j\rho}+\varphi_{j\rho ks})\big)/2\quad ((k,s)\ne(j,0)).
\end{equation}

Using inequality~\eqref{eqVstavka5.16}, Lemma~\ref{l3.10}, and the equivalence of the
norms in subspaces of the spaces $H_a^l(\Theta_k)$ and $W^l(\Theta_k)$ consisting of
compactly supported functions vanishing near the edge $\mathcal P$, we obtain
\begin{multline*}
\Big\|\mathcal T_{j\rho\mu}^t(\mathcal L_{g^t}'')^{-1}F^t\Big(\sum\limits_q\xi_q^tT_3
f\Big)\Big\|_{H_a^{l+2m-m_{j\rho\mu}-1/2}(\Gamma_{j\rho})}
\\
\le c_9\sum\limits_k\Big(\Big\|A_k''\Big[(\mathcal
L_{g^t}'')^{-1}F^t\Big(\sum\limits_q\xi_q^tT_3
f\Big)\Big]_k\Big\|_{H_a^l(\Theta_k)}\\
+\Big\|\Big[(\mathcal L_{g^t}'')^{-1}F^t\Big(\sum\limits_q\xi_q^tT_3
f\Big)\Big]_k\Big\|_{H_{a}^0(\Theta_k\cap B_{2\delta(g^t)})}\Big),
\end{multline*}
where
$$
A_k''=A_k(D_{x'})+\eta(A_k^0(x',D_{x'})-A_k(D_{x'})).
$$
However, the first $N(g^t)$ components of the vector $F^t\Big(\sum\limits_q\xi_q^tT_3
f\Big)$ are equal to zero. Therefore,
\begin{equation}\label{eqStar44}
A_k''\Big[(\mathcal L_{g^t}'')^{-1}F^t\Big(\sum\limits_q\xi_q^tT_3 f\Big)\Big]_k=0,\qquad
k=1,\dots,N(g^t),
\end{equation}
and hence
\begin{multline}\label{eqVstavka5.17}
\Big\|\mathcal T_{j\rho\mu}^t(\mathcal L_{g^t}'')^{-1}F^t\Big(\sum\limits_q\xi_q^tT_3
f\Big)\Big\|_{H_a^{l+2m-m_{j\rho\mu}-1/2}(\Gamma_{j\rho})}\\
\le c_9\sum\limits_k\Big\|\Big[(\mathcal L_{g^t}'')^{-1}F^t\Big(\sum\limits_q\xi_q^tT_3
f\Big)\Big]_k\Big\|_{H_{a}^0(\Theta_k\cap B_{2\delta(g^t)})}.
\end{multline}

Inequalities~\eqref{eqVstavka5.15} and~\eqref{eqVstavka5.17} and the compactness of the
embedding $H_a^{l+2m}(\Theta_k)\subset H_{a}^0(\Theta_k\cap B_{2\delta(g^t)})$ imply that
the operator $T_3$ has a compact square.

Similarly, one can show that the operator $T_1$ has a compact square. To this end, one
must use the relation
$$
\Big[\overline{\omega_{is}(\Gamma_i)}\cap\Big(\bigcup\limits_{t,j}\hat
B_{\delta(g_j^t)}(g_j^t)\Big)\Big]\setminus\mathcal K_1^{\varepsilon/2}\subset Q_b,\qquad
i=1,\dots,N_0,\ s=1,\dots S_i,
$$
which holds for some $b>0$.

Thus, it follows from~\eqref{eqVstavka5.11} and~\eqref{eqVstavka5.14} that
\begin{equation}\label{eqVstavka5.18}
\mathbf L^1 R_{\mathcal K_1}f=\xi_0 f+ T_{\mathcal K_1}f,
\end{equation}
where
\begin{equation}\label{eqVstavka5.18''}
\xi_0(x)=\sum\limits_{t,q}\xi_q^t(x)
\end{equation}
and $ T_{\mathcal K_1}: \mathcal H_a^l(Q,\Gamma)\to \mathcal H_a^l(Q,\Gamma) $ is a
bounded operator with compact square. Moreover, inequalities~\eqref{eqVstavka5.11'}
and~\eqref{eqVstavka5.14'} imply that
\begin{equation}\label{eqVstavka5.18'}
\|T_{\mathcal K_1} f\|_{\mathcal H_a^l(Q,\Gamma)}\le c_{10}\|f\|_{\mathcal
H_a^l(Q,\Gamma)}.
\end{equation}

4. Take a number $\varepsilon$ in the definition of the function $\xi$ so small that
$$
\mathcal K_1^{4\varepsilon}\subset \bigcup\limits_{t,j}\hat B_{\delta(g^t)}(g_j^t)
$$
and
$$
\xi_0(x)\ge 1/2,\qquad x\in\mathcal K_1^{4\varepsilon}
$$
(the existence of such an $\varepsilon$ follows from~\eqref{eqVstavka5.9}
and~\eqref{eqVstavka5.18''}). Later on, we will impose some additional conditions on
$\varepsilon$.

For each $\varepsilon$, we consider a function $\zeta_0\in C_0^\infty(\mathbb R^n)$
depending on $\varepsilon$, such that
\begin{equation}\label{eqVstavka5.19}
\supp\zeta_0\subset\mathcal K_1^{4\varepsilon};\qquad \zeta_0(x)=1,\ x\in\mathcal
K_1^{2\varepsilon};\qquad |D^\alpha\zeta_0(x)|\le c_{11}\varepsilon^{-|\alpha|}.
\end{equation}

For each point $g\in\overline Q\setminus \mathcal K_1^{2\varepsilon}$, we consider its
$(\varepsilon/2)$-neighborhood $B_{\varepsilon/2}(g)$. All these neighborhoods cover
$\overline Q\setminus \mathcal K_1^{2\varepsilon}$. Choose a finite subcovering
$\{B_{\varepsilon/2}(h^\tau)\}_{\tau=1}^{\tau_1}$, where $\tau_1=\tau_1(\varepsilon)$.
Let functions $\tilde\zeta^\tau\in C_0^\infty(\mathbb R^n)$ form a partition of unity for
$\overline Q\setminus \mathcal K_1^{2\varepsilon}$, subordinated to the covering
$\{B_{\varepsilon/2}(h^\tau)\}_{\tau=1}^{\tau_1}$. Then the functions
\begin{equation}\label{eqZetaZetaTau}
\zeta=\xi_0+\zeta_0(1-\xi_0),\qquad \zeta^\tau=(1-\zeta)\tilde\zeta^\tau,\quad
\tau=1,\dots,\tau_1,
\end{equation}
form a partition of unity for $\overline Q$, subordinated to the covering by the sets
$\bigcup\limits_{t,j}\hat B_{\delta(g^t)}(g_j^t)$ and $B_{\varepsilon/2}(h^\tau)$,
$\tau=1,\dots,\tau_1$.

Due to Theorem~\ref{t2.R^n_+Isomorphism} and to the general theory of elliptic
boundary-value problems in the interior of a domain and near a smooth part of the
boundary (see, e.g.,~\cite{Volevich}), there exist bounded operators
$$
 R_{\tau0}:\{f\in \mathcal
H_a^l(Q,\Gamma): \supp f\subset B_{\varepsilon/2}(h^\tau)\}\to \{u\in H_a^{l+2m}(Q):
\supp u\subset B_{\varepsilon}(h^\tau)\}
$$
and compact operators
$$
T_{\tau0}:\{f\in \mathcal H_a^l(Q,\Gamma): \supp f\subset B_{\varepsilon/2}(h^\tau)\}\to
\{f\in\mathcal H_a^l(Q,\Gamma): \supp f\subset B_{\varepsilon}(h^\tau)\}
$$
such that
\begin{equation}\label{eqVstavka5.20}
\mathbf L^0 R_{\tau0} f=f+ T_{\tau0} f.
\end{equation}

For any $f\in \mathcal H_a^l(Q,\Gamma)$, we set
\begin{equation}\label{eqRf}
R f=R_{\mathcal K_1}f+R_{\mathcal K_1}(\eta f)+\sum\limits_{\tau} R_{\tau0}(\zeta^\tau
f),
\end{equation}
where $\eta(x)={\zeta_0(x)(1-\xi_0(x))}/{\xi_0(x)}$ for $x\in \mathcal
K_1^{4\varepsilon}$ and $\eta(x)=0$ for $x\notin\mathcal K_1^{4\varepsilon}$. Note that
$\supp\zeta_0\subset\mathcal K_1^{4\varepsilon}$ and $\xi_0(x)\ge1/2$ for $x\in \mathcal
K_1^{4\varepsilon}$; hence, the function $\eta$ is supported on $\mathcal
K_1^{4\varepsilon}$ and infinitely differentiable on $\mathbb R^n$. We have
\begin{equation}\label{eqVstavka5.21}
\mathbf L^1 R f=\mathbf L^1 R_{\mathcal K_1}f+\mathbf L^1R_{\mathcal
K_1}(\eta f)\\
+\sum\limits_{\tau}\mathbf L^0 R_{\tau0}(\zeta^\tau f)+\sum\limits_\tau\{0,B_{i\mu}^1
R_{\tau0}(\zeta^\tau f)\}.
\end{equation}

Since $\zeta(x)=1$ for $x\in\mathcal K_1^{2\varepsilon}$, it follows that
$\supp\zeta^\tau f\subset\overline Q\setminus\mathcal K_1^{2\varepsilon}$. Thus, we see
that $\supp R_{\tau0}(\zeta^\tau f)\subset \overline Q\setminus\mathcal
K_1^{\varepsilon}$, while $\supp\xi\subset \mathcal K_1^{\varepsilon}$. Therefore,
$$
B_{i\mu}^1 R_{\tau0}(\zeta^\tau f)=0.
$$
Combining this relation with equalities~\eqref{eqVstavka5.18}, \eqref{eqVstavka5.20}, and
\eqref{eqVstavka5.21}, we obtain
\begin{equation}\label{eqVstavka5.22}
\mathbf L^1 R f=\xi_0 f+T_{\mathcal K_1}f+\zeta_0(1-\xi_0)f+T_{\mathcal K_1}(\eta f)
+\sum\limits_{\tau}\zeta^\tau f+Tf =f+T_{\mathcal K_1}f+Mf+Tf,
\end{equation}
where
$$
Mf=T_{\mathcal K_1}(\eta f),
$$
while $ T:\mathcal H_a^l(Q,\Gamma)\to\mathcal H_a^l(Q,\Gamma) $ is a compact operator
(whose norm may increase as $\varepsilon\to0$). By~\eqref{eqVstavka5.18'}, we have
$$
\|Mf\|_{\mathcal H_a^l(Q,\Gamma)}\le c_{10}\|\eta f\|_{\mathcal H_a^l(Q,\Gamma)}.
$$
However, $(1-\xi_0(x))/\xi_0(x)=0$ for $x\in\mathcal K_1$ and the function $\zeta_0$ is
supported in $\mathcal K_1^{4\varepsilon}$ and satisfies the inequality
in~\eqref{eqVstavka5.19}. Therefore, it follows from the last estimate, from
Lemmas~\ref{lzeta0v} and~\ref{lZetaDeltaU} and from Remark~\ref{rzeta0v} that
\begin{equation}\label{eqVstavka5.22'}
\|Mf\|_{\mathcal H_a^l(Q,\Gamma)}\le c_{11}\varepsilon\|f\|_{\mathcal H_a^l(Q,\Gamma)}.
\end{equation}

It follows from~\eqref{eqVstavka5.18'}, \eqref{eqVstavka5.22'}, and from
Lemma~\ref{lI+T+M} that the operator $\mathbf I+T_{\mathcal K_1}+M$ has the Fredholm
property, provided that $\varepsilon>0$ is sufficiently small. Applying Theorem~16.4
in~\cite{Kr}, we see that the operator
$$
\mathbf L^1 R=\mathbf I+T_{\mathcal K_1}+M+T
$$
also has the Fredholm property. Therefore, by Theorem~15.2 in~\cite{Kr}, there exist a
bounded operator $R':\mathcal H_a^l(Q,\Gamma)\to \mathcal H_a^l(Q,\Gamma)$ and a compact
operator $\mathbf T_1:\mathcal H_a^l(Q,\Gamma)\to \mathcal H_a^l(Q,\Gamma)$ such that
$$
\mathbf L^1 R R'=\mathbf I+\mathbf T_1.
$$
Setting $\mathbf R_1=R R'$, we complete the proof.
\end{proof}

Set $\mathcal H_a^l(\partial Q)=\prod\limits_{i,\mu} H_a^{l+2m-m_{i\mu}-1/2}(\Gamma_i)$.

To construct a right regularizer for the operator $\mathbf L$, we also need to prove the
existence of a ``right regularizer'' $\mathbf R_1'$ for the operator $\mathbf L^1$, which
is defined on the functions $f'\in\mathcal H_a^l(\partial Q)$ and possesses the following
properties: $\mathbf R_1'f'$ is supported near the boundary $\partial Q$ for all $f'$ and
$\mathbf R_1'f'$ is supported near the set $\mathcal K_1$ for $f'$ supported near
$\mathcal K_1$.

\begin{lemma}\label{lVstavka5.2}
Let the conditions of Theorem~$\ref{t5.1}$ be fulfilled. Then there exist a linear
bounded operator $\mathbf R_1':\mathcal H_a^l(\partial Q)\to H_a^{l+2m}(Q)$ and a compact
operator $\mathbf T_1': \mathcal H_a^l(\partial Q)\to \mathcal H_a^l(Q,\Gamma)$ such that
\begin{equation}\label{eqVstavka0}
 \mathbf L^1\mathbf R_1'f'=\{0,f'\}+\mathbf T_1'f',
\end{equation}
\begin{equation}\label{eqVstavka1}
\supp \mathbf R_1'f'\subset \overline Q\setminus Q_{\sigma}
\end{equation}
for any $f'\in \mathcal H_a^l(\partial Q)$, and
\begin{equation}\label{eqVstavka2}
\supp \mathbf R_1'f'\subset \mathcal K_1^{2\varkappa}
\end{equation}
for $f'\in \mathcal H_a^l(\partial Q)$, $\supp f'\subset\overline{\mathcal
K_1^{\varkappa}}$, where $\varkappa,\sigma>0$ are the constants from Lemma~$\ref{l4.6}$.
\end{lemma}
\begin{proof}
1. Fix an arbitrary number $\hat\varepsilon>0$ independent of $\varepsilon$. Similarly to
the proof of Lemma~\ref{l5.1}, we can construct functions $\hat\zeta,\hat\zeta^\tau\in
C_0^\infty(\mathbb R^n)$, $\tau=1,\dots,\hat\tau_1$,
$\hat\tau_1=\hat\tau_1(\hat\varepsilon)$, which form a partition of unity for $\overline
Q$, subordinated to the covering by the sets $\mathcal K_1^{2\hat\varepsilon}$ and
$B_{\hat\varepsilon/2}(h^\tau)$, $\tau=1,\dots,\hat\tau_1$, where $h^\tau\in\overline
Q\setminus\mathcal K_1^{2\hat\varepsilon}$. In particular, the function $\hat\zeta$ can
be chosen in such a way that
\begin{equation}\label{eqVstavka3}
\hat\zeta(x)=1,\qquad x\in\mathcal K_1^{\hat\varepsilon}.
\end{equation}

2. Due to Theorem~\ref{t2.R^n_+Isomorphism} and to the general theory of elliptic
boundary-value problems near a smooth part of the boundary (see, e.g.,~\cite{Volevich}),
there exist bounded operators
$$
R_{\tau0}':\{f'\in \mathcal H_a^l(\partial Q): \supp f'\subset
B_{\hat\varepsilon/2}(h^\tau)\}\to \{u\in H_a^{l+2m}(Q): \supp u\subset
B_{\hat\varepsilon}(h^\tau)\}
$$
and compact operators
$$
T_{\tau0}':\{f'\in \mathcal H_a^l(\partial Q): \supp f'\subset
B_{\hat\varepsilon/2}(h^\tau)\}\to \{f\in\mathcal H_a^l(Q,\Gamma): \supp f\subset
B_{\hat\varepsilon}(h^\tau)\}
$$
such that
\begin{equation}\label{eqVstavka4}
\mathbf L^0 R_{\tau0}' f'=\{0,f'\}+ T_{\tau0}' f'.
\end{equation}

For any $f'\in \mathcal H_a^l(\partial Q)$, we set
\begin{equation}\label{eqVstavka5}
\mathbf R_1' f'=\hat\zeta u+\sum\limits_{\tau} u^\tau,
\end{equation}
where
$$
u=\mathbf R_1\{0,f'\},\qquad u^\tau=R_{\tau0}'(\hat\zeta^\tau\{0,f'\}).
$$

Clearly, property~\eqref{eqVstavka1} holds for $2\hat\varepsilon<\sigma$, while
property~\eqref{eqVstavka2} holds for $2\hat\varepsilon<2\varkappa$ and
$\varkappa+\hat\varepsilon/2<2\varkappa$. Let us prove relation~\eqref{eqVstavka0}.

Using~\eqref{eqVstavka5} and Leibniz' formula, we have
\begin{equation}\label{eqVstavka6}
\mathbf L^1\mathbf R_1' f'=\hat\zeta\mathbf L^1u+\Big\{0,\sum\limits_{s=1}^{S_i}T_{i\mu
s} f'\Big\}+Tu+\sum\limits_\tau\mathbf L^0 u^\tau+\{0,B_{i\mu}^1u^\tau\},
\end{equation}
where
\begin{equation}\label{eqVstavka7}
T_{i\mu s} f'=\big(\hat\zeta(\omega_{is}(x))-\hat\zeta(x)\big)\big(B_{i\mu
s}^0(x,D_x)(\xi u)\big)\big(\omega_{is}(x)\big)\big|_{\Gamma_i},
\end{equation}
while $T:H_a^{l+2m-1}(Q)\to \mathcal H_a^l(Q,\Gamma)$ is a bounded operator. By virtue of
the compactness of the embedding $H_a^{l+2m}(Q)\subset H_a^{l+2m-1}(Q)$, the operator
$T:H_a^{l+2m}(Q)\to \mathcal H_a^l(Q,\Gamma)$ is compact.

Now it follows from Lemma~\ref{l5.1}, from~\eqref{eqVstavka4}, and
from~\eqref{eqVstavka6} that
\begin{equation}\label{eqVstavka8}
\mathbf L^1\mathbf R_1' f'=\{0,f'\}+T'f'+\Big\{0,\sum\limits_{s=1}^{S_i}T_{i\mu s}
f'\Big\}+\{0,B_{i\mu}^1u^\tau\},
\end{equation}
where $T':\mathcal H_a^l(\partial Q)\to \mathcal H_a^l(Q,\Gamma)$ is a compact operator.

3. Let us prove that the operator $T_{i\mu s}:\mathcal H_a^l(\partial Q)\to
H_a^{l+2m-m_{i\mu}-1/2}(\Gamma_i)$ is compact. Since $\omega_{is}$ are $C^\infty$
diffeomorphisms, it follows from Lemma~\ref{l4.3} that
\begin{multline}\label{eqVstavka9}
\|T_{i\mu s}f'\|_{H_a^{l+2m-m_{i\mu}-1/2}(\Gamma_i)}\\
\le k_1\big\|\big(\hat\zeta(x)-\hat\zeta(\omega_{is}^{-1}(x))\big)\big(B_{i\mu
s}^0(x,D_x)(\xi
u)\big)\big|_{\omega_{is}(\Gamma_i)}\big\|_{H_a^{l+2m-m_{i\mu}-1/2}(\omega_{is}(\Gamma_i))},
\end{multline}
where $k_1,k_2,{\dots}>0$ do not depend on $f'$.

Denote $\mathcal M_{is}=\overline{\omega_{is}(\Gamma_i)}\setminus\omega_{is}(\Gamma_i)$.
For every $x\in\mathcal M_{is}$, we have either $x\in\mathcal K_1$, or $x\in\mathcal
K_2$, or $x\in\mathcal K_3$. If $x\in\mathcal M_{is}\cap\mathcal K_{1}$, then both $x$
and $\omega_{is}^{-1}(x)$ belong to $\mathcal K_1$. Therefore,
$$
\hat\zeta(x)-\hat\zeta(\omega_{is}^{-1}(x))=0,\qquad x\in(\mathcal M_{is}\cap\mathcal
K_1)^d,
$$
where $(\mathcal M_{is}\cap\mathcal K_1)^d$ is the $d$-neighborhood of the set $\mathcal
M_{is}\cap\mathcal K_1$ and $d>0$ is sufficiently small.

If $\mathcal M_{is}\cap(\mathcal K_{2}\cup\mathcal K_3)\ne\varnothing$, then
$$
\xi(x)=0,\qquad x\in(\mathcal M_{is}\cap(\mathcal K_2\cup\mathcal K_3))^d,
$$
where $(\mathcal M_{is}\cap(\mathcal K_2\cup\mathcal K_3))^d$ is the $d$-neighborhood of
the set $\mathcal M_{is}\cap(\mathcal K_2\cup\mathcal K_3)$ and $d>0$ is sufficiently
small.

In all these cases, we see that
$$
\supp \big(\hat\zeta(x)-\hat\zeta(\omega_{is}^{-1}(x))\big)\big(B_{i\mu s}^0(x,D_x)(\xi
u)\big)\big|_{\omega_{is}(\Gamma_i)}\subset Q_b,
$$
where $b>0$ is sufficiently small.

Using estimate~\eqref{eqVstavka9}, Lemma~\ref{l3.10}, and equivalence of the norms in
$H_a^l(Q_b)$ and $W^l(Q_b)$, we obtain
\begin{equation}\label{eqVstavka10}
\|T_{i\mu s}f'\|_{H_a^{l+2m-m_{i\mu}-1/2}(\Gamma_i)}\\
\le k_2\|u\|_{W^{l+2m}(Q_b)}\le k_3(\|A^0 u\|_{H_a^l(Q)}+\|u\|_{H_{a}^0(Q)}).
\end{equation}
By Lemma~\ref{l5.1}, $A^0u=A^0\mathbf R^1\{0,f'\}=\mathbf T_{11}f'$, where $\mathbf
T_{11}:\mathcal H_a^l(\partial Q)\to H_a^l(Q)$ is a compact operator. Hence,
inequality~\eqref{eqVstavka10} takes the form
$$
\|T_{i\mu s}f'\|_{H_a^{l+2m-m_{i\mu}-1/2}(\Gamma_i)}\\
\le k_3(\|\mathbf T_{11}f'\|_{H_a^l(Q)}+\|\mathbf R^1\{0,f'\}\|_{H_{a}^0(Q)}),
$$
which implies that $T_{i\mu s}$ is a compact operator (because $\mathbf T_{11}$ is
compact and $H_{a}^{l+2m}(Q)$ is compactly embedded into $H_{a}^0(Q)$).

4. The expression $B_{i\mu}^1 u^\tau$ consists of the terms
\begin{equation}\label{eqVstavka11}
\big(B_{i\mu s}^0(x,D_x)(\xi R_{\tau
0}'(\hat\zeta^\tau\{0,f'\}))\big)\big(\omega_{is}(x)\big)\big|_{\Gamma_i},\qquad
s=1,\dots, S_i.
\end{equation}
Since $\supp R_{\tau 0}'(\hat\zeta^\tau\{0,f'\})\subset B_{\hat\varepsilon}(h^\tau)$,
where $h^\tau\in\overline Q\setminus\mathcal K_1^{2\hat\varepsilon}$, it follows that
$$\supp R_{\tau 0}'(\hat\zeta^\tau\{0,f'\})\subset \overline Q\setminus\mathcal
K_1^{\hat\varepsilon}.$$ On the other hand, $\supp\xi\subset\mathcal K_1^{\varepsilon}$.
Hence,
$$
\supp B_{i\mu s}^0(x,D_x)(\xi R_{\tau
0}'(\hat\zeta^\tau\{0,f'\}))|_{\omega_{is}(\Gamma_i)}\subset Q_b,
$$
where $b>0$ is sufficiently small. Therefore, applying Lemma~\ref{l3.10} and
equality~\eqref{eqVstavka4}, we can show similarly to the above that each of the
operators in~\eqref{eqVstavka11} is compact. Thus, we see that~\eqref{eqVstavka8} is
equivalent to~\eqref{eqVstavka0}.
\end{proof}

Now we can prove Theorem~\ref{t5.2}.

\begin{proof}[Proof of Theorem~$\ref{t5.2}$]
1. We set
\begin{equation}\label{5.Phi}
\Phi= B^2\mathbf R_1f,\qquad f=\{f_0,f'\}\in\mathcal H_a^l(Q,\Gamma).
\end{equation}
Introduce the bounded operator ${\mathbf R}:\mathcal H_a^l(Q,\Gamma)\to H_a^{l+2m}(Q)$ by
the formula
$$
{\mathbf R}f=\mathbf R_1f-\mathbf R_1'\Phi+\mathbf R_1'
 B^2\mathbf R_1'\Phi,
$$
where $\mathbf R_1$ and $\mathbf R_1'$ are the operators occurring in Lemmas~\ref{l5.1}
and~\ref{lVstavka5.2}, respectively. Let us show that ${\mathbf R}$ is the desired
operator.

For simplicity, we denote diverse compact operators by the same letter $T$.

It follows from Lemmas~\ref{l5.1} and~\ref{lVstavka5.2} that
\begin{equation}\label{eqRegL_1}
A{\mathbf R}f=A^0{\mathbf R}f=A^0\mathbf R_1f-A^0\mathbf R_1'(\Phi- B^2\mathbf
R_1'\Phi)=f_0+Tf
\end{equation}
(recall that $A^1=0$) and
\begin{multline}\label{eqRegL_2}
C{\mathbf R}f=C\mathbf R_1f-C\mathbf R_1'\Phi+
C\mathbf R_1'B^2\mathbf R_1'\Phi\\
=(f'+Tf)-(\Phi+T\Phi)+(B^2\mathbf R_1'\Phi+T\mathbf B^2\mathbf
R_1'\Phi)=f'-\Phi+B^2\mathbf R_1'\Phi+Tf.
\end{multline}
Applying the operator $B^2$ to the function ${\mathbf R}f$ and using~\eqref{5.Phi}, we
obtain
\begin{equation}\label{eqRegL_3}
B^2{\mathbf R}f=\Phi-B^2\mathbf R_1'\Phi+B^2\mathbf R_1' B^2\mathbf R_1'\Phi.
\end{equation}
Summing relations~\eqref{eqRegL_2} and \eqref{eqRegL_3} and recalling that $B^3=0$, we
obtain
\begin{equation}\label{eqRegL_4}
B{\mathbf R}f=f'+Tf+B^2\mathbf R_1' B^2\mathbf R_1'\Phi.
\end{equation}

2. Let us show that
\begin{equation}\label{eqRegL_5}
B^2\mathbf R_1' B^2\mathbf R_1'\Phi=0.
\end{equation}
It follows from relation~\eqref{eqVstavka1} in Lemma~\ref{lVstavka5.2} that
$$
\supp\mathbf R_1'\Phi\subset\overline Q\setminus Q_{\sigma}.
$$
Therefore, estimate~\eqref{4.16} implies that
$$
\supp B^2\mathbf R_1'\Phi\subset\overline{\mathcal K_1^\varkappa}.
$$
Furthermore, it follows from relation~\eqref{eqVstavka2} in Lemma~\ref{lVstavka5.2} that
$$
\supp\mathbf R_1' B^2\mathbf R_1'\Phi\subset\mathcal K_1^{2\varkappa}.
$$
Combining this fact with inequality~\eqref{4.15} yields~\eqref{eqRegL_5}.

Relations~\eqref{eqRegL_1},  \eqref{eqRegL_4}, and~\eqref{eqRegL_5} prove the theorem.
\end{proof}

\begin{remark}
Using the results in~\cite{GurIzvRAN}, one can show that Theorem~\ref{t5.1} remains true
for the case in which the transformations $\omega_{is}$ are nonlinear near the set
$\mathcal K_1$, while their linear parts at the points of the set $\mathcal K_1$ satisfy
Condition~\ref{cond1.4}. Moreover, the index of the problem with nonlinear
transformations $\omega_{is}$ is equal to the index of the corresponding problem with
transformations linearized near the set~$\mathcal K_1$.
\end{remark}

\section{Some Generalizations}\label{sec6}

\subsection{}

In this section, we generalize the results of Secs.~\ref{sec4} and~\ref{sec5} to the case
where diffeomorphisms $\omega_{is}$ are defined only on some neighborhood of the set
$\mathcal K_1$, the operators $B_{i\mu}^2$ are abstract nonlocal operators supported
outside the set $\mathcal K_1$, and $A^1$ and $B_{i\mu}^3$ are compact perturbations on
the corresponding weighted spaces.

Consider the differential operators
$$
A^0\equiv A^0(x,D)=\sum\limits_{|\alpha|=2m}a_\alpha(x)D^\alpha,\qquad B_{i\mu s}^0\equiv
B_{i\mu s}^0(x,D)=\sum\limits_{|\alpha|=m_{i\mu}}b_{i\mu s \alpha}(x)D^\alpha,
$$
where $a_\alpha, b_{i\mu s \alpha}\in C^\infty(\mathbb R^n)$ are complex-valued functions
($i=1,\dots,N_0$; $\mu=1,\dots,m$; $s=0,\dots, S_i$), $m_{i\mu}\le 2m-1$.

Let a domain $Q\subset \mathbb R^n$ satisfy the assumptions of Sec.~\ref{sec1}. As in
Sec.~\ref{sec1}, we denote
$$
\mathcal K_1=\bigcup\limits_{\nu=1}^{N_1}\mathcal K_{1\nu}=\partial
Q\setminus\bigcup_i\Gamma_i,
$$
where $\mathcal K_{1\nu}$ are mutually disjoint $(n-2)$-dimensional connected $C^\infty$
manifolds without a boundary.

Let $\omega_{is}$ ($i=1,\dots,N_0$; $s=1,\dots,S_i$) denote a $C^\infty$ diffeomorphism
mapping $(\overline{\Gamma_i}\setminus\Gamma_i)^{\varepsilon_0}$ onto the set
$\omega_{is}((\overline{\Gamma_i}\setminus\Gamma_i)^{\varepsilon_0})$, where
$\varepsilon_0>0$ is some number, in such a way that
\begin{enumerate}
\item
$\omega_{is}(\Gamma_i\cap\mathcal K_1^{\varepsilon_0})\subset Q$,
\item
if $\eta$ ($1\le\eta\le N_1$) and $i$ ($1\le i\le N_0$) are such that $\mathcal
K_{1\eta}\subset\overline{\Gamma_i}\setminus\Gamma_i$, then, for every $s$ ($1\le s\le
S_i$), there is $\nu$ ($1\le\nu\le N_1$) such that $\omega_{is}(\mathcal
K_{1\eta})=\mathcal K_{1\nu}$.
\end{enumerate}

Along with the set $\mathcal K_1$, we introduce the set
$$
K_2=\bigcup\limits_{\nu=1}^{N_2} K_{2\nu}\subset\bigcup_i\Gamma_i,
$$
where $ K_{2\nu}$ are mutually disjoint connected $C^\infty$ manifolds without a
boundary. In particular, the set $K_2$ can be empty. We use either the set
$$
K=\mathcal K_1
$$
or the set
$$
K=\mathcal K_1\cup K_2
$$
in the definition of the space $H_a^l(Q)=H_a^l(Q,K)$.

We study the following nonlocal elliptic problem:
\begin{equation}\label{6.2}
Au\equiv A^0u+A^1u=f_0(x),\qquad x\in Q,
\end{equation}
\begin{equation}\label{6.3}
B_{i\mu}u\equiv \sum\limits_{j=0}^3 B_{i\mu}^ju=f_{i\mu}(x),\qquad x\in \Gamma_i;\
i=1,\dots,N_0;\ \mu=1,\dots, m.
\end{equation}
Here
$$
B_{i\mu}^0=B_{i\mu0}^0u|_{\Gamma_i},\qquad B_{i\mu}^1u=\sum\limits_{s=1}^{S_i}\big(
B_{i\mu s}^0(x,D)(\xi u)\big)\big(\omega_{is}(x)\big)\big|_{\Gamma_i},
$$
a function $\xi\in C_0^\infty(\mathbb R^n)$ is such that $\xi(x)=1$ for $x\in\mathcal
K_1^{\varepsilon/2}$ and $\supp\xi\subset K_1^{\varepsilon}$, while $\varepsilon>0$ is so
small that if $\omega_{is}(\mathcal K_{1\eta})=\mathcal K_{1\nu}$, then $\mathcal
K_{1\nu}^\varepsilon\subset \omega_{is}(\mathcal K_{1\eta}^{\varepsilon_0})$.

Assume that the operators $A^0$ and $B_{i\mu0}^0$ satisfy Conditions~\ref{cond1.1}
and~\ref{cond1.2}, and the transformations $\omega_{is}$ satisfy Conditions~\ref{cond1.3}
and~\ref{cond1.4}. We also suppose that the following conditions for the operators $A^1$,
$B_{i\mu}^2$, and $B_{i\mu}^3$ hold.

\begin{condition}[compactness of perturbations]\label{cond6.1}
The linear operators $$A^1:H_a^{l+2m-1}(Q)\to H_a^l(Q),\quad
B_{i\mu}^3:H_a^{l+2m-1}(Q)\to H_a^{l+2m-m_{i\mu}-1/2}(\Gamma_i)$$ are bounded {\rm
(}$i=1,\dots,N_0;$ $\mu=1,\dots,m${\rm )}.
\end{condition}

\begin{condition}[separability from the conjugation points]\label{cond6.2}
The linear operators $B_{i\mu}^2:H_a^{l+2m}(Q)\to H_a^{l+2m-m_{i\mu}-1/2}(\Gamma_i)$ are
bounded, and there exist numbers $\sigma>0$ and $\varkappa_1>\varkappa_2>0$ such that
\begin{equation}\label{6.4}
\|B_{i\mu}^2u\|_{H_a^{l+2m-m_{i\mu}-1/2}(\Gamma_i)}\le
c_1\|u\|_{H_a^{l+2m}(Q\setminus\overline{\mathcal K_1^{\varkappa_1}})}
\end{equation}
for all $u\in H_a^{l+2m}(Q\setminus\overline{\mathcal K_1^{\varkappa_1}})$ and
\begin{equation}\label{6.5}
\|B_{i\mu}^2u\|_{H_a^{l+2m-m_{i\mu}-1/2}(\Gamma_i\setminus\overline{\mathcal
K_1^{\varkappa_2}})}\le c_2\|u\|_{H_a^{l+2m}(Q_\sigma)}
\end{equation}
for all $u\in H_a^{l+2m}(Q_\sigma);$ here $i=1,\dots,N_0;$ $\mu=1,\dots,m;$ $c_1,c_2>0$
do not depend on $u$.
\end{condition}

\begin{remark}
It follows from Lemma~\ref{l4.6} that problem~\eqref{1.3}, \eqref{1.4} can be represented
in the form~\eqref{6.2}, \eqref{6.3} with the operators $A^1$, $B_{i\mu}^2$, and
$B_{i\mu}^3$ satisfying Conditions~\ref{cond6.1} and~\ref{cond6.2}.
\end{remark}

\begin{remark}
The proofs of Theorems~\ref{t4.1} and~\ref{t5.2} use only Conditions~\ref{cond6.1}
and~\ref{cond6.2}, rather than any explicit form of the operators~$A^1$, $B_{i\mu}^2$,
and $B_{i\mu}^3$. To prove that those operators satisfied Conditions~\ref{cond6.1}
and~\ref{cond6.2}, we used Condition~\ref{cond4.1}, which provided the choice of the set
$K$ and the number $a$ in Secs.~\ref{sec1}--\ref{sec5}. However, Condition~\ref{cond4.1}
is needless in this section because the fulfilment of Conditions~\ref{cond6.1}
and~\ref{cond6.2} is postulated rather than proved.
\end{remark}

\subsection{}

We introduce the linear bounded operator
$$
\mathbf L=\{A,B_{i\mu}\}:H_a^{l+2m}(Q)\to\mathcal H_a^l(Q,\Gamma)
$$
corresponding to problem~\eqref{6.2}, \eqref{6.3}.

For every fixed point $g\in\mathcal K_1$, we consider the linear bounded operator
$\mathcal L_g(\omega):\mathcal E_a^{l+2m}(\theta)\to\mathcal E_a^l(\theta,\gamma)$ given
by~\eqref{2.1}, where $\omega\in S^{n-3}$. We also consider the analytic operator-valued
function $\hat{\mathcal L}_g(\lambda): \mathcal W^{l+2m}(d_1,d_2)\to\mathcal
W^l[d_1,d_2]$ given by~\eqref{2.4}.

For every fixed point $g\in K_2$, we consider the linear bounded operator $\mathcal
L_g(\omega):E_a^{l+2m}(\mathbb R_+^2)\to\mathcal E_a^l(\mathbb R_+^2,\gamma)$ given
by~\eqref{2.19}, where $\omega\in S^{n-3}$. We also consider the analytic operator-valued
function $\hat{\mathcal L}_g(\lambda): W^{l+2m}(-\pi/2,\pi/2)\to\mathcal
W^l[-\pi/2,\pi/2]$ given by~\eqref{2.22}.

Similarly to the proofs of Theorems~\ref{t4.1} and~\ref{t5.2}, using
Conditions~\ref{cond6.1} and~\ref{cond6.2} and assuming that the number $\varepsilon>0$
in the definition of the function $\xi$ is sufficiently small, we obtain the following
two results.

\begin{theorem}\label{t6.1}
Let Conditions~$\ref{cond1.1}$--$\ref{cond1.4}$, $\ref{cond6.1}$, and~$\ref{cond6.2}$
hold. Assume that the line $\Im\lambda=a+1-l-2m$ contains no eigenvalues of
$\hat{\mathcal L}_g(\lambda)$ for any $g\in K$ and $\dim\mathcal N(\mathcal
L_g(\omega))=\codim\mathcal R(\mathcal L_g(\omega))=0$ for any $g\in K$ and $\omega\in
S^{n-3}$. Then the following estimate holds for all $u\in H_a^{l+2m}(Q)$$:$
$$
\|u\|_{H_a^{l+2m}(Q)}\le c(\|\mathbf L u\|_{\mathcal
H_a^l(Q,\Gamma)}+\|u\|_{H_{a}^{l+2m-1}(Q)}),
$$
where $c>0$ does not depend on $u$.
\end{theorem}

\begin{theorem}\label{t6.2}
Let the conditions of Theorem~$\ref{t6.1}$ be fulfilled. Then there exists a linear
bounded operator $\mathbf R:\mathcal H_a^l(Q,\Gamma)\to H_a^{l+2m}(Q)$ such that
$$
\mathbf L\mathbf R=\mathbf I+\mathbf T
$$
where $\mathbf T: \mathcal H_a^l(Q,\Gamma)\to \mathcal H_a^l(Q,\Gamma)$ is a compact
operator.
\end{theorem}

Along with the operator $\mathbf L$, we consider the bounded operator
$$
\mathbf L^1=\{A^0,B_{i\mu}^0+B_{i\mu}^1\}:H_a^{l+2m}(Q)\to\mathcal H_a^l(Q,\Gamma).
$$
It follows from Theorem~\ref{t5.1} that the operator $\mathbf L^1$ has the Fredholm
property.

Similarly to the proofs of Theorem~\ref{t5.1} and Corollary~\ref{c5.1}, using
Theorems~\ref{t6.1} and~\ref{t6.2}, we obtain the following  result.
\begin{theorem}\label{t6.3}
Let the conditions of Theorem~$\ref{t6.1}$ be fulfilled. Then the operator $\mathbf L:
H_a^{l+2m}(Q)\to \mathcal H_a^l(Q,\Gamma)$ has the Fredholm property and $\ind\mathbf
L=\ind\mathbf L^1$.
\end{theorem}

Theorem~\ref{t6.3} shows that the addition of the operators $A^1$, $B_{i\mu}^2$, and
$B_{i\mu}^3$ satisfying Conditions~\ref{cond6.1} and~\ref{cond6.2} neither violates the
Fredholm property nor changes the index.

\subsection{}

Now we consider an example of an elliptic problem with distributed nonlocal terms
satisfying Condition~\ref{cond6.2}.

\begin{example}\label{exDistrNonlocalTerm}
Let $Q\subset \mathbb R^3$ be a bounded domain with boundary $\partial Q$ which is a
surface of revolution about the axis $x_3$. Denote $P=\{0,0,1\}\cup\{0,0,-1\}$. Assume
that, outside $P^{1/4}$, the surface $\partial Q$ coincides with the boundary of the
domain
$$
\{x:\ x_3<1-\sqrt{x_1^2+x_2^2}\}\cap\{x:\ x_3>-1+\sqrt{x_1^2+x_2^2}\}.
$$
Denote
$$
\Gamma_1=\{x\in\partial Q:\ x_3<0\},\qquad \Gamma_2=\{x\in\partial Q:\ x_3>0\}.
$$
In this case, we have
$$
\mathcal K_1=\{x:\ x_1^2+x_2^2=1,\ x_3=0\}.
$$
Assume that the boundary $\partial Q$ is infinitely smooth outside the set $\mathcal
K_1$.

We introduce the operators
\begin{equation}\label{6.7}
B_l^1u=-\alpha_l(\xi u)(\omega_l(x))|_{\Gamma_l},\qquad l=1,2.
\end{equation}
The transformations $\omega_l(x)$ in~\eqref{6.7} are defined for $x\in\mathcal
K_1^{\varepsilon_0}$ by the formula
\begin{equation}\label{6.8}
\begin{aligned}
\omega_l(x)=&\left(\cos   \varphi\left[1-\frac{1}{\sqrt 2}((1-r)+(-1)^lx_3)\right],\right.\\
&\ \ \left. \sin  \varphi\left[1-\frac{1}{\sqrt 2}((1-r)+(-1)^lx_3)\right],\
\frac{1}{\sqrt 2}[(-1)^{l+1}(1-r)+x_3]\right),
\end{aligned}
\end{equation}
where $r,\varphi,x_3$ are the cylindrical coordinates of the point $x$, the number
$\varepsilon_0>0$ is sufficiently small, the function $\xi\in C_0^\infty(\mathbb R^3)$ is
such that $\xi(x)=1$ for $x\in\mathcal K_1^{\varepsilon/2}$ and $\supp\xi\subset\mathcal
K_1^\varepsilon$, $0<\varepsilon\le\varepsilon_0$, and $\alpha_1,\alpha_2\in\mathbb R$.
Clearly, we have $\omega_l(\mathcal K_1)=\mathcal K_1$.

Since $\Gamma_l\in C^\infty$, one can find a sufficiently small $\varkappa>0$ possessing
the following property: for any $x\in\Gamma_l^{5\varkappa}\cap Q$, there exists a unique
pair $(y,t)$, $y\in\Gamma_l$, $t>0$, such that $x=y+n_y t$, where $n_y$ denotes the unit
normal to $\Gamma_l$ at the point $y$, directed inside the domain $Q$. One can show that
the transformation $x\mapsto (y,t)$ is a $C^\infty$ diffeomorphism mapping
$\Gamma_l^{5\varkappa}\cap Q$ onto $\Gamma_l\times (0,5\varkappa),$ provided that
$\varkappa>0$ is sufficiently small.

Introduce a function $\eta\in C_0^\infty(\mathbb R)$ such that $\eta(t)=1$ for
$t\in(3\varkappa,4\varkappa)$ and $\supp\eta\subset (2\varkappa,5\varkappa)$. Consider
the operators
\begin{equation}\label{6.9}
B_l^2u= F^{-1}(b_l F(\eta u))|_{t=0},\qquad l=1,2,
\end{equation}
where
$$
F(\eta u)(y,\lambda)=\frac{1}{\sqrt{2\pi}}\int\limits_{\mathbb R} e^{-i\lambda
t}\eta(t)u(y,t)\,dt
$$
is the Fourier transform with respect to $t$, $F^{-1}$ is the inverse Fourier transform,
$b_l(\lambda)$ is a function continuous on $\mathbb R$ and such that
\begin{equation}\label{6.9'}
\sup\limits_{\lambda\in\mathbb R} |b_l(\lambda)|<\infty.
\end{equation}

We consider the following nonlocal boundary-value problem:
\begin{equation}\label{6.10}
-\Delta u=f_0(x),\qquad x\in Q,
\end{equation}
\begin{equation}\label{6.11}
u|_{\Gamma_l}+B_l^1u+B_l^2u=f_l(x),\qquad x\in\Gamma_l,\ l=1,2.
\end{equation}
We set $K=\mathcal K_1$ in the definition of the space $H_a^l(Q)=H_a^l(Q,K)$.

The Jacobian $\dfrac{D\omega_l}{Dx}$ can be calculated on $\mathcal K_1$ as follows:
$$
\left.\frac{D\omega_l}{Dx}\right|_{\mathcal
K_1}=\left.\frac{D\omega_l}{D(r,\varphi,x_3)}\right|_{\mathcal
K_1}\left.\frac{D(r,\varphi,x_3)}{Dx}\right|_{\mathcal K_1}.
$$
Since $\left.\dfrac{D(r,\varphi,x_3)}{Dx}\right|_{\mathcal K_1}=1$, we have
$$
\begin{aligned}
\frac{D\omega_l}{Dx}\Big|_{\mathcal K_1}&=\left. \det
\begin{pmatrix}
\dfrac{\partial\omega_{l1}}{\partial r} & \dfrac{\partial\omega_{l1}}{\partial \varphi} &
\dfrac{\partial\omega_{l1}}{\partial x_3}\\
\dfrac{\partial\omega_{l2}}{\partial r} & \dfrac{\partial\omega_{l2}}{\partial \varphi} &
\dfrac{\partial\omega_{l2}}{\partial x_3}\\
\dfrac{\partial\omega_{l3}}{\partial r} & \dfrac{\partial\omega_{l3}}{\partial \varphi} &
\dfrac{\partial\omega_{l3}}{\partial x_3}
\end{pmatrix}\right|_{\mathcal K_1}\\
&= \det
\begin{pmatrix}
\frac{1}{\sqrt 2}\cos\varphi & -\sin\varphi &
\frac{1}{\sqrt 2}(-1)^{l+1}\cos\varphi\\
\frac{1}{\sqrt 2}\sin\varphi & \cos\varphi &
\frac{1}{\sqrt 2}(-1)^{l+1}\sin\varphi\\
\frac{1}{\sqrt 2}(-1)^{l} & 0 & \frac{1}{\sqrt 2}
\end{pmatrix}
=1.
\end{aligned}
$$
Therefore, we can choose $\varepsilon_0>0$ so small that $\omega_l:\mathcal
K_1^{\varepsilon_0}\to\omega_l(\mathcal K_1^{\varepsilon_0})$ is a $C^\infty$
diffeomorphism. Furthermore, $\omega_l(\mathcal K_1^{\varepsilon_0})=\mathcal
K_1^{\varepsilon_0}$ and $\omega_l(\Gamma_l\cap\mathcal K_1^{\varepsilon_0})=\{x:\
1-\varepsilon_0<r<1,\ x_3=0\}\subset Q$.

Since $\omega_l(g)=g$ for $g\in\mathcal K_1$ and $l=1,2$, the orbit of each point
$g\in\mathcal K_1$ consists of one point $g_1=g$.  For $x\ne 0$, we introduce the new
variables $x'=(y',z')$ given by
$$
y_1'=1-r,\qquad y_2'=x_3,\qquad z'=\varphi-\varphi_1,
$$
where  $y'=(y_1',y_2')$, while $r,\varphi,x_3$ and $1,\varphi_1,0$ are the cylindrical
coordinates of the points $x$ and $g_1$, respectively.

Set $V(0)=\{x':\ |y'|<\varepsilon_1,\ |z'|<\varepsilon_1\}$, where
$\varepsilon_1<\min\{\varepsilon_0,2\varkappa\}$. Let $\hat V(g_1)=V(g_1)$ be the
pre-image of the set $V(0)$ under the change of variables $x\mapsto x'$. Assume that
$\supp u\subset V(g_1)\cap\overline Q$. Introduce the function $v(x')=u(x(x'))$. Denote
$x'=(y',z')$ by $x=(y_1,y_2,z)$. Then the boundary-value problem~\eqref{6.10},
\eqref{6.11} takes the form
\begin{equation}\label{6.12}
-\dfrac{\partial^2 v}{\partial y_1^2}-\dfrac{\partial^2 v}{\partial
y_2^2}-\dfrac{1}{(1-y_1)^2}\dfrac{\partial^2 v}{\partial
z^2}-\dfrac{1}{1-y_1}\dfrac{\partial v}{\partial y_1}=f_0'(x),\qquad x\in\Theta,
\end{equation}
\begin{equation}\label{6.13}
v(y,z)|_{\Gamma_{1l}}-\alpha_l v(\mathcal G_l y,z)|_{\Gamma_{1l}}=f_l'(x),\qquad
x\in\Gamma_{1l},\ l=1,2,
\end{equation}
where $\mathcal G_l$ is the operator of rotation by the angle $(-1)^{l+1}\pi/4$,
$$
\Theta=\{x:\ r>0,\ |\varphi|<\pi/4,\ z\in\mathbb R\},\qquad \Gamma_{1l}=\{x:\ r>0,\
\varphi=(-1)^l\pi/4,\ z\in\mathbb R\},
$$
and $r,\varphi$ are the polar coordinates of the point $y$.

Clearly, Conditions~\ref{cond1.1}--\ref{cond1.4} are fulfilled in this example.

Passing to the principal homogeneous part in Eq.~\eqref{6.12} with the coefficients
freezed at the origin, we obtain
\begin{equation}\label{6.14}
-\Delta v=f_0'(x),\qquad x\in\Theta.
\end{equation}
Nonlocal boundary conditions~\eqref{6.13} do not change.

In the case of problem~\eqref{6.14}, \eqref{6.13}, the operator $\mathcal L_g=\mathcal L:
H_a^2(\Theta)\to\mathcal H_a^0(\Theta,\Gamma)$ given by~\eqref{1.11} takes the form
\begin{equation}\label{6.14'}
\mathcal L v=(-\Delta v,\
v(\varphi,r,z)|_{\Gamma_{11}}-\alpha_1v(\varphi+\pi/4,r,z)|_{\Gamma_{11}},\
v(\varphi,r,z)|_{\Gamma_{12}}-\alpha_2v(\varphi-\pi/4,r,z)|_{\Gamma_{12}}).
\end{equation}

Hence, the operators $$\mathcal L_g(\omega)=\mathcal L(\omega):E_a^2(\theta)\to\mathcal
E_a^0(\theta,\gamma)$$ and
$$\hat{\mathcal L}_g(\lambda)=\hat{\mathcal
L}(\lambda):W^2(-\pi/4,\pi/4)\to \mathcal W^0[-\pi/4,\pi/4]$$ given by~\eqref{2.1}
and~\eqref{2.4} have the form
\begin{equation}\label{6.14''}
\begin{aligned}
\mathcal L(\omega) V&=(-\Delta_y V+V,\\
&\quad V(\varphi,r)|_{\gamma_{11}}-\alpha_1V(\varphi+\pi/4,r)|_{\gamma_{11}},\
V(\varphi,r)|_{\gamma_{12}}-\alpha_2V(\varphi-\pi/4,r)|_{\gamma_{12}})
\end{aligned}
\end{equation}
and
\begin{equation}\label{6.14'''}
\hat{\mathcal L}(\lambda) w=(-w_{\varphi\varphi}+\lambda^2w,\ w(-\pi/4)-\alpha_1w(0),\
w(\pi/4)-\alpha_2w(0)),
\end{equation}
respectively, where
$$
\theta=\{y\in\mathbb R^2:\ r>0,\ |\varphi|<\pi/4\},\qquad \gamma_{1l}=\{y\in\mathbb R^2:\
r>0,\ \varphi=(-1)^l\pi/4\},
$$
and $\omega=\pm 1$ (cf. Example~\ref{ex2.1} for $d=\pi/2$).

It follows from Example~\ref{ex2.1} that the strip $-1\le\Im\lambda\le 1$ contains no
eigenvalues of the operator-valued function $\hat{\mathcal L}(\lambda)$ and the operator
${\mathcal L}(\omega)$, $\omega=\pm 1$, is an isomorphism for $0\le a\le 2$ and
$\alpha_1,\alpha_2\in \mathbb R$ such that $0<|\alpha_1+\alpha_2|<2$ and
$\pi/4<\arctan\sqrt{4(\alpha_1+\alpha_2)^{-2}-1}$. Therefore, by Theorem~\ref{t2.2}, the
operator $\mathcal L$ is an isomorphism for the above $a$ and $\alpha_l$.

Suppose that $a>1$ and prove that the operators $B_l^2$ satisfy Condition~\ref{cond6.2}
with $\varkappa_1=2\varkappa$, $\varkappa_2=\varkappa$, and $\sigma=\varkappa$. Using
Lemma~\ref{l4.3}, we have
\begin{multline*}
\|B_l^2u\|_{H_a^{3/2}(\Gamma_l)}\le k_1\|F^{-1}(b_l F(\eta
u))\|_{W^2(\Gamma_l\times(-\infty,\infty))}\\
=\dfrac{k_1}{\sqrt{2\pi}}\sum\limits_{|\alpha|+\beta\le 2}\Big\|D^\alpha_y
D^\beta_t\int\limits_{\mathbb R}e^{i\lambda t}b_l(\lambda)F(\eta
u)(y,\lambda)\,d\lambda\Big\|_{L_2(\Gamma_l\times(-\infty,\infty))}\\
=\dfrac{k_1}{\sqrt{2\pi}}\sum\limits_{|\alpha|+\beta\le 2}\Big\| \int\limits_{\mathbb
R}e^{i\lambda t}\lambda^\beta b_l(\lambda)D^\alpha_yF(\eta
u)(y,\lambda)\,d\lambda\Big\|_{L_2(\Gamma_l\times(-\infty,\infty))},
\end{multline*}
where $D_t=-i\partial/\partial t$. Applying the Plancherel theorem and using
property~\eqref{6.9'}, we obtain
\begin{multline*}
\|B_l^2u\|_{H_a^{3/2}(\Gamma_l)}\le
\dfrac{k_1}{\sqrt{2\pi}}\sum\limits_{|\alpha|+\beta\le 2}\|\lambda^\beta
b_l(\lambda)D^\alpha_yF(\eta
u)(y,\lambda)\|_{L_2(\Gamma_l\times(-\infty,\infty))}\\
\le k_2\sum\limits_{|\alpha|+\beta\le 2}\|\lambda^\beta D^\alpha_yF(\eta
u)(y,\lambda)\|_{L_2(\Gamma_l\times(-\infty,\infty))}.
\end{multline*}
Applying the Plancherel theorem once more and taking into account that
$\supp\eta\subset(2\varkappa,5\varkappa)$, we have
$$
\|B_l^2u\|_{H_a^{3/2}(\Gamma_l)}\le k_3\|u\|_{W^2(\Omega_{l\varkappa}^1)},
$$
where
$$
\Omega_{l\varkappa}^1=\{x=y+n_y t:\ y\in\Gamma_l,\ t\in(2\varkappa,5\varkappa)\}
$$
(see Fig.~\ref{fig6.1}).
\begin{figure}[ht]
{ \hfill\epsfxsize120mm\epsfbox{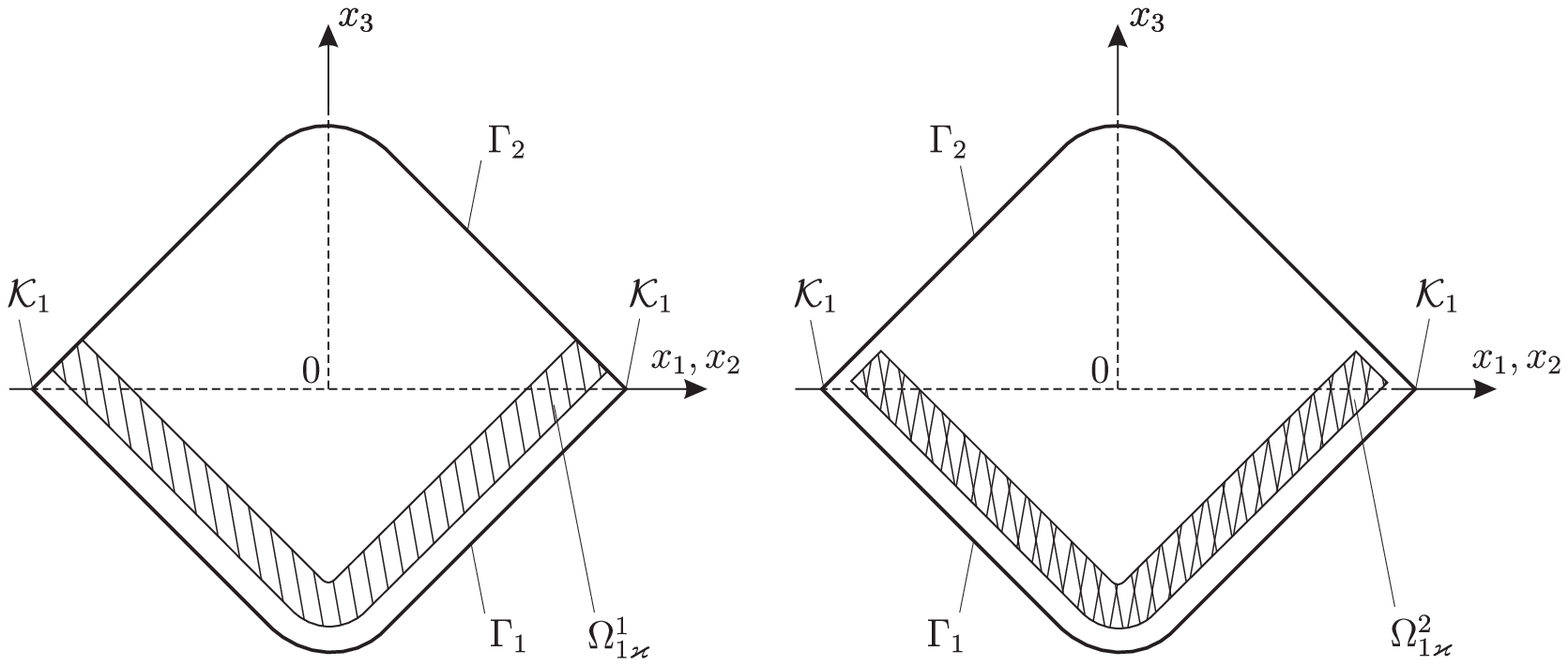}\hfill\ } \caption{Problem~\eqref{6.10},
\eqref{6.11}}\label{fig6.1}
\end{figure}
Since $\Omega_{l\varkappa}^1\subset Q\setminus\overline{\mathcal K_1^{2\varkappa}}$ and
the norms in $W^2(Q\setminus\overline{\mathcal K_1^{2\varkappa}})$ and
$H_a^2(Q\setminus\overline{\mathcal K_1^{2\varkappa}})$ are equivalent, the last
inequality yields
\begin{equation}\label{6.18}
\|B_l^2u\|_{H_a^{3/2}(\Gamma_l)}\le k_4\|u\|_{H_a^2(Q\setminus\overline{\mathcal
K_1^{2\varkappa}})}.
\end{equation}

Denote
$$
\Omega_{l\varkappa}^2=\{x=y+n_y t:\ y\in\Gamma_l\setminus\overline{\mathcal
K_1^\varkappa},\ t\in(2\varkappa,5\varkappa)\}.
$$
Since $\Omega_{l\varkappa}^2\subset Q_\varkappa$, similarly to~\eqref{6.18}, we obtain
\begin{equation}\label{6.19}
\|B_l^2u\|_{H_a^{3/2}(\Gamma_l\setminus\overline{\mathcal K_1^\varkappa})}\le
k_5\|u\|_{H_a^2(Q_\varkappa)}.
\end{equation}

It follows from~\eqref{6.18} and~\eqref{6.19} that the operators $B_l^2$ satisfy
Condition~\ref{cond6.2} for $a>1$.

We consider the linear bounded operators
$$
\mathbf L,\mathbf L^1: H_a^2(Q)\to H_a^0(Q)\times H_a^{3/2}(\Gamma_1)\times
H_a^{3/2}(\Gamma_2)
$$
given by
$$
\mathbf L u=\{-\Delta u,\ u|_{\Gamma_l}+B_l^1u+B_l^2u\},\qquad \mathbf L^1 u=\{-\Delta
u,\ u|_{\Gamma_l}+B_l^1u\}.
$$

It follows from Theorem~\ref{t5.1} that the operator $\mathbf L^1$ has the Fredholm
property for $0\le a\le 2$ and $\alpha_1,\alpha_2\in \mathbb R$ such that
$0<|\alpha_1+\alpha_2|<2$ and $\pi/4<\arctan\sqrt{4(\alpha_1+\alpha_2)^{-2}-1}$. By
Theorem~\ref{t6.3}, the operator $\mathbf L$ has the Fredholm property and $\ind\mathbf
L=\ind\mathbf L^1$ for $1< a\le 2$, $\alpha_1,\alpha_2\in \mathbb R$ such that
$0<|\alpha_1+\alpha_2|<2$ and $\pi/4<\arctan\sqrt{4(\alpha_1+\alpha_2)^{-2}-1}$, and a
continuous function $b(\lambda)$ satisfying relation~\eqref{6.9'}.
\end{example}

\section{Setting of Nonlocal Elliptic Problems with a Parameter. Model Operators}\label{sec7}

\subsection{}

In Secs.~\ref{sec7} and~\ref{sec8}, we prove the unique solvability of nonlocal elliptic
problems with a parameter $p=(p_1,\dots,p_d)\in\mathbb R^d$, where $d\ge 1$. Similarly to
the above, we first establish the unique solvability of model nonlocal problems with a
parameter in dihedral angles. Combining these results with those in~\cite{AV}
and~\cite{MP} and making use of a partition of unity, we will consider nonlocal problems
in bounded domains.

Let the domain $Q$, the transformations $\omega_{is}$, and the sets $\mathcal K_1$,
$K_2$, and $K$ be the same as in Sec.~\ref{sec6}.

To consider nonlocal problems with a parameter, we introduce norms for the weighted
spaces, depending on this parameter. First, we introduce the norms on the dihedral angle
\begin{equation}\label{eqTheta}
\Theta=\{x=(y,z)\in\mathbb R^n:\ d_1<\varphi<d_2,\ z\in\mathbb R^{n-2}\}
\end{equation}
and on the half-plane
\begin{equation}\label{eqGamma}
 \Gamma=\{x=(y,z)\in\mathbb R^n:\
\varphi=d',\ z\in\mathbb R^{n-2}\},\qquad d_1\le d'\le d_2.
\end{equation}
Consider the space $V_a^k(\Theta)=H_a^k(\Theta)\cap H_a^0(\Theta)$ with the norm
\begin{equation}\label{eqTripleNorm}
\n u \n_{V_a^k(\Theta)}=
(\|u\|_{H_{a}^k(\Theta)}^2+|p|^{2k}\|u\|_{H_{a}^0(\Theta)}^2)^{1/2}\quad (u\in
V_a^k(\Theta)),
\end{equation}
where $k\ge0$ is an integer.

For an integer $\nu\ge0$, we denote
$$
\|\psi\|_{H_{a}^\nu(\Gamma)}^2=\sum\limits_{|\alpha|\le\nu} \int\limits_\Gamma
r^{2(a-\nu+|\alpha|)}|D^\alpha \psi|^2\, d\Gamma.
$$
Consider the space $V_a^{k-1/2}(\Gamma)=H_a^{k-1/2}(\Gamma)\cap H_a^0(\Gamma)$ with the
norm
\begin{equation}\label{eqTripleNormTrace}
\n \psi \n_{V_a^{k-1/2}(\Gamma)}= (\|\psi\|_{H_a^{k-1/2}(\Gamma)}^2+|p|^{2(k-1/2)}\|
\psi\|_{H_{a}^0(\Gamma)}^2)^{1/2}\quad (\psi\in V_a^{k-1/2}(\Gamma),\ k\ge1).
\end{equation}

Now we introduce the norms for the domain $Q$ and for the manifolds $\Gamma_i$.
Set\footnote{We do not introduce the spaces $V_a^k(Q)=H_a^k(Q)\cap H_a^0(Q)$ and
$V_a^{k-1/2}(\Gamma_i)=H_a^{k-1/2}(\Gamma_i)\cap H_a^0(\Gamma_i)$ because they coincide
with $H_a^k(Q)$ and $H_a^{k-1/2}(\Gamma_i)$, respectively, in the case of bounded domain
$Q$.}
$$
\n u \n_{H_a^k(Q)}= (\|u\|_{H_{a}^k(Q)}^2+|p|^{2k}\|u\|_{H_{a}^0(Q)}^2)^{1/2}\qquad (u\in
H_a^k(Q)),
$$
$$
\n \psi \n_{H_a^{k-1/2}(\Gamma_i)}=
(\|\psi\|_{H_a^{k-1/2}(\Gamma_i)}^2+|p|^{2(k-1/2)}\|\psi\|_{H_{a}^0(\Gamma_i)}^2)^{1/2}\qquad
(\psi\in H_a^{k-1/2}(\Gamma_i),\ k\ge1),
$$
where
$$
\|\psi\|_{H_{a}^0(\Gamma_i)}^2= \int\limits_{\Gamma_i} \rho^{2a}| \psi|^2\, d\Gamma_i
$$
and $\rho(x)$ is the   function occurring in the definition of the spaces $H_a^k(Q)$.

We also set
$$
\n f \n_{\mathcal H_a^l(Q,\Gamma)}=\Big(\n f_0 \n_{H_a^l(Q)}^2+\sum\limits_{i,\mu}\n
f_{i\mu}\n_{H_a^{l+2m-m_{i\mu}-1/2}(\Gamma_i)}^2 \Big)^{1/2}
$$
for $f=\{f_0,f_{i\mu}\}\in \mathcal H_a^l(Q,\Gamma).$

\begin{lemma}\label{l7.1-new}
\begin{enumerate}
\item
For all $u\in V_a^{k}(\Theta)$, $k\ge2$ is an integer, $p\in\mathbb R^d$, and integer
$s$, $0<s<k$, we have
\begin{equation}\label{7.3-new'}
|p|^{k-s}\|u\|_{H_{a}^{s}(\Theta)}\le c_1 \n u\n_{V_{a}^{k}(\Theta)},
\end{equation}
where $\Theta$ is defined in~\eqref{eqTheta} and $c_1=c_1(k,s)>0$ does not depend on $u$
and $p.$
\item
For all $u\in H_a^{k}(Q)$, $k\ge2$ is an integer, $p\in\mathbb R^d$, and integer $s$,
$0<s<k$, we have
\begin{equation}\label{7.3-new}
|p|^{k-s}\|u\|_{H_{a}^{s}(Q)}\le c_2\n u\n_{H_{a}^{k}(Q)},
\end{equation}
where $c_2=c_2(k,s)>0$ does not depend on $u$ and $p$.
\end{enumerate}
\end{lemma}
\begin{proof}
First, we prove the interpolation inequality~\eqref{7.3-new'}.

Let $\{\xi_l\}_{l=-\infty}^{+\infty}$ be a partition of unity subordinated to the
covering of the angle $\theta=\{y\in\mathbb R^2:\ d_1<\varphi<d_2\}$ by the sets
$\theta_l=\{y\in\theta:\ 2^{l-1}<r<2^{l+1}\}$ such that
\begin{equation}\label{7.5-new}
|D^\alpha\xi_l(y)|\le k_\alpha 2^{-|\alpha|l},\qquad y\in\theta_l,\ l=0,\pm1,\pm2,\dots,
\end{equation}
where $k_\alpha>0$ does not depend on $l$.

Denote $\Theta_l=\theta_l\times\mathbb R^{n-2}$, $l=0,\pm1,\pm2,\dots$.

Using~\eqref{7.5-new} and the fact that $\supp\xi_l\cap\overline{\Theta_j}=\varnothing$
for $j\ne l-1,l,l+1$, one can easily verify  that
\begin{equation}\label{7.6-new}
\|u\|_{H_a^{s}(\Theta)}\approx\left(\sum\limits_{l=-\infty}^{+\infty}\|\xi_l
u\|^2_{H_a^{s}(\Theta_l)}\right)^{1/2}
\end{equation}
for $s=0,1,2,\dots$, where the symbol $\approx$ means the equivalence of the norms.

Using the theorem about  extension of functions from a domain with Lipshitz boundary to
$\mathbb R^n$ and applying the interpolation inequality for the Sobolev space
$W^k(\mathbb R^n)$ (see~\cite[Sec.~1]{AV}), we obtain
\begin{equation}\label{7.7-new}
q^{2(k-s)}\|v\|^2_{W^{s}(\Theta_0)}\le
k_1(\|v\|^2_{W^{k}(\Theta_0)}+q^{2k}\|v\|^2_{L_2(\Theta_0)})
\end{equation}
for all $v\in W^k(\Theta_0)$ and $q>0$,  where $k_1>0$ does not depend on $v$ and $q$.

We introduce the new variables $x'=2^{-l}x$. Using the equivalence of the
norms~\eqref{7.6-new} and interpolation inequality~\eqref{7.7-new} with $q=|p|2^l$ and
passing back to the variables $x=2^lx'$, we have
\begin{multline*}
|p|^{2(k-s)}\|u\|^2_{H_a^{s}(\Theta)}\le k_2|p|^{2(k-s)}
\sum\limits_{l=-\infty}^{+\infty}2^{(2(a-s)+n)l}\sum\limits_{|\alpha|\le
s}\,\int\limits_{\Theta_0}|D^\alpha_{x'}(\xi_l u)(x')|^2dx'  \\
\le k_3 \sum\limits_{l=-\infty}^{+\infty}2^{(2(a-k)+n)l}\left(\,\sum\limits_{|\alpha|\le
k}\,\int\limits_{\Theta_0}|D^\alpha_{x'}(\xi_l
u)(x')|^2dx'+|p|^{2k}2^{2lk}\,\int\limits_{\Theta_0}|(\xi_l u)(x')|^2dx'\right)  \\ \le
k_4 \sum\limits_{l=-\infty}^{+\infty}(\|\xi_l u\|_{H_a^k(\Theta_l)}^2+|p|^{2k}\|\xi_l
u\|^2_{H_a^0(\Theta_l)}),
\end{multline*}
where   $k_2,k_3,k_4>0$ do not depend on  $u$ and $p$. Combining this inequality
with~\eqref{7.6-new} yields~\eqref{7.3-new'}.

2. Note that the relation $u\in H_a^k(\Theta)$ implies that $u\in V_a^k(\Theta)$,
provided that $u$ is compactly supported. Therefore, using a partition of unity,
interpolation inequality~\eqref{7.3-new'}, and the interpolation inequality of the
kind~\eqref{7.7-new} for Sobolev spaces, we obtain~\eqref{7.3-new} for all $u\in
H_a^k(Q)$ and $p\in\mathbb R^d$.
\end{proof}

\begin{lemma}\label{lInterpWeightedTrace}
\begin{enumerate}
\item
For all $u\in V_a^{1}(\Theta)$ and $p\in\mathbb R^d$, we have
\begin{equation}\label{7.8-new'}
|p|^{1/2}\|u|_{\Gamma}\|_{H_a^0(\Gamma)}\le c_1 \n u\n_{V_a^1(\Theta)},
\end{equation}
where $\Theta$ and $\Gamma$ are defined in~\eqref{eqTheta} and~\eqref{eqGamma}
respectively, while $c_1>0$ does not depend on $u$ and $p$.
\item
For all $u\in H_a^{1}(Q)$ and $p\in\mathbb R^d$, we have
\begin{equation}\label{7.8-new}
|p|^{1/2}\|u|_{\Gamma_i}\|_{H_a^0(\Gamma_i)}\le c_2 \n u\n_{H_a^1(Q)},
\end{equation}
where $c_2>0$ does not depend on $u$ and $p$.
\end{enumerate}
\end{lemma}
\begin{proof}
Similarly to the proof of Lemma~\ref{l7.1-new}, it suffices to prove
inequality~\eqref{7.8-new'}.

Denote $\Gamma_l=\{x=(y,z)\in\Gamma:\ 2^{l-1}<r<2^{l+1}\}$, $l=0,\pm1,\pm2,\dots$. Let
$\xi_l$ be the same functions as in the proof of Lemma~\ref{l7.1-new}.

Similarly to~\eqref{7.6-new}, we have
\begin{equation}\label{7.10-new}
\|u|_\Gamma\|^2_{H_a^{0}(\Gamma)}\approx\sum\limits_{l=-\infty}^{+\infty}\|(\xi_l
u)|_{\Gamma_l}\|^2_{H_a^{0}(\Gamma_l)},
\end{equation}
where the symbol $\approx$ means the equivalence of the norms.

Using the theorem about  extension of functions from a domain with Lipshitz boundary to
$\mathbb R^n$ and applying the interpolation inequality for the Sobolev space
$W^1(\mathbb R^n)$ (see~\cite[Sec.~1]{AV}), we obtain
\begin{equation}\label{7.11-new}
q\|v|_{\Gamma_0}\|^2_{L_2(\Theta_0)}\le
k_1(\|v\|^2_{W^{1}(\Theta_0)}+q^{2}\|v\|^2_{L_2(\Theta_0)}),
\end{equation}
for all $v\in W^k(\Theta_0)$ and $q>0$, where $k_1>0$ does not depend on $v$ and $q$.

We introduce the new variables $x'=2^{-l}x$. Using the equivalence of the
norms~\eqref{7.10-new} and the interpolation inequality~\eqref{7.11-new} with $q=|p| 2^l$
and passing back to the variables $x=2^lx'$, we have
\begin{multline*}
|p|\cdot\|u|_\Gamma\|^2_{H_a^{0}(\Gamma)}\le k_2|p|
\sum\limits_{l=-\infty}^{+\infty}2^{(2a+n-1)l}\,\int\limits_{\Gamma_0}|(\xi_l u)(x')|_{\Gamma_0}|^2d\Gamma_0  \\
\le k_3 \sum\limits_{l=-\infty}^{+\infty}2^{(2a+n-2)l}\left\{\sum\limits_{|\alpha|\le
1}\,\int\limits_{\Theta_0}|D^\alpha_{x'}(\xi_l
u)(x')|^2dx'+|p|^{2}2^{2l}\,\int\limits_{\Theta_0}|(\xi_l u)(x')|^2dx'\right\}  \\ \le
k_4 \sum\limits_{l=-\infty}^{+\infty}(\|\xi_l u\|_{H_a^1(\Theta_l)}^2+|p|^{2}\|\xi_l
u\|^2_{H_a^0(\Theta_l)}),
\end{multline*}
where  $k_2,k_3,k_4>0$ do not depend on  $u$ and $p$. Combining this inequality
with~\eqref{7.6-new} yields~\eqref{7.8-new'}.
\end{proof}

In particular, it follows from Lemmas~\ref{l7.1-new} and~\ref{lInterpWeightedTrace} that
\begin{align}
\n u|_{\Gamma} \n_{V_a^{k-1/2}(\Gamma)}&\le c \n u
\n_{V_a^{k}(\Theta)},\label{eqBoundedTrace'}\\
 \n u|_{\Gamma_i}
\n_{H_a^{k-1/2}(\Gamma_i)}&\le C \n u \n_{H_a^{k}(Q)},\label{eqBoundedTrace}
\end{align}
where $c,C>0$ do not depend on $u$ and $p$.

\begin{lemma}\label{lInterpWeightedTraceGamma}
For all $\psi\in V_a^{k-1/2}(\Gamma)$, $k\ge2$ is an integer, $p\in\mathbb R^d$, and
integer $s$, $0<s<k$, we have
\begin{equation}\label{eqInterpWeightedTraceGamma1}
|p|^{k-s-1/2}\|\psi\|_{H_a^s(\Gamma)}\le c \n \psi\n_{V_a^{k-1/2}(\Gamma)},
\end{equation}
where  $\Gamma$ is defined by~\eqref{eqGamma} and $c>0$ does not depend on $\psi$ and
$p$.
\end{lemma}
\begin{proof}
The proof is based on the following interpolation inequality for Sobolev spaces in
$\mathbb R^{n-1}$ (see~\cite[Sec.~1]{AV}):
\begin{equation}\label{eqInterpWeightedTraceGamma2}
q^{2(k-s-1/2)}\|v\|^2_{W^s(\mathbb R^{n-1})}\le k_1(\|v\|^2_{W^{k-1/2}(\mathbb
R^{n-1})}+q^{2(k-1/2)}\|v\|^2_{L_2(\mathbb R^{n-1})})
\end{equation}
for all $v\in W^{k-1/2}(\mathbb R^{n-1})$,  $q>0$, and integer $s$, $0<s<k$, where
$k_1>0$ does not depend on $v$ and $q$. We will use the following equivalent norm in the
space $W^{k-1/2}(\mathbb R^{n-1})$:
\begin{equation}\label{eqInterpWeightedTraceGamma3}
\left(\sum\limits_{|\alpha|=k-1}\,\int\limits_{\mathbb R^{n-1}}\int\limits_{\mathbb
R^{n-1}}|D^\alpha v(x_1)-D^\alpha v(x_2)|^2\dfrac{dx_1 dx_2}{|x_1-x_2|^n}+
\sum\limits_{|\alpha|\le k-1}\,\int\limits_{\mathbb R^{n-1}}|D^\alpha v(x)|^2
dx\right)^{1/2}
\end{equation}
(see, e.g., \cite{Slob}).

Denote $\Gamma_l=\{x=(y,z)\in\Gamma:\ 2^{l-1}<r<2^{l+1}\}$, $l=0,\pm1,\pm2,\dots$. Let
$\xi_l$ be the same functions as in the proof of Lemmas~\ref{l7.1-new}
and~\ref{lInterpWeightedTrace}.

Using relations~\eqref{7.5-new} and the fact that
$\supp\xi_l\cap\overline{\Theta_j}=\varnothing$ for $j\ne l-1,l,l+1$, one can easily
verify that
\begin{equation}\label{eqInterpWeightedTraceGamma3'}
\|\psi\|_{H_a^s(\Gamma)}\approx
\left(\sum\limits_{l=-\infty}^{+\infty}\|\xi_l\psi\|^2_{H_a^{s}(\Gamma_l)}\right)^{1/2}
\end{equation}
for $s=0,1/2,1,3/2,\dots$, where the symbol $\approx$ means the equivalence of the norms
(in particular, see Lemma~1.1 in~\cite{MP} for noninteger $s$).

Further, introducing the new variables $x'=2^{-l}x$ and using
equivalence~\eqref{eqInterpWeightedTraceGamma3'}, the interpolation
inequality~\eqref{eqInterpWeightedTraceGamma2} with $q=|p|2^l$, and the equivalent
norm~\eqref{eqInterpWeightedTraceGamma3} in the Sobolev space (which is possible because
the functions $\xi_l$ are compactly supported),  we obtain
\begin{multline*}
|p|^{2(k-s-1/2)}\|\psi\|^2_{H_a^s(\Gamma)}\\
\le k_2\sum\limits_{l=-\infty}^{+\infty}2^{2l(a-k+1/2)}2^{l(n-1)}\cdot |p|^{2(k-s-1/2)}
2^{2l(k-s-1/2)}
\sum\limits_{|\alpha|\le s}\, \int\limits_{\Gamma_0}|D^\alpha_{x'}(\xi_l\psi)(x')|^2dx' \\
\le
k_3\sum\limits_{l=-\infty}^{+\infty}2^{2l(a-k+1/2)}2^{l(n-1)}\\
\times\left(\sum\limits_{|\alpha|=k-1}\,\int\limits_{
\Gamma}\int\limits_{\Gamma}|D^\alpha_{x'} (\xi_l\psi)(x_1')-D^\alpha_{x'}
(\xi_l\psi)(x_2')|^2\dfrac{dx_1' dx_2'}{|x_1'-x_2'|^n}\right.\\
+ \left.\sum\limits_{|\alpha|\le
k-1}\,\int\limits_{\Gamma_0}|D^\alpha_{x'}(\xi_l\psi)(x')|^2 d x'+
|p|^{2(k-1/2)}2^{2l(k-1/2)}\int\limits_{\Gamma_0}|(\xi_l\psi)(x')|^2
d x'\right)\\
\le k_4
\sum\limits_{l=-\infty}^{+\infty}2^{2l(a-k+1/2)}2^{l(n-1)}\\
\times\left(\sum\limits_{|\alpha|=k-1}\,\int\limits_{
\Gamma}\int\limits_{\Gamma}\big||y_1'|^aD^\alpha_{x'} (\xi_l\psi)(x_1')-|y_2'|^a
D^\alpha_{x'}
(\xi_l\psi)(x_2')\big|^2\dfrac{dx_1' dx_2'}{|x_1'-x_2'|^n}\right.\\
+ \sum\limits_{|\alpha|\le
k-1}\,\int\limits_{\Gamma_0}|y'|^{2(a+|\alpha|-k+1/2)}|D^\alpha_{x'}(\xi_l\psi)(x')|^2 d
x'\\
+ \left. |p|^{2(k-1/2)}2^{2l(k-1/2)}\int\limits_{\Gamma_0}|y'|^{2a}|(\xi_l\psi)(x')|^2 d
x'\right),
\end{multline*}
where $k_2,k_3,{\dots}>0$ do not depend on $\psi$ and $p$, $x_i'=(y_i',z_i')\in\Gamma$,
$y_i'\in\mathbb R^2$, and $z_i'\in\mathbb R^{n-2}$, $i=1,2$. In the last inequality, we
have also used the fact that $\supp\xi_l\cap\overline{Q_j}=\varnothing$ for $j\ne
l-1,l,l+1$. Passing back to the variables $x=2^l x'$, we have
\begin{multline*}
|p|^{2(k-s-1/2)}\|\psi\|^2_{H_a^s(\Gamma)}\\
\le k_5 \sum\limits_{l=-\infty}^{+\infty}\left(\sum\limits_{|\alpha|=k-1}\,\int\limits_{
\Gamma}\int\limits_{\Gamma}\big||y_1|^a D^\alpha (\xi_l\psi)(x_1)- |y_2|^a D^\alpha
(\xi_l\psi)(x_2)\big|^2\dfrac{dx_1 dx_2}{|x_1-x_2|^n}\right.\\
+ \left.\sum\limits_{|\alpha|\le
k-1}\,\int\limits_{\Gamma_l}|y|^{2(a+|\alpha|-k+1/2)}|D^\alpha(\xi_l\psi)(x)|^2 dx+
|p|^{2(k-1/2)}\int\limits_{\Gamma_l}|y|^{2a}|(\xi_l\psi)(x)|^2 dx\right),
\end{multline*}
where $x_i=(y_i,z_i)\in\Gamma$, $y_i\in\mathbb R^2$, and $z_i\in\mathbb R^{n-2}$,
$i=1,2$. It follows from this inequality and from Lemma~1.3 in~\cite{MP} (about the
equivalent norms in the weighted trace spaces) that
\begin{equation}\label{eqInterpWeightedTraceGamma4}
|p|^{2(k-s-1/2)}\|\psi\|^2_{H_a^s(\Gamma)}\le k_6
\sum\limits_{l=-\infty}^{+\infty}(\|\xi_l\psi\|_{H_a^{k-1/2}(\Gamma_l)}^2+
|p|^{2(k-1/2)}\|\xi_l\psi\|^2_{H_a^0(\Gamma_l)}).
\end{equation}
Combining~\eqref{eqInterpWeightedTraceGamma4} with  the equivalence of the
norms~\eqref{eqInterpWeightedTraceGamma3'}, we
obtain~\eqref{eqInterpWeightedTraceGamma1}.
\end{proof}

\subsection{}

Consider the differential operators
$$
A^0(p)\equiv A^0(x,D,p)=\sum\limits_{|\alpha|+|\beta|=2m}a_{\alpha\beta}(x)p^\beta
D^\alpha,
$$
$$
B_{i\mu s}^0(p)\equiv B_{i\mu s}^0(x,D,p)=\sum\limits_{|\alpha|+|\beta|=m_{i\mu}}b_{i\mu
s \alpha\beta}(x)p^\beta D^\alpha,
$$
where $a_{\alpha\beta}, b_{i\mu s \alpha\beta}\in C^\infty(\mathbb R^n)$ are
complex-valued functions ($i=1,\dots,N_0$; $\mu=1,\dots,m$; $s=0,\dots, S_i$),
$\beta=(\beta_1,\dots,\beta_d)$, $|\beta|=|\beta_1|+\dots+|\beta_d|$,
$p^\beta=p_1^{\beta_1}\dots p_d^{\beta_d}$, and $m_{i\mu}\le 2m-1$.

We study the following nonlocal elliptic problem:
\begin{equation}\label{7.A}
A(p)u\equiv A^0(p)u+A^1(p)u=f_0(x),\qquad x\in Q,
\end{equation}
\begin{equation}\label{7.B}
B_{i\mu}(p)u\equiv \sum\limits_{j=0}^3 B_{i\mu}^j(p)u=f_{i\mu}(x),\qquad x\in \Gamma_i;\
i=1,\dots,N_0;\ \mu=1,\dots, m.
\end{equation}
Here
$$
B_{i\mu}^0(p)u=B_{i\mu0}^0(p)u|_{\Gamma_i},\qquad
B_{i\mu}^1(p)u=\sum\limits_{s=1}^{S_i}\big( B_{i\mu s}^0(x,D,p)(\xi
u)\big)\big(\omega_{is}(x)\big)\big|_{\Gamma_i},
$$
the function $\xi$ and the transformations $\omega_{is}$ are the same as in
Sec.~\ref{sec6}. In particular, we assume that Conditions~\ref{cond1.3} and~\ref{cond1.4}
hold.

Introduce the variable $t=(t_1,\dots,t_d)\in\mathbb R^d$ and formally replace the
expressions $p^\beta$ in the operators $A^0(x,D_x,p)$ and $B_{i\mu 0}^0(x,D_x,p)$ by the
differential operators $D_t^\beta$. Assume that the following conditions hold
(cf.~\cite{AV,MP}).
\begin{condition}\label{cond7.1}
The operator $A^0(x,D_x,D_t)$ is properly elliptic for $(x,t)\in\overline{Q}\times\mathbb
R^d$.
\end{condition}
\begin{condition}\label{cond7.2}
The system $\{B_{i\mu 0}^0(x,D_x,D_t)\}_{\mu=1}^m$ covers the operator $A^0(x,D_x,D_t)$
and is normal for all $i=1,\dots,N_0$ and $(x,t)\in\overline{\Gamma_i}\times\mathbb R^d$.
\end{condition}

We also assume that the following conditions for the operators $A^1(p)$, $B_{i\mu}^2(p)$,
and $B_{i\mu}^3(p)$ hold.

\begin{condition}[smallness of perturbations]\label{cond7.3}
We have
$$
\n A^1(p) u\n_{H_a^{l}(Q)}\le c_1 \n u\n_{H_a^{l+2m-1}(Q)},
$$
$$
\n B_{i\mu}^3(p) u\n_{H_a^{l+2m-m_{i\mu}-1/2}(\Gamma_i)}\le c_2 \n u\n_{H_a^{l+2m-1}(Q)},
$$
where $i=1,\dots,N_0,$ $\mu=1,\dots,m$, and $c_1,c_2>0$ do not depend on $u$ and $p$.
\end{condition}

\begin{condition}[separability from the conjugation points]\label{cond7.4}
There exist numbers $\sigma>0$ and $\varkappa_1>\varkappa_2>0$ such that
\begin{equation}\label{eqSepar1}
\n B_{i\mu}^2(p)u\n_{H_a^{l+2m-m_{i\mu}-1/2}(\Gamma_i)}\le c_1\n
u\n_{H_a^{l+2m}(Q\setminus\overline{\mathcal K_1^{\varkappa_1}})}
\end{equation}
for all $u\in H_a^{l+2m}(Q\setminus\overline{\mathcal K_1^{\varkappa_1}})$ and
\begin{equation}\label{eqSepar2}
\n B_{i\mu}^2(p)u\n_{H_a^{l+2m-m_{i\mu}-1/2}(\Gamma_i\setminus\overline{\mathcal
K_1^{\varkappa_2}})}\le c_2\n u \n_{H_a^{l+2m}(Q_\sigma)}
\end{equation}
for all $u\in H_a^{l+2m}(Q_\sigma);$ here $i=1,\dots,N_0;$ $\mu=1,\dots,m;$ $c_1,c_2>0$
do not depend on $u$ and $p$.
\end{condition}

\begin{remark} 1. It follows from Condition~\ref{cond7.3} and from
Lemma~\ref{l7.1-new} that
$$
\n A^1(p) u\n_{H_a^{l}(Q)}\le c_1 |p|^{-1}\n u\n_{H_a^{l+2m}(Q)},
$$
$$
\n B_{i\mu}^3(p) u\n_{H_a^{l+2m-m_{i\mu}-1/2}(\Gamma_i)}\le c_2 |p|^{-1} \n
u\n_{H_a^{l+2m}(Q)}.
$$
Therefore, the norms of the operators $A^1(p)$ and $B_{i\mu}^3(p)$ are small, provided
that $|p|$ is large.

2. Condition~\ref{cond7.4} is an analog of Condition~\ref{cond6.2}.
\end{remark}

\begin{remark}\label{remRepresentationParam}
Let the transformations $\omega_{is}$ and the set $K$ be the same as in
Secs.~\ref{sec1}--\ref{sec5}. Consider the problem
$$
\sum\limits_{|\alpha|+|\beta|\le 2m}a_{\alpha\beta}(x)p^\beta D^\alpha u(x)=f_0(x),\qquad
x\in Q,
$$
$$
\sum\limits_{|\alpha|+|\beta|\le m_{i\mu}}\sum\limits_{s=0}^{S_i} b_{i\mu
s\alpha\beta}(\omega_{is}(x))p^\beta(D^\alpha u
)(\omega_{is}(x))|_{\Gamma_i}=f_{i\mu}(x), $$ $$ x\in\Gamma_i;\ i=1,\dots,N_0;\
\mu=1,\dots,m.
$$

This problem can be represented in the form~\eqref{7.A}, \eqref{7.B}. Indeed, set
$$
A^0(p)=\sum\limits_{|\alpha|+|\beta|=2m}a_{\alpha\beta}(x)p^\beta D^\alpha,\qquad
A^1(p)=\sum\limits_{|\alpha|+|\beta|\le 2m-1}a_{\alpha\beta}(x)p^\beta D^\alpha,
$$
$$
B_{i\mu}^0(p)u=\sum\limits_{|\alpha|+|\beta|=m_{i\mu}}b_{i\mu 0 \alpha\beta}(x)p^\beta
D^\alpha u|_{\Gamma_i},
$$
$$
B_{i\mu}^1(p)u=\sum\limits_{|\alpha|+|\beta|=m_{i\mu}}\sum\limits_{s=1}^{S_i} b_{i\mu
s\alpha\beta}(\omega_{is}(x))p^\beta(D^\alpha (\xi u))(\omega_{is}(x))|_{\Gamma_i},
$$
$$
B_{i\mu}^2(p)u=\sum\limits_{|\alpha|+|\beta|=m_{i\mu}}\sum\limits_{s=1}^{S_i} b_{i\mu
s\alpha\beta}(\omega_{is}(x))p^\beta(D^\alpha ((1-\xi) u))(\omega_{is}(x))|_{\Gamma_i},
$$
$$
B_{i\mu}^3(p)u=\sum\limits_{|\alpha|+|\beta|\le m_{i\mu}-1}\sum\limits_{s=0}^{S_i}
b_{i\mu s\alpha\beta}(\omega_{is}(x))p^\beta(D^\alpha u)(\omega_{is}(x))|_{\Gamma_i}.
$$
Clearly, the operator $A^1(p)$ satisfies Condition~\ref{cond7.3}.

Let $s=0,\dots S_i$ and $|\alpha|+|\beta|\le m_{i\mu}-1$. Denote by $u_1$ an extension of
the function $u$ to $Q\cup\omega_{is}(\Omega_i)$, defined by Lemma~\ref{l4.2} and
satisfying the inequalities
\begin{equation}\label{eqRepresentationParam1}
\|u_1\|_{H_a^{\nu}(Q\cup\omega_{is}(\Omega_i))}\le k_1\|u\|_{H_a^{\nu}(Q)},\qquad
\nu=0,\dots,l+2m-1.
\end{equation}
Clearly,
$$
b_{i\mu s\alpha\beta}(\omega_{is}(x))p^\beta(D^\alpha
u)(\omega_{is}(x))|_{\Gamma_i}=b_{i\mu s\alpha\beta}(\omega_{is}(x))p^\beta(D^\alpha
u_1)(\omega_{is}(x))|_{\Gamma_i}.
$$
Therefore, using~\eqref{eqBoundedTrace} and~\eqref{eqRepresentationParam1}, we have
\begin{multline*}
\n b_{i\mu s\alpha\beta}(\omega_{is}(x))p^\beta(D^\alpha u)(\omega_{is}(x))|_{\Gamma_i}
\n_{H_a^{l+2m-m_{i\mu}-1/2}(\Gamma_i)}\\
\le k_2 \n u_1\n_{H_a^{l+2m-1}(\omega_{is}(\Omega_i))}\le k_3\n u\n_{H_a^{l+2m-1}(Q)}
\end{multline*}
for $\alpha$ and $\beta$ such that $|\alpha|+|\beta|\le m_{i\mu}-1$. This inequality
proves that the operators $B_{i\mu}^3(p)$ satisfy Condition~\ref{cond7.3}.

To show that the operators $B_{i\mu}^2(p)$ satisfy Condition~\ref{cond7.4}, one must
repeat the proof of Lemma~\ref{l4.6} with the norms $\|\cdot\|$ replaced by the norms
$\n\cdot\n$, taking into account inequality~\eqref{eqBoundedTrace}.
\end{remark}

We consider the linear bounded operators $\mathbf L^0(p), \mathbf L^1(p), \mathbf
L(p):H_a^{l+2m}(Q)\to\mathcal H_a^l(Q,\Gamma)$ given by
$$
\mathbf L^0(p)=\{A(p),B_{i\mu}^0(p)\},\quad \mathbf
L^1(p)=\{A(p),B_{i\mu}^0(p)+B_{i\mu}^1(p)\},\quad \mathbf L(p)=\{A(p),B_{i\mu}(p)\}.
$$

The invertibility  of the operator $\mathbf L^0(p)$ was proved in~\cite{MP}. Our aim is
to study the operator  $\mathbf L^1(p)$ and then $\mathbf L(p)$. First, we will consider
model problems with a parameter corresponding to the points of the sets $\mathcal K_1$
and $K_2$.

\subsection{}

Fix a point $g\in\mathcal K_1$. Using the reasoning similar to that in Sec.~\ref{sec1},
we arrive at the following model problem (cf.~\eqref{1.9}, \eqref{1.10}):
\begin{equation}\label{7.AModelx'}
A_j(x,D_{y},D_{z},p)v_j(x)=f_j(x),\qquad x\in \Theta_j;\ j=1,\dots,N,
\end{equation}
\begin{equation}\label{7.BModelx'}
\begin{aligned}
\sum\limits_{k=1}^N\sum\limits_{s=0}^{S_{j\rho k}}(B_{j\rho\mu
ks}(x,D_{y},D_{z},p)v_k)(\mathcal G_{j\rho
ks}y,z)|_{\Gamma_{j\rho}}=f_{j\rho\mu}(x),\qquad
x\in\Gamma_{j\rho};\\
j=1,\dots,N;\ \rho=1,2;\ \mu=1,\dots,m,
\end{aligned}
\end{equation}
where $A_j(x,D_{y},D_{z},p)$ and $B_{j\rho\mu ks}(x,D_{y},D_{z},p)$ are differential
operators of order $2m$ and $m_{j\rho\mu}$, respectively, with the parameter $p$, having
variable coefficients of class $C^\infty$, while $\Theta_j$, $\Gamma_{j\rho}$, and
$\mathcal G_{j\rho ks}$ are the same as in~\eqref{1.9}, \eqref{1.10}.

Introduce the spaces of vector-valued functions
$$
\mathcal V_a^k(\Theta)=\prod\limits_{j=1}^N V_a^k(\Theta_j),\quad \mathcal
V_a^l(\Theta,\Gamma)=\mathcal
V_a^l(\Theta)\times\prod\limits_{j=1}^N\prod\limits_{\rho=1,2}\prod\limits_{\mu=1}^m
V_a^{l+2m-m_{j\rho\mu}-1/2}(\Gamma_{j\rho}),
$$
where $m_{j\rho\mu}$ is the order of the operator $B_{j\rho\mu ks}(x,D_{y},D_{z},p)$.
Consider the linear bounded operator $\mathcal L_g(p): \mathcal
V_a^{l+2m}(\Theta)\to\mathcal V_a^l(\Theta,\Gamma)$ given by
\begin{equation}\label{eqLg(p)}
\mathcal L_g(p) v=\Big\{A_j(D_{y},D_{z},p)v_j(y,z),\ \sum\limits_{k,s}(B_{j\rho\mu
ks}(D_{y},D_{z},p)v_k)(\mathcal G_{j\rho ks}y,z)|_{\Gamma_{j\rho}}\Big\},
\end{equation}
where $A_j(D_{y},D_{z},p)$ and $B_{j\rho\mu ks}(D_{y},D_{z},p)$ are principal homogeneous
parts\footnote{In this section, the notion ``principal homogeneous part'' takes into
account the parameter $p$, e.g., the operator $A_j(D_{y},D_{z},p)$ consists of the terms
$a_{j\alpha\beta}(x)p^\beta D^\alpha$, where $|\alpha|+|\beta|=2m$.} of the operators
$A_j(0,D_{y},D_{z},p)$ and $B_{j\rho\mu ks}(0,D_{y},D_{z},p)$, respectively. Clearly, if
we replace $p^\beta$ by $D_t^\beta$, then each of the obtained operators
$A_j(D_{y},D_{z},D_t)$ will be properly elliptic for
$(x,t)\in\overline{\Theta_{j}}\times\mathbb R^d$, while the system $\{B_{j\rho\mu
j0}(D_{y},D_{z},D_t)\}_{\mu=1}^m$ will cover the operator $A_j(D_{y},D_{z},D_t)$ and be
normal for all $(x,t)\in\Gamma_{j\rho}\times\mathbb R^d$, $j=1,\dots,N$, and  $\rho=1,2$.

We also set
$$
\mathcal L_g'(p) v=\Big\{A_j^0(x,D_{y},D_{z},p)v_j(y,z),\ \sum\limits_{k,s}(B_{j\rho\mu
ks}^0(x,D_{y},D_{z},p) v_k)(\mathcal G_{j\rho ks}y,z)|_{\Gamma_{j\rho}}\Big\},
$$
where $A_j^0(x,D_{y},D_{z},p)$ and $B_{j\rho\mu ks}^0(x,D_{y},D_{z},p)$ are principal
homogeneous parts of the operators $A_j(x,D_{y},D_{z},p)$ and $B_{j\rho\mu
ks}(x,D_{y},D_{z},p)$, respectively.

Further, we set
$$
\mathcal L_g(\eta,p) V=\Big\{A_j(D_{y},\eta,p)V_j(y),\ \sum\limits_{k,s}(B_{j\rho\mu
ks}(D_{y},\eta,p)V_k)(\mathcal G_{j\rho ks}y)|_{\Gamma_{j\rho}}\Big\},\qquad
\eta\in\mathbb R^{n-2}.
$$
Replacing $(\eta,p)$ by $\omega=(\eta,p)/|(\eta,p)|$, we obtain the bounded operator
$$\mathcal L_g(\omega):\mathcal E_a^{l+2m}(\theta)\to\mathcal E_a^l(\theta,\gamma)$$ given
by
\begin{equation}\label{eqLgOmegaP}
\mathcal L_g(\omega) V=\Big\{A_j(D_{y},\omega)V_j(y),\ \sum\limits_{k,s}(B_{j\rho\mu
ks}(D_{y},\omega)V_k)(\mathcal G_{j\rho ks}y)|_{\Gamma_{j\rho}}\Big\},\qquad \omega\in
S^{n+d-3}.
\end{equation}

Finally, we consider the analytic operator-valued function
$$\hat{\mathcal L}_g(\lambda): \mathcal
W^{l+2m}(d_1,d_2)\to\mathcal W^l[d_1,d_2]$$ given by~\eqref{2.4}.

In this subsection, we prove that the absence of eigenvalues of $\hat{\mathcal
L}_g(\lambda)$ on the line $\Im\lambda=a+1-l-2m$ and the triviality of the kernel and
cokernel of $\mathcal L_g(\omega)$ guarantees the existence of the inverse operators
$\mathcal L_g^{-1}(p)$ for $p\in\mathbb R^d\setminus\{0\}$, uniformly bounded in the
corresponding norms $\n\cdot\n$. We introduce these norms by setting
$$
\n u\n_{\mathcal V_a^{k}(\Theta)}=\Big(\sum\limits_j \n u\n_{
V_a^{k}(\Theta_j)}^2\Big)^{1/2},
$$
$$
\n f \n_{\mathcal V_a^l(\Theta,\Gamma)}=\Big(\sum\limits_j\n f_j\n_{
V_a^{l}(\Theta_j)}^2+\sum\limits_{j,\rho,\mu}\n f_{j\rho\mu}\n_{
V_a^{l+2m-m_{j\rho\mu}-1/2}(\Gamma_j)}^2\Big)^{1/2},\qquad f=\{f_j,f_{j\rho\mu}\}.
$$

\begin{theorem}\label{tLpIsomH}
Let Conditions~$\ref{cond7.1}$, $\ref{cond7.2}$, $\ref{cond1.3}$, and $\ref{cond1.4}$
hold. Assume that the line $\Im\lambda=a+1-l-2m$ contains no eigenvalues of
$\hat{\mathcal L}_g(\lambda)$ and $\dim\mathcal N(\mathcal L_g(\omega))=\codim\mathcal
R(\mathcal L_g(\omega))=0$ for any $\omega\in S^{n+d-3}$. Then the operator $\mathcal
L_g(p)$ is an isomorphism for $p\in\mathbb R^d\setminus\{0\}$ and
\begin{equation}\label{eqLpIsomH0}
\n u\n_{\mathcal V_a^{l+2m}(\Theta)}\le c\n \mathcal L_g(p)u\n_{\mathcal
V_a^l(\Theta,\Gamma)},\qquad p\in\mathbb R^d\setminus\{0\},
\end{equation}
where $c>0$ does not depend on $u$ and $p$.
\end{theorem}

To prove Theorem~\ref{tLpIsomH}, we preliminary consider the bounded operator $$\mathcal
L_g(p): \mathcal E_a^{l+2m}(\Theta)\to\mathcal E_a^l(\Theta,\Gamma)$$ given
by~\eqref{eqLg(p)} for $p\in S^{d-1}$, where
$$
\mathcal E_a^k(\Theta)=\prod\limits_{j=1}^N E_a^k(\Theta_j),\quad \mathcal
E_a^l(\Theta,\Gamma)=\mathcal E_a^l(\Theta)\times\prod\limits_{j=1}^N
\prod\limits_{\rho=1,2}\prod\limits_{\mu=1}^m
E_a^{l+2m-m_{j\rho\mu}-1/2}(\Gamma_{j\rho}),
$$
while $E_a^k(\Theta_j)$ is the completion of the set
$C_0^\infty(\overline{\Theta_j}\setminus\{0\})$ with respect to the norm
$$
\|v\|_{E_a^k(\Theta_j)}=\bigg(\sum\limits_{|\alpha|\le k}\,\int\limits_{\Theta_j}
|y|^{2a}(|y|^{2(|\alpha|-k)}+1)|D_x^\alpha v(x)|^2\, dx\bigg)^{1/2}
$$
and $E_a^{k-1/2}(\Gamma_{j\rho})$ ($k\ge1$ is an integer) is the space of traces on
$\Gamma_{j\rho}$ with the norm
$$
\|\psi\|_{E_a^{k-1/2}(\Gamma_{j\rho})}=\inf\|v\|_{E_a^k(\Theta_j)}\qquad (v\in
E_a^k(\Theta_j):\ v|_{\Gamma_{j\rho}}=\psi).
$$

\begin{lemma}\label{lLpEReductHomog}
Let Conditions~$\ref{cond7.2}$, $\ref{cond1.3}$, and $\ref{cond1.4}$ hold, and let
$f_{j\rho\mu}\in E_a^{l+2m-m_{j\rho\mu}-1/2}(\Gamma_{j\rho})$. Then there exists a
function $u\in \mathcal E_a^{l+2m}(\Theta)$ such that
$$
\sum\limits_{k,s}(B_{j\rho\mu ks}(D_{y},D_{z},p)u_k)(\mathcal G_{j\rho
ks}y,z)|_{\Gamma_{j\rho}}=f_{j\rho\mu},\qquad p\in S^{d-1},
$$
$$
\|u\|_{\mathcal E_a^{l+2m}(\Theta)}\le
c\sum\limits_{j,\rho,\mu}\|f_{j\rho\mu}\|_{E_a^{l+2m-m_{j\rho\mu}-1/2}(\Gamma_{j\rho})},\qquad
p\in S^{d-1},
$$
where $c>0$ does not depend on $u$ and $p$.
\end{lemma}
\begin{proof}
By Lemma~$9.2'$ in~\cite{MP}, there exists a function $v=(v_1,\dots,v_N)\in \mathcal
E_a^{l+2m}(\Theta)$ such that $ B_{j\rho\mu
j0}(D_{y},D_{z},p)v_j|_{\Gamma_{j\rho}}=f_{j\rho\mu}$ and
$$
\|v\|_{\mathcal E_a^{l+2m}(\Theta)}\le
k_1\sum\limits_{j,\rho,\mu}\|f_{j\rho\mu}\|_{E_a^{l+2m-m_{j\rho\mu}-1/2}(\Gamma_{j\rho})},\qquad
p\in S^{d-1}.
$$
Let $d_0$ be the number defined in~\eqref{eqd0}. Consider functions $\xi_k\in
C^\infty(\mathbb R)$ depending on the polar angle $\varphi$ of the point $y\in\mathbb
R^2$, such that $\xi_k(\varphi)=1$ for $d_{k1}\le\varphi\le d_{k1}+d_0/2$ and
$d_{k2}-d_0/2\le\varphi\le d_{k2}$ and $\xi_k(\varphi)=0$ for $d_{k1}+d_0\le\varphi\le
d_{k2}-d_0$.

Since the operator of multiplication by $\xi_j$ is bounded in $E_a^{l+2m}(\Theta_j)$, we
see that the function $u=(\xi_1v_1,\dots,\xi_Nv_N)$ is the desired one.
\end{proof}

\begin{lemma}\label{lLpIsomE}
Let the conditions of Theorem~$\ref{tLpIsomH}$ be fulfilled. Then the operator $\mathcal
L_g(p): \mathcal E_a^{l+2m}(\Theta)\to\mathcal E_a^l(\Theta,\Gamma)$ is an isomorphism
for $p\in S^{d-1}$ and
$$
\|u\|_{\mathcal E_a^{l+2m}(\Theta)}\le c\| \mathcal L_g(p)u\|_{\mathcal
E_a^l(\Theta,\Gamma)},\qquad p\in S^{d-1},
$$
where $c>0$ does not depend on $u$ and $p$.
\end{lemma}
\begin{proof}
1. Due to Lemma~\ref{lLpEReductHomog}, it suffices to prove the unique solvability of the
problem
\begin{equation}\label{eqLpIsomE1}
A_j(D_{y},D_{z},p)u_j(y,z)=f_j(y,z),\qquad (y,z)\in\Theta_j,
\end{equation}
\begin{equation}\label{eqLpIsomE2}
\sum\limits_{k,s}(B_{j\rho\mu ks}(D_{y},D_{z},p)u_k)(\mathcal G_{j\rho
ks}y,z)|_{\Gamma_{j\rho}}=0
\end{equation}
and to show that
$$
\|u\|_{\mathcal E_a^{l+2m}(\Theta)}\le k_1\| \{f_j\}\|_{\mathcal E_a^l(\Theta)},\qquad
p\in S^{d-1},
$$
where $k_1>0$ does not depend on $u$ and $p$.

2. Making the Fourier transform with respect to $z$, we see that
problem~\eqref{eqLpIsomE1}, \eqref{eqLpIsomE2} is equivalent to the following one:
\begin{equation}\label{eqLpIsomE3}
A_j(D_{y},\eta,p)\tilde u_j(y,\eta)=\tilde f_j(y,\eta),\qquad y\in\theta_j,\
\eta\in\mathbb R^{n-2},
\end{equation}
\begin{equation}\label{eqLpIsomE4}
\sum\limits_{k,s}(B_{j\rho\mu ks}(D_{y},\eta,p)\tilde u_k)(\mathcal G_{j\rho
ks}y,\eta)|_{\gamma_{j\rho}}=0,\qquad \eta\in\mathbb R^{n-2},
\end{equation}
where $\tilde u_j(y,\eta)$ is the Fourier transform of $u_j(y,z)$ with respect to $z$.

Set $\tilde u_j(y,\eta)=|(\eta,p)|^{-2m}U_j(|(\eta,p)|y,\eta)$ and $\tilde
f_j(y,\eta)=F_j(|(\eta,p)|y,\eta)$. Then problem~\eqref{eqLpIsomE3}, \eqref{eqLpIsomE4}
takes the form
\begin{equation}\label{eqLpIsomE5}
A_j(D_{Y},\omega)U_j(Y,\eta)=F_j(Y,\eta),\qquad y\in\theta_j,\ \eta\in\mathbb R^{n-2},
\end{equation}
\begin{equation}\label{eqLpIsomE6}
\sum\limits_{k,s}(B_{j\rho\mu ks}(D_{Y},\omega)U_k)(\mathcal G_{j\rho
ks}Y,\eta)|_{\gamma_{j\rho}}=0,\qquad \eta\in\mathbb R^{n-2},
\end{equation}
where $\omega=(\eta,p)/|(\eta,p)|\in S^{n+d-3}$ and $Y=|(\eta,p)|y$.

It follows from the conditions of the lemma and from Theorem~\ref{t2.1} that
problem~\eqref{eqLpIsomE5}, \eqref{eqLpIsomE6} has a unique solution $U\in \mathcal
E_a^{l+2m}(\theta)$ for any right-hand side $\{F_j\}\in \mathcal E_a^{l}(\theta)$ and
$$
\|U\|_{\mathcal E_a^{l+2m}(\theta)}\le k_2\| \{F_j\}\|_{\mathcal E_a^l(\theta)},\qquad
\omega\in S^{n+d-3},
$$
where $k_2>0$ does not depend on $u$ and $\omega$.

Thus, the lemma will be proved if we show that
\begin{equation}\label{eqLpIsomE8}
\|\{f_j\}\|_{\mathcal E_a^l(\Theta)}^2\approx \int\limits_{\mathbb
R^{n-2}}|(\eta,p)|^{-2(a-l+1)}\|\{F_j(\cdot,\eta)\}\|_{\mathcal E_a^{l}(\theta)}^2 d\eta,
\end{equation}
\begin{equation}\label{eqLpIsomE7}
\|u\|_{\mathcal E_a^{l+2m}(\Theta)}^2\approx \int\limits_{\mathbb
R^{n-2}}|(\eta,p)|^{-2(a-l+1)}\|U(\cdot,\eta)\|_{\mathcal E_a^{l+2m}(\theta)}^2 d\eta;
\end{equation}
here the symbol $\approx$ between two expressions means that the first expression can be
estimated from the below and from the above by the second expression with positive
constants independent of $p\in S^{d-1}$.

3. Let us prove relation~\eqref{eqLpIsomE8}. Using the Parseval equality, we have
$$
\begin{aligned}
\|f_j\|_{E_a^l(\Theta_j)}^2&=\sum\limits_{|\alpha|+|\beta|\le
l}\,\int\limits_{\theta_j}\int\limits_{\mathbb
R^{n-2}}|y|^{2a}(|y|^{2(|\alpha|+|\beta|-l)}+1)|D^\alpha_y
D^\beta_z f_j(y,z)|^2 dydz\\
&=\sum\limits_{|\alpha|+|\beta|\le l}\,\int\limits_{\theta_j}\int\limits_{\mathbb
R^{n-2}}|y|^{2a}(|y|^{2(|\alpha|+|\beta|-l)}+1)|\eta^\beta|^2 |D^\alpha_y \tilde
f_j(y,\eta)|^2 dyd\eta.
\end{aligned}
$$
Using Fubini's theorem and making the change of variables $Y=|(\eta,p)|y$ for each fixed
$\eta$, we obtain
\begin{multline}\label{eqLpIsomE9}
\|f_j\|_{E_a^l(\Theta_j)}^2 =\sum\limits_{|\alpha|+|\beta|\le l}\,\int\limits_{\mathbb
R^{n-2}}\int\limits_{\theta_j}|(\eta,p)|^{-2(a+1-|\alpha|)}|Y|^{2a}\\
\times\big(|(\eta,p)|^{-2(|\alpha|+|\beta|-l)}|Y|^{2(|\alpha|+|\beta|-l)}+1\big)
|\eta^\beta|^2  |D^\alpha_Y  F(Y,\eta)|^2 dYd\eta\\
=\sum\limits_{|\alpha|\le l}\sum\limits_{\nu=0}^{l-|\alpha|}\,\int\limits_{\mathbb
R^{n-2}}\int\limits_{\theta_j}\Phi_{\alpha\nu}(\eta,p,Y)
  |D^\alpha_Y  F(Y,\eta)|^2 dYd\eta,
\end{multline}
where
$$
\Phi_{\alpha\nu}(\eta,p,Y)=\sum\limits_{|\beta|=\nu}|(\eta,p)|^{-2(a-l+1)}|Y|^{2a}\dfrac{|\eta^\beta|^2}{|(\eta,p)|^{2|\beta|}}
\big(|Y|^{2(|\alpha|+|\beta|-l)}+|(\eta,p)|^{2(|\alpha|+|\beta|-l)}\big).
$$

Note that $p\in S^{d-1}$ and $|\alpha|+|\beta|-l\le 0$, which implies that
$$
|(\eta,p)|^{2(|\alpha|+|\beta|-l)}\le 1.
$$
Therefore,
\begin{equation}\label{eqLpIsomE10}
|Y|^{2(|\alpha|+|\beta|-l)}+|(\eta,p)|^{2(|\alpha|+|\beta|-l)}\le
|Y|^{2(|\alpha|+|\beta|-l)}+1.
\end{equation}

Furthermore, since $p\in S^{d-1}$, we have
\begin{equation}\label{eq-new7.33}
\sum\limits_{|\beta|=\nu}\dfrac{|\eta^\beta|^2}{|(\eta,p)|^{2|\beta|}}\le k_3\qquad
(\eta\in\mathbb R^{n-2},\ p\in S^{d-1},\ \nu=0,\dots,l),
\end{equation}
where $k_3>0$ does not depend on $\eta$ and $p$.

Relations~\eqref{eqLpIsomE9}--\eqref{eq-new7.33} imply that
\begin{multline}\label{eq-new7.34}
\|f_j\|_{E_a^l(\Theta_j)}^2\\
 \le k_3 \sum\limits_{|\alpha|\le
l}\sum\limits_{\nu=0}^{l-|\alpha|}\,\int\limits_{\mathbb
R^{n-2}}\int\limits_{\theta_j}|(\eta,p)|^{-2(a-l+1)}|Y|^{2a}
(|Y|^{2(|\alpha|+\nu-l)}+1)|D^\alpha_Y  F_j(Y,\eta)|^2 dYd\eta.
\end{multline}

Clearly,
\begin{equation}\label{eq-new7.35}
|Y|^{2(|\alpha|-l)}+1\le\sum\limits_{\nu=0}^{l-|\alpha|}|Y|^{2(|\alpha|+\nu-l)}\le k_4(
|Y|^{2(|\alpha|-l)}+1),\qquad Y\in\mathbb R^2,
\end{equation}
where $k_4>0$ does not depend on $Y$.

Inequalities \eqref{eq-new7.34} and \eqref{eq-new7.35} imply that
\begin{equation}\label{eq-new7.36}
\|f_j\|_{E_a^l(\Theta_j)}^2\le k_3k_4 \int\limits_{\mathbb
R^{n-2}}|(\eta,p)|^{-2(a-l+1)}\|F_j(\cdot,\eta)\|_{ E_a^{l}(\theta_j)}^2 d\eta.
\end{equation}

Now we estimate the norm $\|f_j\|_{E_a^l(\Theta_j)}^2$ from below. To do this, we write
it as follows:
\begin{multline}\label{eq-new7.37}
\|f_j\|_{E_a^l(\Theta_j)}^2=\sum\limits_{|\alpha|\le l}\sum\limits_{\nu=0}^{l-|\alpha|}
\left(\,\int\limits_{|\eta|<1}\int\limits_{\theta_j}\Phi_{\alpha\nu}(\eta,p,Y)
  |D^\alpha_Y  F(Y,\eta)|^2 dYd\eta\right.\\
  + \left.\int\limits_{|\eta|>1}\int\limits_{\theta_j}\Phi_{\alpha\nu}(\eta,p,Y)
  |D^\alpha_Y  F(Y,\eta)|^2 dYd\eta\right).
\end{multline}

For $|\eta|<1$, we have
 \begin{multline*}
\sum\limits_{\nu=0}^{l-|\alpha|}\sum\limits_{|\beta|=\nu}\dfrac{|\eta^\beta|^2}{|(\eta,p)|^{2|\beta|}}
\big(|Y|^{2(|\alpha|+|\beta|-l)}+|(\eta,p)|^{2(|\alpha|+|\beta|-l)}\big)\\
\ge |Y|^{2(|\alpha|-l)}+|(\eta,p)|^{2(|\alpha|-l)} \ge k_5(|Y|^{2(|\alpha|-l)}+1),
\end{multline*}
where $k_5>0$ does not depend on $Y$, $\eta$, and $p$. Therefore,
\begin{multline}\label{eq-new7.38}
\sum\limits_{|\alpha|\le
l}\sum\limits_{\nu=0}^{l-|\alpha|}\int\limits_{|\eta|<1}\int\limits_{\theta_j}\Phi_{\alpha\nu}(\eta,p,Y)
|D^\alpha_Y  F(Y,\eta)|^2 dYd\eta\\
\ge k_5 \sum\limits_{|\alpha|\le l}\,
\int\limits_{|\eta|<1}\int\limits_{\theta_j}|(\eta,p)|^{-2(a-l+1)}|Y|^{2a}(|Y|^{2(|\alpha|-l)}+1)
  |D^\alpha_Y  F(Y,\eta)|^2 dYd\eta.
\end{multline}

For $|\eta|>1$, we have
\begin{equation}\label{eq-new7.39}
\sum\limits_{|\beta|=\nu}\dfrac{|\eta^\beta|^2}{|(\eta,p)|^{2|\beta|}}\ge k_6\qquad (p\in
S^{d-1},\ \nu=0,\dots,l),
\end{equation}
where $k_6>0$ does not depend on $\eta$ and $p$. It follows from~\eqref{eq-new7.35} and
\eqref{eq-new7.39} that
\begin{multline}\label{eq-new7.40}
 \sum\limits_{|\alpha|\le
l}\sum\limits_{\nu=0}^{l-|\alpha|}\,\int\limits_{|\eta|>1}\int\limits_{\theta_j}\Phi_{\alpha\nu}(\eta,p,Y)
  |D^\alpha_Y  F(Y,\eta)|^2 dYd\eta\\
 \ge k_6\sum\limits_{|\alpha|\le
l}\sum\limits_{\nu=0}^{l-|\alpha|}\,\int\limits_{|\eta|>1}\int\limits_{\theta_j}|(\eta,p)|^{-2(a-l+1)}|Y|^{2(a+|\alpha|+\nu-l)}
|D^\alpha_Y  F(Y,\eta)|^2 dYd\eta\\
 \ge k_6\sum\limits_{|\alpha|\le
l}\,\int\limits_{|\eta|>1}\int\limits_{\theta_j}|(\eta,p)|^{-2(a-l+1)}|Y|^{2a}(|Y|^{2(|\alpha|-l)}+1)
|D^\alpha_Y  F(Y,\eta)|^2 dYd\eta.
\end{multline}
 Inequalities \eqref{eq-new7.37}, \eqref{eq-new7.38}, and
\eqref{eq-new7.40} imply that
\begin{equation}\label{eq-new7.41}
\|f_j\|_{E_a^l(\Theta_j)}^2\ge k_7 \int\limits_{\mathbb
R^{n-2}}|(\eta,p)|^{-2(a-l+1)}\|F_j(\cdot,\eta)\|_{ E_a^{l}(\theta_j)}^2 d\eta,
\end{equation}
where $k_7>0$ does not depend on $p$.

Relation~\eqref{eqLpIsomE8} follows from~\eqref{eq-new7.36} and~\eqref{eq-new7.41}.
Similarly, one can prove~\eqref{eqLpIsomE7}.
\end{proof}

Now we can prove Theorem~\ref{tLpIsomH}.
\begin{proof}[Proof of Theorem~$\ref{tLpIsomH}$]
It is easy to see that $u(x)$ is a solution of the problem
$$
A_j(D_{y},D_{z},p)u_j(y,z)=f_j(y,z),\qquad (y,z)\in\Theta_j,
$$
$$
B_{j\rho\mu}(p)u\equiv\sum\limits_{k,s}(B_{j\rho\mu ks}(D_{y},D_{z},p)u_k)(\mathcal
G_{j\rho ks}y,z)|_{\Gamma_{j\rho}}=f_{j\rho\mu}(y,z),\qquad (y,z)\in\Gamma_{j\rho},
$$
iff the function $v(x)=u(|p|^{-1}x)$ is a solution of the problem
$$
A_j(D_{y},D_{z},p|p|^{-1})v_j(y,z)=|p|^{-2m}f_j(|p|^{-1}y,|p|^{-1}z),\qquad
(y,z)\in\Theta_j,
$$
$$
\begin{aligned}
B_{j\rho\mu}(p|p|^{-1})v\equiv\sum\limits_{k,s}(B_{j\rho\mu
ks}(D_{y},D_{z},p|p|^{-1})v_k)(\mathcal G_{j\rho
ks}y,z)|_{\Gamma_{j\rho}}&\\=|p|^{-m_{j\rho\mu}}f_{j\rho\mu}(|p|^{-1}y,|p|^{-1}z),\quad
(y,z)\in\Gamma_{j\rho}&,
\end{aligned}
$$
where $p\in\mathbb R^d\setminus\{0\}$.

Further, we shall use the following inequalities:
\begin{equation}\label{eqLpIsomH1}
\|v_j\|^2_{E_a^{l+2m}(\Theta_j)}\ge |p|^{2a+n-2(l+2m)}\n u_j
\n^2_{V_a^{l+2m}(\Theta_j)},\qquad p\in\mathbb R^d\setminus\{0\},
\end{equation}
\begin{multline}\label{eqLpIsomH2}
\|A_j(D_{y},D_{z},p|p|^{-1})v_j\|^2_{E_a^{l}(\Theta_j)}\\
\le k_1|p|^{2a+n-2(l+2m)}\n A_j(D_{y},D_{z},p)u_j \n^2_{ V_a^{l}(\Theta_j)},\qquad
p\in\mathbb R^d\setminus\{0\},
\end{multline}
\begin{multline}\label{eqLpIsomH3}
\|B_{j\rho\mu}(p|p|^{-1})v\|^2_{E_a^{l+2m-m_{j\rho\mu}-1/2}(\Gamma_{j\rho})}\\
\le k_2|p|^{2a+n-2(l+2m)}\n B_{j\rho\mu}(p)u
\n_{V_a^{l+2m-m_{j\rho\mu}-1/2}(\Gamma_{j\rho})}^2,\qquad p\in\mathbb R^d\setminus \{0\},
\end{multline}
where $k_1,k_2>0$ do not depend on $u$ and $p$. To obtain inequality~\eqref{eqLpIsomH1},
we introduce the new variables $x'=|p|^{-1}x$ and $y'=|p|^{-1}y$, where $x=(y,z)$,
$x'=(y',z')$, $y,y'\in\mathbb R^2$, and $z,z'\in\mathbb R^{n-2}$. Then we have
\begin{multline*}
\|v_j\|_{E_a^{l+2m}(\Theta_j)}^2 =\sum\limits_{|\alpha|\le
l+2m}\,\int\limits_{\Theta_j}|y|^{2a}(|y|^{2(|\alpha|-l-2m)}+1)|D_x^\alpha
u_j(|p|^{-1}x)|^2dx\\
 =|p|^{2a+n-2(l+2m)}\sum\limits_{|\alpha|\le
l+2m}\,\int\limits_{\Theta_j}|y'|^{2a}(|y'|^{2(|\alpha|-l-2m)}+|p|^{2(l+2m-|\alpha|)})|D^\alpha_{x'}
u_j(x')|^2dx'\\
 \ge |p|^{2a+n-2(l+2m)}\n u_j\n_{V_a^{l+2m}(\Theta_j)}^2.
\end{multline*}

Using Lemma~\ref{l7.1-new}, similarly to~\eqref{eqLpIsomH1}, we
derive~\eqref{eqLpIsomH2}.

To obtain inequality~\eqref{eqLpIsomH3}, one can use the equivalent norms in
$E_a^{k-1/2}(\Gamma_{j\rho})$ and $H_a^{k-1/2}(\Gamma_j\rho)$ given by
\begin{multline*}
\|u\|'_{E_a^{k-1/2}(\Gamma_{j\rho})}=\left(\sum\limits_{|\alpha|=k-1}\,\int\limits_{\Gamma_{j\rho}}\int\limits_{\Gamma_{j\rho}}\big||y_1|^a
D^\alpha u(x_1)-|y_2|^a D^\alpha
u(x_2)\big|^2\dfrac{dx_1dx_2}{|x_1-x_2|^n}\right. \\
\left. +\sum\limits_{|\alpha|\le
k-1}\,\int\limits_{\Gamma_{j\rho}}|y|^{2a}(|y|^{2(|\alpha|-k+1/2)}+1)|D^\alpha
u(x)|^2dx\right)^{1/2},
\end{multline*}
\begin{multline*}
\|u\|'_{H_a^{k-1/2}(\Gamma_{j\rho})}=\left(\sum\limits_{|\alpha|=k-1}\,\int\limits_{\Gamma_{j\rho}}
\int\limits_{\Gamma_{j\rho}}\big||y_1|^a D^\alpha u(x_1)-|y_2|^a D^\alpha
u(x_2)\big|^2\dfrac{dx_1dx_2}{|x_1-x_2|^n}\right. \\
\left. +\sum\limits_{|\alpha|\le
k-1}\,\int\limits_{\Gamma_{j\rho}}|y|^{2(a+|\alpha|-k+1/2)} |D^\alpha
u(x)|^2dx\right)^{1/2}
\end{multline*}
(see Lemmas~9.1 and 1.3 in~\cite{MP}) and Lemma~\ref{lInterpWeightedTraceGamma}.

Now the assertion of the theorem follows from Lemma~\ref{lLpIsomE} and
inequalities~\eqref{eqLpIsomH1}, \eqref{eqLpIsomH2}, and~\eqref{eqLpIsomH3}.
\end{proof}

Further, we prove an analog of Corollary~\ref{cort2.2'}. Introduce the linear bounded
operator $\mathcal L_g''(p): \mathcal V_a^{l+2m}(\Theta)\to\mathcal V_a^l(\Theta,\Gamma)$
by the formula
$$
\mathcal L_g''(p) v=\mathcal L_g(p) v+\eta(\mathcal L_g'(p)-\mathcal L_g(p))v,
$$
where $\eta$ is the same function as in Sec.~\ref{subsec2.3}.

\begin{corollary}\label{cor-theorem71a}
Let the conditions of Theorem~$\ref{tLpIsomH}$ hold. Then the operator $\mathcal
L_g''(p): \mathcal V_a^{l+2m}(\Theta)\to\mathcal V_a^l(\Theta,\Gamma)$ is an isomorphism
for all sufficiently small $\delta>0$ and $p\in\mathbb R^d\setminus\{0\}$ and
$$
\n u\n_{\mathcal V_a^{l+2m}(\Theta)}\le c\n \mathcal L_g''(p)u\n_{\mathcal
V_a^l(\Theta,\Gamma)},\qquad p\in\mathbb R^d\setminus\{0\},
$$
where   $c>0$ does not depend on $u$ and $p$.
\end{corollary}
\begin{proof}
By Theorem~\ref{tLpIsomH}, there is a bounded operator $\mathcal L_g^{-1}(p)$ and
estimate~\eqref{eqLpIsomH0} holds. We have
$$
\mathcal L_g''(p)\mathcal L_g^{-1}(p)=\mathcal I+\mathcal M(p),
$$
where $\mathcal I$ denotes the identity operator on $\mathcal V_a^l(\Theta,\Gamma)$ and
$$
\mathcal M(p)=\eta(\mathcal L_g'(p)-\mathcal L_g(p))\mathcal L_g^{-1}(p).
$$

It follows from~\eqref{eqBoundedTrace'} that
$$
\n u|_{\mathcal G_{j\rho k s}(\Gamma_i)} \n_{V_a^{l+2m-m_{j\rho\mu}-1/2}(\mathcal
G_{j\rho ks}(\Gamma_i))}\le k_1 \n u \n_{V_a^{l+2m-m_{j\rho\mu}}(\Theta_k)},
$$
where $k_1,k_2,{\dots}>0$ do not depend on $u$ and $p$. Therefore, similarly to the proof
of Corollary~\ref{cort2.2'}, we obtain
$$
\n \mathcal M(p)f\n_{\mathcal V_a^l(\Theta,\Gamma)}\le k_2\delta\n \mathcal
L_g^{-1}(p)f\n_{\mathcal V_a^{l+2m}(\Theta)}.
$$
Combining this inequality with~\eqref{eqLpIsomH0} yields
$$
\n \mathcal M(p)f\n_{\mathcal V_a^l(\Theta,\Gamma)}\le k_3\delta\n f\n_{\mathcal
V_a^l(\Theta,\Gamma)}.
$$
If $\delta>0$ is so small that $k_3\delta\le 1/2$, then there exists the inverse operator
$(\mathcal I+\mathcal M(p))^{-1}$ bounded in the norms $\n\cdot\n_{\mathcal
V_a^l(\Theta,\Gamma)}$ uniformly with respect to $p\in\mathbb R^d\setminus\{0\}$.

Clearly, the operator $\mathcal L_g^{-1}(p)(\mathcal I+\mathcal M(p))^{-1}$ is the right
inverse for $\mathcal L_g''(p)$ and
$$
\n \mathcal L_g^{-1}(p)(\mathcal I+\mathcal M(p))^{-1}f\n_{\mathcal
V_a^{l+2m}(\Theta)}\le k_4\n f\n_{\mathcal V_a^l(\Theta,\Gamma)},\qquad p\in\mathbb
R^d\setminus\{0\}.
$$
Similarly, one can prove the existence of a left inverse operator for $\mathcal
L_g''(p)$.
\end{proof}

\subsection{}
Now we fix an arbitrary point $g\in K_2$. Similarly to Sec.~\ref{subsec1.4}, we arrive at
the following model operator:
\begin{multline*}
\mathcal L_g(p): V_a^{l+2m}(\mathbb R_+^n)\to\mathcal V_a^l(\mathbb
R_+^n,\Gamma)\\
=V_a^l(\mathbb R_+^n)\times V_a^{l+2m-m_{i\mu}-1/2}(\mathbb R_-^{n-1})\times
V_a^{l+2m-m_{i\mu}-1/2}(\mathbb R_+^{n-1})
\end{multline*}
given by
$$
 \mathcal L_g(p)u=\big(A(D_{y},D_{z},p)u,\
B_{i\mu0}(D_{y},D_{z},p)u|_{\varphi=-\pi/2},\
B_{i\mu0}(D_{y},D_{z},p)u|_{\varphi=\pi/2}\big)
$$
(cf.~\eqref{1.12}). We assume that the space $V_a^{l+2m}(\mathbb R_+^n)$ is equipped with
the norm~\eqref{eqTripleNorm} and $\mathcal V_a^l(\mathbb R_+^n,\Gamma)$ with the norm
$$
\n f\n_{\mathcal V_a^l(\mathbb R_+^n,\Gamma)}=\big(\n f_0\n_{V_a^l(\mathbb R_+^n)}^2+\n
f_{-}\n_{V_a^{l+2m-m_{i\mu}-1/2}(\mathbb R_-^{n-1})}^2+ \n
f_{+}\n_{V_a^{l+2m-m_{i\mu}-1/2}(\mathbb R_+^{n-1})}^2\big)^{1/2},
$$
where $f=(f_0,f_-,f_+)$ and the norm $\n f_{\pm}\n_{V_a^{l+2m-m_{i\mu}-1/2}(\mathbb
R_\pm^{n-1})}$ is defined by~\eqref{eqTripleNormTrace}.

Analogously to Sec.~\ref{subsec2.5}, we consider the linear bounded operator $$\mathcal
L_g(\omega): E_a^{l+2m}(\mathbb R_+^2)\to\mathcal E_a^l(\mathbb R_+^2,\gamma)$$ given by
$$
\mathcal L_g(\omega) V=\big(A(D_{y},\omega)V,\ B_{i\mu0}(D_y,\omega)V|_{\mathbb R_-},\
B_{i\mu0}(D_y,\omega)V|_{\mathbb R_+}\big),
$$
where $\omega=(\eta,p)/|(\eta,p)|\in S^{n+d-3}$ (cf.~\eqref{2.19}
and~\eqref{eqLgOmegaP}).

Finally, we consider the analytic operator-valued function
$$\hat{\mathcal L}_g(\lambda): W^{l+2m}(-\pi/2,\pi/2)\to\mathcal
W^l[-\pi/2,\pi/2]$$ given by~\eqref{2.22}.

The following theorem is an analog of Theorem~\ref{t2.R^n_+Isomorphism} (cf. Theorem~9.1
and Corollary~9.1 in~\cite{MP}).
\begin{theorem}\label{tR^n_+IsomorphismP}
Let Conditions~$\ref{cond7.1}$ and $\ref{cond7.2}$ hold. Assume that the line
$\Im\lambda=a+1-l-2m$ contains no eigenvalues of $\hat{\mathcal L}_g(\lambda)$ and
$\dim\mathcal N(\mathcal L_g(\omega))=\codim\mathcal R(\mathcal L_g(\omega))=0$ for any
$\omega\in S^{n+d-3}$. Then the operator $\mathcal L_g(p): \mathcal V_a^{l+2m}(\mathbb
R^n_+)\to\mathcal V_a^l(\mathbb R^n_+,\Gamma)$ is an isomorphism for $p\in\mathbb
R^d\setminus\{0\}$ and
$$
\n u\n_{\mathcal V_a^{l+2m}(\mathbb R^n_+)}\le c\n \mathcal L_g(p)u\n_{\mathcal
V_a^l(\mathbb R^n_+,\Gamma)},\qquad p\in\mathbb R^d\setminus\{0\},
$$
where   $c>0$ does not depend on $u$ and $p$.
\end{theorem}
The proof  is similar to the proof of Theorem~\ref{tLpIsomH}.

\section{Solvability of Nonlocal Elliptic Problems with a Parameter}\label{sec8}

\subsection{}
In this section, we prove the unique solvability of nonlocal elliptic problems  with a
parameter in bounded domains.
\begin{lemma}\label{lI+T+Mp}
Let $H$ be a Hilbert space and $I$ the identity operator on $H$. Let $M_\varepsilon(p)$
and $S_\varepsilon(p)$ {\rm (}$\varepsilon>0$, $p\in\mathbb R^d$, and $|p|$ is
sufficiently large{\rm )} be families of bounded operators on $H$, such that
\begin{equation}\label{eqI+T+Mp}
\|M_\varepsilon(p)\|\le c_1\varepsilon,\qquad \|S_\varepsilon(p)\|\le c_2,\qquad
\|S^2_\varepsilon(p)\|\le c_3|p|^{-1},
\end{equation}
where $c_1,c_2,c_3>0$ do not depend on $\varepsilon$ and $p$. Then the operators
$$
L_\varepsilon(p)=I+M_\varepsilon(p)+S_\varepsilon(p)
$$
have bounded inverse operators $L_\varepsilon^{-1}(p)$ and the estimate
$$
\|L_\varepsilon^{-1}(p)\|\le c_4
$$
holds for sufficiently small $\varepsilon>0$ and sufficiently large $|p|$, where $c_4>0$
does not depend on $\varepsilon$ and $p$.
\end{lemma}
\begin{proof}
To prove the lemma, we will construct a right and a left inverse operators for
$L_\varepsilon(p)$. We have
$$
L_\varepsilon(p)\big(I-(M_\varepsilon(p)+S_\varepsilon(p))\big)=I-M_\varepsilon^2(p)-M_\varepsilon(p)
S_\varepsilon(p)-S_\varepsilon(p) M_\varepsilon(p)-S_\varepsilon^2(p).
$$
It follows from~\eqref{eqI+T+Mp} that
$$
\|S_\varepsilon^2(p)\|\le 1/6
$$
for sufficiently large $|p|$ and
$$
\|M_\varepsilon(p)\|\le \min\big(1\big/\sqrt{6},\ 1/(12 c_2)\big)
$$
for sufficiently small $\varepsilon$. Therefore,
$$
\|M_\varepsilon^2(p)+M_\varepsilon(p) S_\varepsilon(p)+S_\varepsilon(p)
M_\varepsilon(p)+S_\varepsilon^2(p)\|\le 1/2.
$$

Thus, the operators $\big(I-M_\varepsilon^2(p)-M_\varepsilon(p)
S_\varepsilon(p)-S_\varepsilon(p) M_\varepsilon(p)-S_\varepsilon^2(p)\big)^{-1}$ exist
and are uniformly bounded with respect to $\varepsilon$ and $p$. Combining this fact with
the uniform boundedness of the operators $I-(M_\varepsilon(p)+S_\varepsilon(p))$, we see
that the operators
$$
L_\varepsilon^{-1}(p)=\big(I-(M_\varepsilon(p)+S_\varepsilon(p))\big)\big(I-M_\varepsilon^2(p)-M_\varepsilon(p)
S_\varepsilon(p)-S_\varepsilon(p) M_\varepsilon(p)-S_\varepsilon^2(p)\big)^{-1}
$$
are the right inverse for the operators $L_\varepsilon(p)$ and
$$
\|L_\varepsilon^{-1}(p)\|\le c_4.
$$
Similarly, one can prove that there exist uniformly bounded left inverse operators for
the operators $L_\varepsilon(p)$.
\end{proof}

\begin{lemma}\label{l-t8.1}
Let Conditions~$\ref{cond7.1}$, $\ref{cond7.2}$, $\ref{cond1.3}$, and $\ref{cond1.4}$
hold. Assume that the line $\Im\lambda=a+1-l-2m$ contains no eigenvalues of
$\hat{\mathcal L}_g(\lambda)$ for any $g\in K$ and $\dim\mathcal N(\mathcal
L_g(\omega))=\codim\mathcal R(\mathcal L_g(\omega))=0$ for any $g\in K$ and $\omega\in
S^{n+d-3}$. Then there is a number $p_0>0$ such that the operator $\mathbf L^1(p)$,
$|p|\ge p_0$, has a bounded inverse and
\begin{equation}\label{8.1}
c_1\n\mathbf L^1(p) u\n_{\mathcal H_a^l(Q,\Gamma)}\le \n u \n_{H_a^{l+2m}(Q)}\le
c_2\n\mathbf L^1(p) u\n_{\mathcal H_a^l(Q,\Gamma)},\qquad |p|\ge p_0,
\end{equation}
where $c_1,c_2>0$ do not depend on $u$ and $p$.
\end{lemma}
\begin{proof}
1. The first inequality in~\eqref{8.1} follows from the definition of the norms
$\n\cdot\n$, from Lemma~\ref{l7.1-new}, and from estimate~\eqref{eqBoundedTrace}. To
prove the second inequality in~\eqref{8.1}, we repeat the proof of Lemma~\ref{l4.4},
replacing there the norms $\|\cdot\|$ by the norms $\n\cdot\n$, Corollary~\ref{cort2.2'}
and Theorem~\ref{t2.R^n_+Isomorphism} by Corollary~\ref{tLpIsomH} and
Theorem~\ref{tR^n_+IsomorphismP}, respectively, and the results on elliptic problems in
the interior of the domain and near a smooth part of the boundary by the corresponding
results on elliptic problems with a parameter~\cite{AV} and taking into account
estimate~\eqref{eqBoundedTrace}. Then we obtain the following a priori estimate:
$$
\n u \n_{H_a^{l+2m}(Q)}\le k_1(\n\mathbf L^1(p) u\n_{\mathcal H_a^l(Q,\Gamma)}+\n u
\n_{H_a^{l+2m-1}(Q)}),\qquad p\in\mathbb R^d\setminus\{0\},
$$
where $k_1>0$ does not depend on $u$ and $p$. Combining this estimate with the relation
$$
\n u \n_{H_a^{l+2m-1}(Q)}\le |p|^{-1}\n u \n_{H_a^{l+2m}(Q)}
$$
and taking $|p|\ge p'$, where $p'>0$ is sufficiently large, we obtain the second
inequality in~\eqref{8.1}.

2. It remains to prove the existence of a right inverse operator for $\mathbf L^1(p)$.

Using the notation from the proof of Lemma~\ref{l5.1}, we introduce the operator
$$
R_{\mathcal K_1}(p)f=\sum\limits_t(U^t)^{-1}\Big(\hat\xi^t(\mathcal
L_{g^t}''(p))^{-1}F^t\Big(\sum\limits_q\xi_q^tf\Big)\Big)
$$
(cf.~\eqref{eqVstavka5.10}). Similarly to~\eqref{eqVstavka5.18}, we  prove that
\begin{equation}\label{eq8.1_1}
\mathbf L^1(p) R_{\mathcal K_1}(p)f=\xi_0 f+ T_{\mathcal K_1}(p)f,
\end{equation}
where $T_{\mathcal K_1}(p): \mathcal H_a^l(Q,\Gamma)\to \mathcal H_a^l(Q,\Gamma)$ is a
bounded operator such that
\begin{equation}\label{eq8.1_2}
\n T_{\mathcal K_1}(p) f\n_{\mathcal H_a^l(Q,\Gamma)}\le c_1\n f\n_{\mathcal
H_a^l(Q,\Gamma)},
\end{equation}
\begin{equation}\label{eq8.1_3}
\n T_{\mathcal K_1}^2(p) f\n_{\mathcal H_a^l(Q,\Gamma)}\le c_2|p|^{-1}\n f\n_{\mathcal
H_a^l(Q,\Gamma)},
\end{equation}
and $c_1,c_2,{\dots}>0$ do not depend on $f$, $p$, and on the number $\varepsilon$ in the
definition of the function $\xi$, $|p|\ge p'$.

Estimate~\eqref{eq8.1_2} follows from Corollary~\ref{cor-theorem71a} and
inequalities~\eqref{7.3-new} and~\eqref{eqBoundedTrace}. Let us prove~\eqref{eq8.1_3}. By
analogy with the operators $\mathcal T_{j\rho\mu}^t$, $A_k''$, and $T_i$, $i=1,2,3$, from
the proof of Lemma~\ref{l5.1}, we consider the corresponding operators $\mathcal
T_{j\rho\mu}^t(p)$, $A_k''(p)$ and $T_i(p)$, $i=1,2,3$, depending on the parameter $p$.
First, we estimate the norm of $T_3^2(p)f$. Introduce functions $\psi_k^t\in
C_0^\infty(\Theta_k)$ such that $\psi_k^t(x')=1$ for $x'\in\Omega_k^t$. Using
inequality~\eqref{eqVstavka5.16} with the norms $\|\cdot\|$ replaced by the norms
$\n\cdot\n$,  the equivalence of the norms $\n\cdot\n$ in the subspaces of
$H_a^l(\Theta_k)$ and $W^l(\Theta_k)$ consisting of compactly supported functions
vanishing near the edge $\mathcal P$, Theorem~4.1 in~\cite{AV},
equality~\eqref{eqStar44}, Leibniz' formula, inequality~\eqref{7.3-new}, and
Corollary~\ref{cor-theorem71a}, we have
\begin{multline*}
\Bn\mathcal T_{j\rho\mu}^t(p)(\mathcal
L_{g^t}''(p))^{-1}F^t\Big(\sum\limits_q\xi_q^tT_3(p)
f\Big)\Bn_{V_a^{l+2m-m_{j\rho\mu}-1/2}(\Gamma_{j\rho})}\\
 \le
k_1\sum\limits_k \Bn A_k''(p)\Big(\psi_k^t\Big[(\mathcal
L_{g^t}''(p))^{-1}F^t\Big(\sum\limits_q\xi_q^tT_3(p)
f\Big)\Big]_k\Big)\Bn_{V_a^l(\Theta_k)}\\
\le k_2 \Bn  (\mathcal L_{g^t}''(p))^{-1}F^t\Big(\sum\limits_q\xi_q^tT_3(p) f\Big)
\Bn_{\mathcal V_a^{l+2m-1}(\Theta)} \le k_3 |p|^{-1}\n f\n_{\mathcal H_a^l(Q,\Gamma)},
\end{multline*}
where $p\in\mathbb R^d\setminus\{0\}$ and $k_1,\dots,k_4>0$ do not depend on $f$, $p$,
and $\varepsilon$.

The latter inequality implies that
$$
\n T_3^2(p)f\n_{\mathcal H_a^l(Q,\Gamma)}\le k_4|p|^{-1}\n f\n_{\mathcal
H_a^l(Q,\Gamma)},\qquad p\in\mathbb R^d\setminus\{0\}.
$$
Similarly, we estimate the norm of $T_1^2(p)f$. The analogous estimate for $T_2(p)f$ is
evident. Thus, we obtain inequality~\eqref{eq8.1_3} for $T_{{\mathcal
K}_1}(p)=T_1(p)+T_2(p)+T_3(p)$.

Let $\zeta$ be a function defined in~\eqref{eqZetaZetaTau}. Set $\zeta_1=1-\zeta$. Since
$\zeta(x)=1$ for $x\in\mathcal K_1^{2\varepsilon}$, it follows that
$\supp\zeta_1\subset\overline Q\setminus\mathcal K_1^{2\varepsilon}$. Introduce a
function $\hat\zeta_1\in C^\infty(\mathbb R^n)$ such that $\hat\zeta_1(x)=1$ for
$x\in\overline Q\setminus\mathcal K_1^{2\varepsilon}$ and
$\supp\hat\zeta_1\subset\overline Q\setminus\mathcal K_1^{\varepsilon}$.

Due to Theorem~10.4 in~\cite{MP}, there exists a bounded operator $\mathbf R^0(p)$ such
that $$\mathbf L_0(p)\mathbf R_0(p)f=f$$ for $f\in\mathcal H_a^l(Q,\Gamma)$, $\supp
f\subset\overline Q\setminus\mathcal K_1^{2\varepsilon}$, provided that $|p|\ge p''$,
where $p''\ge p'$ is sufficiently large. Thus, we can set
$$
R(p) f=R_{\mathcal K_1}(p)f+R_{\mathcal K_1}(p)(\eta f)+\hat\zeta_1\mathbf R^0(p)(\zeta_1
f),
$$
where $\eta(x)={\zeta_0(x)(1-\xi_0(x))}/{\xi_0(x)}$ for $x\in\mathcal K_1^{4\varepsilon}$
and $\eta(x)=0$ for $x\notin\mathcal K_1^{4\varepsilon}$ (cf.~\eqref{eqRf}). Since
$\supp\hat\zeta_1\subset\overline Q\setminus\mathcal K_1^{\varepsilon}$, we have
$$
B_{i\mu}^1(p)\big(\hat\zeta_1\mathbf R^0(p)(\zeta_1 f)\big)=0
$$
and hence
\begin{equation}\label{eq8.1_4}
\mathbf L^1(p) R(p) f=\mathbf L^1(p) R_{\mathcal K_1}(p)f+\mathbf L^1(p)R_{\mathcal
K_1}(p)(\eta f) +\mathbf L^0(p)\big(\hat\zeta_1\mathbf R^0(p)(\zeta_1 f)\big).
\end{equation}

Combining this relation with~\eqref{eq8.1_1} and using Leibniz' formula and
Lemmas~\ref{l7.1-new} and~\ref{lInterpWeightedTrace}, we obtain
\begin{multline*}
\mathbf L^1(p) R(p) f=\xi_0 f+T_{\mathcal K_1}(p)f+\zeta_0(1-\xi_0)f+T_{\mathcal
K_1}(p)(\eta f)
+\zeta_1 f+T(p)f\\
=f+T_{\mathcal K_1}(p)f+M(p)f+T(p)f
\end{multline*}
or, equivalently,
\begin{equation}\label{eq8.1_5}
\mathbf L^1(p) R(p)=\mathbf I+T_{\mathcal K_1}(p)+M(p)+T(p),
\end{equation}
 where
$$
M(p)f=T_{\mathcal K_1}(p)(\eta f),
$$
while $T(p):\mathcal H_a^l(Q,\Gamma)\to\mathcal H_a^l(Q,\Gamma)$ is a bounded operator
such that
\begin{equation}\label{eq8.1_6}
\n T(p)f\n_{\mathcal H_a^l(Q,\Gamma)}\le k_1|p|^{-1}\n f\n_{\mathcal H_a^l(Q,\Gamma)},
\end{equation}
where $k_1=k_1(\varepsilon)>0$ does not depend on $f$ and $p$.

By inequality~\eqref{eq8.1_2}, we have
$$
\n M(p)f\n_{\mathcal H_a^l(Q,\Gamma)}\le c_{3}\n \eta f\n_{\mathcal H_a^l(Q,\Gamma)}.
$$
However, $(1-\xi_0(x))/\xi_0(x)=0$ for $x\in\mathcal K_1$, while the function $\zeta_0$
is supported in $\mathcal K_1^{4\varepsilon}$ and satisfies the inequality
in~\eqref{eqVstavka5.19}. Therefore, it follows from the last estimate, from
Lemmas~\ref{lzeta0v} and~\ref{lZetaDeltaU} and from Remark~\ref{rzeta0v} that
\begin{equation}\label{eq8.1_7}
\n M(p)f\n_{\mathcal H_a^l(Q,\Gamma)}\le c_{4}\varepsilon\n f\n_{\mathcal
H_a^l(Q,\Gamma)}.
\end{equation}

By virtue of inequalities~\eqref{eq8.1_2}, \eqref{eq8.1_3}, and~\eqref{eq8.1_7} and
Lemma~\ref{lI+T+Mp}, the operator
\begin{equation}\label{eq8.1_8}
 \big(\mathbf I+T_{\mathcal K_1}(p)+M(p)\big)^{-1}
\end{equation}
exists and is bounded in the norms $\n\cdot\n_{\mathcal H_a^l(Q,\Gamma)}$, uniformly with
respect to $p$, $|p|\ge p'''$, where $p'''\ge p''$ is sufficiently large, provided that
$\varepsilon>0$ is a sufficiently small fixed number. Therefore, relation~\eqref{eq8.1_5}
is equivalent to the following one:
$$
\mathbf L^1(p) R(p)\big(\mathbf I+T_{\mathcal K_1}(p)+M(p)\big)^{-1}=\mathbf I+T'(p),
$$
where
$$
T'(p)=T(p)\big(\mathbf I+T_{\mathcal K_1}(p)+M(p)\big)^{-1}.
$$
By virtue of the uniform boundedness of the operator~\eqref{eq8.1_8} and
estimate~\eqref{eq8.1_6}, there is a sufficiently large number $p_0\ge p'''$ such that
$\n T'(p)\n\le 1/2$ for $|p|\ge p_0$ (recall that $\varepsilon$ is fixed) and hence
$$
\mathbf L^1(p) R(p)\big(\mathbf I+T_{\mathcal K_1}(p)+M(p)\big)^{-1}\big(\mathbf
I+T'(p)\big)^{-1}=\mathbf I.
$$

Thus, we have proved the existence of the right inverse operator for $\mathbf L^1(p)$,
$|p|\ge p_0$. Combining this with the second estimate in~\eqref{8.1}, we complete the
proof.
\end{proof}

\subsection{}

In this subsection, we generalize the result of the previous subsection to the operator
$\mathbf L(p)$.

\begin{theorem}\label{t8.2}
Let Conditions~$\ref{cond7.1}$--$\ref{cond7.4}$, $\ref{cond1.3}$, and $\ref{cond1.4}$
hold. Assume that the line $\Im\lambda=a+1-l-2m$ contains no eigenvalues of
$\hat{\mathcal L}_g(\lambda)$ for any $g\in K$ and $\dim\mathcal N(\mathcal
L_g(\omega))=\codim\mathcal R(\mathcal L_g(\omega))=0$ for any $g\in K$ and $\omega\in
S^{n+d-3}$. Then the following assertions are true{\rm :}
\begin{enumerate}
\item
there is a number $p_1>0$ such that the operator $\mathbf L(p)$, $|p|\ge p_1$, has a
bounded inverse and
\begin{equation}\label{eq8.8_0}
c_1\n\mathbf L(p) u\n_{\mathcal H_a^l(Q,\Gamma)}\le \n u \n_{H_a^{l+2m}(Q)}\le
c_2\n\mathbf L(p) u\n_{\mathcal H_a^l(Q,\Gamma)},\qquad |p|\ge p_1,
\end{equation}
where $c_1,c_2>0$ do not depend on $u$ and $p${\rm ;}
\item
the operator $\mathbf L(p)$ has the Fredholm property and $\ind\mathbf L(p)=0$ for
$p\in\mathbb R^d$.
\end{enumerate}
\end{theorem}

To prove Theorem~\ref{t8.2}, we preliminarily consider the operators
$$
L_t(p)=\mathbf L^1(p)+t\big(\mathbf L(p)-\mathbf L^1(p)\big),\qquad 0\le t\le 1.
$$
Clearly, $L_0(p)=\mathbf L^1(p)$, $L_1(p)=\mathbf L(p)$.

\begin{lemma}\label{lEstimateParamT}
Let the conditions of Theorem~$\ref{t8.2}$ hold. Then there is a number $p_1>0$ such that
the following estimates hold for $u\in H_a^{l+2m}(Q)$$:$
\begin{equation}\label{eqEstimateParamT0}
c_1\n L_t(p) u\n_{\mathcal H_a^l(Q,\Gamma)}\le \n u \n_{H_a^{l+2m}(Q)}\le c_2\n L_t(p)
u\n_{\mathcal H_a^l(Q,\Gamma)},\qquad |p|\ge p_1,\ 0\le t\le 1,
\end{equation}
where $c_1,c_2>0$ do not depend on $u$, $p$, and $t$.
\end{lemma}
\begin{proof}
The definition of the norms $\n\cdot\n$,  Lemma~\ref{l7.1-new},
inequality~\eqref{eqBoundedTrace}, and Conditions~\ref{cond7.3} and~\ref{cond7.4} imply
the  first estimate in~\eqref{eqEstimateParamT0}.

Let us prove the second estimate in~\eqref{eqEstimateParamT0}. Applying
Lemma~\ref{l-t8.1} and using Condition~\ref{cond7.3} and the fact that $0\le t\le 1$, we
have
\begin{multline}\label{eqEstimateParamT1}
\n u \n_{H_a^{l+2m}(Q)}\le k_1\n L_0(p) u\n_{\mathcal
H_a^l(Q,\Gamma)}\\
\le k_2\big(\n L_t(p) u\n_{\mathcal H_a^l(Q,\Gamma)}+\sum\limits_{i,\mu}\n
B_{i\mu}^2(p)\n_{H_a^{l+2m-m_{i\mu}-1/2}(\Gamma_i)}+\n u \n_{H_a^{l+2m-1}(Q)}\Big),
\end{multline}
where $|p|\ge p_0$ and $k_1,k_2,{\dots}>0$ do not depend on $u$, $p$, and $t$.

Further, we can repeat the proof of inequalities~\eqref{4.26'} and~\eqref{4.27}, using
Lemma~\ref{l7.1-new} and estimate~\eqref{eqBoundedTrace} and replacing the operators
$B_{i\mu}^2$ by $B_{i\mu}^2(p)$, the norms $\|\cdot\|$ by the norms $\n\cdot\n$,
Lemma~\ref{l4.6} by Condition~\ref{cond7.4}, and the results on elliptic problems in the
interior of the domain and near a smooth part of  boundary by the corresponding results
for elliptic problems with a parameter~\cite{AV}. Thus, we obtain
\begin{equation}\label{eqEstimateParamT2}
\n B_{i\mu}^2(p) u\n_{H_a^{l+2m-m_{i\mu}-1/2}(\Gamma_i)}\le k_3(\n L_t u\n_{\mathcal
H_a^l(Q,\Gamma)}+\n u\n_{H_a^{l+2m-1}(Q)}).
\end{equation}

Combining estimates~\eqref{eqEstimateParamT1} and~\eqref{eqEstimateParamT2} with
Lemma~\ref{l7.1-new} and taking $|p|\ge p_1$, where $p_1\ge p_0$ is sufficiently large,
we complete the proof.
\end{proof}

Now we will prove Theorem~\ref{t8.2}, using Lemmas~\ref{l-t8.1}
and~\ref{lEstimateParamT}, and the method of continuation along the parameter $t$.
\begin{proof}[Proof of Theorem~$\ref{t8.2}$]
1. Applying Lemmas~\ref{l-t8.1} and~\ref{lEstimateParamT}, we see that the operator
$$
L_t(p)=L_0(p)\big(\mathbf I+t L_0^{-1}(p)(L_1(p)-L_0(p))\big)
$$
has a bounded inverse for $0\le t\le t_1=c_1/(4 c_2)$ with the norm $\n L_t^{-1}(p)\n\le
c_2$. Therefore, the operator
$$
L_t(p)=L_{t_1}(p)\big(\mathbf I+(t-t_1) L_{t_1}^{-1}(p)(L_1(p)-L_0(p))\big)
$$
has a bounded inverse for $t_1\le t\le 2t_1$ with the norm $\n L_t^{-1}(p)\n\le c_2$.
Repeating this procedure finitely many times, we see that the operator $L_1(p)=\mathbf
L(p)$ has a bounded inverse. Estimate~\eqref{eq8.8_0} follows
from~\eqref{eqEstimateParamT0} for $t=1$.

2. Fix $\hat p\subset\mathbb R^d$ such that $|\hat p|\ge p_0$. In this case, there exists
a bounded operator $(\mathbf L^1(\hat p))^{-1}:\mathcal H_a^l(Q,\Gamma)\to H_a^{l+2m}(Q)$
due to Lemma~\ref{l-t8.1}. Thus, we have
$$
\mathbf L^1(p)=\mathbf L^1(\hat p)\big(\mathbf I+\mathbf T(p)\big),
$$
where
$$
\mathbf T(p)=(\mathbf L^1(\hat p))^{-1}(\mathbf L^1(p)-\mathbf L^1(\hat p)).
$$

Clearly, the operator $\mathbf L^1(p)-\mathbf L^1(\hat p): H_a^{l+2m-1}(Q)\to\mathcal
H_a^l(Q,\Gamma)$ is bounded. It follows from the compactness of the embedding
$H_a^{l+2m}(Q)\subset H_a^{l+2m-1}(Q)$ that the operator $\mathbf L^1(p)-\mathbf L^1(\hat
p): H_a^{l+2m}(Q)\to\mathcal H_a^l(Q,\Gamma)$ is compact. Therefore, the operator
$\mathbf T(p):H_a^{l+2m}(Q)\to H_a^{l+2m}(Q)$ is also compact. By Theorem~13.2
in~\cite{Kr}, the operator $\mathbf I+\mathbf T(p)$ has the Fredholm property and
$\ind\big(\mathbf I+\mathbf T(p)\big)=0$. Now, applying Theorem~12.2 in~\cite{Kr}, we see
that the operator $\mathbf L^1(p)$ has the Fredholm property and
$$
\ind\mathbf L^1(p)=\ind\mathbf L^1(\hat p)+\ind\big(\mathbf I+\mathbf T(p)\big)=0.
$$

Finally, we note that the fulfillment of Conditions~\ref{cond7.3} and~\ref{cond7.4}
implies the fulfillment of Conditions~\ref{cond6.1} and~\ref{cond6.2}, respectively.
Therefore, by Theorem~\ref{t6.3}, the operator $\mathbf L(p)$ has the Fredholm property
and
$$
\ind\mathbf L(p)=\ind\mathbf L^1(p)=0.
$$
\end{proof}

\subsection{}
In this subsection, we consider an example of an elliptic problem with a parameter,
having distributed nonlocal terms and satisfying Conditions~\ref{cond7.1}--\ref{cond7.4},
$\ref{cond1.3}$, and $\ref{cond1.4}$.
\begin{example}\label{exDistrNonlocalTermP}
1. In the notation of Example~\ref{exDistrNonlocalTerm}, we consider the following
nonlocal problem:
\begin{equation}\label{eqDistrNonlocalTermP1}
-\Delta u+e^{ih}p^2u=f_0(x),\qquad x\in Q,
\end{equation}
\begin{equation}\label{eqDistrNonlocalTermP2}
u|_{\Gamma_l}+B_l^1u+B_l^2u=f_l(x),\qquad x\in\Gamma_l,\ l=1,2,
\end{equation}
where $-\pi/2<h<\pi/2$ and $p\ge 0$.

For each point $g\in\mathcal K_1$, the operator $\mathcal L_g(p)=\mathcal L(p):
H_a^2(\Theta)\to\mathcal H_a^0(\Theta,\Gamma)$ given by~\eqref{eqLg(p)} takes the form
(cf.~\eqref{6.14'})
\begin{multline*}
\mathcal L(p) v=(-\Delta v+e^{ih}p^2,\\
v(\varphi,r,z)|_{\Gamma_{11}}-\alpha_1v(\varphi+\pi/4,r,z)|_{\Gamma_{11}},\
v(\varphi,r,z)|_{\Gamma_{12}}-\alpha_2v(\varphi-\pi/4,r,z)|_{\Gamma_{12}}).
\end{multline*}
Hence, the operators $$\mathcal L_g(\omega)=\mathcal L(\omega):E_a^2(\theta)\to\mathcal
E_a^0(\theta,\gamma),$$  $$\hat{\mathcal L}_g(\lambda)=\hat{\mathcal
L}(\lambda):W^2(-\pi/4,\pi/4)\to \mathcal W^0[-\pi/4,\pi/4]$$ given by~\eqref{eqLgOmegaP}
and~\eqref{2.4} have the form
\begin{multline*}
\mathcal L(\omega) V=(-\Delta_y V+(\omega_1^2+e^{ih}\omega_2^2)
V,\\
V(\varphi,r)|_{\gamma_{11}}-\alpha_1V(\varphi+\pi/4,r)|_{\gamma_{11}},\
V(\varphi,r)|_{\gamma_{12}}-\alpha_2V(\varphi-\pi/4,r)|_{\gamma_{12}})
\end{multline*}
and
$$
\hat{\mathcal L}(\lambda) w=(-w_{\varphi\varphi}+\lambda^2w,\ w(-\pi/4)-\alpha_1w(0),\
w(\pi/4)-\alpha_2w(0)),
$$
respectively, where $\omega=(\omega_1,\omega_2)\in S^1$.

Let the numbers $a,\alpha_1,\alpha_2$ satisfy the following relations:
\begin{equation}\label{eqDistrNonlocalTermP2'}
0\le a\le2,\qquad 0<|\alpha_1+\alpha_2|<2,\qquad
\pi/4<\arctan\sqrt{4(\alpha_1+\alpha_2)^{-2}-1}.
\end{equation}

2. We claim that there is a number $h_1=h_1(\alpha_1,\alpha_2)>0$ such that the operator
$\mathcal L(\omega)$, $\omega\in S^1$, is an isomorphism for $|h|\le h_1$. To prove this
fact, one must repeat the reasoning of Example~\ref{ex2.1}, where the sesquilinear
form~\eqref{eqSesquilForm} is replaced by the form
$$
b_{\mathcal R}[u,v]=\int\limits_\theta\Big(\sum\limits_{i=1,2}(\mathcal R_\theta
u)_{y_i}\overline{v_{y_i}}+(\omega_1^2+e^{ih}\omega_2^2)\mathcal R_\theta
u\overline{v}\Big)dy
$$
with the same domain $\Dom(b_{\mathcal R})=\mathaccent23W^1(\theta)$. Let us show that
this sesquilinear form remains to be a closed sectorial form.

It follows from the Schwarz inequality and from~\eqref{2.10} that
\begin{equation}\label{eqDistrNonlocalTermP3}
|b_{\mathcal R}[u,v]|\le
k_1\|u\|_{\mathaccent23W^1(\theta)}\|v\|_{\mathaccent23W^1(\theta)},
\end{equation}
where $k_1>0$ does not depend on $u$ and $v$.

Introduce the isomorphism $\mathcal U:L_2(\theta)\to L_2(\theta_1)\times L_2(\theta_1)$
by the formula
$$
(\mathcal U u)_i(y)=u(\varphi+(i-1)d/2,r),\qquad y\in\theta_1,\ i=1,2,
$$
and let $R_1=\begin{pmatrix} 1&\alpha_1\\ \alpha_2&1
\end{pmatrix}$. Then, using~\eqref{2.10} and~\eqref{2.11'}, we obtain
\begin{multline}\label{eqDistrNonlocalTermP3'}
{\rm Re\,}b_{\mathcal R}[u,u]=
\int\limits_{\theta_1}\Big\{\sum\limits_{i}\Big(\dfrac{(R_1+R_1^*)}{2}(\mathcal U
u_{y_i}),\mathcal U u_{y_i}\Big)_{\mathbb C^2}
\\
+\omega_1^2\Big(\dfrac{(R_1+R_1^*)}{2}\mathcal U u,\mathcal U u\Big)_{\mathbb
C^2}+\omega_2^2\Big(\dfrac{e^{ih}R_1+(e^{ih}R_1)^*)}{2}\mathcal U u,\mathcal U
u\Big)_{\mathbb C^2}\Big\}dy.
\end{multline}
Since $|\alpha_1+\alpha_2|<2$, it follows that the matrix
\begin{equation}\label{eqDistrNonlocalTermP4}
R_1+R_1^*=\begin{pmatrix} 2&\alpha_1+\alpha_2\\
\alpha_1+\alpha_2 & 2\end{pmatrix}
\end{equation}
is positively definite. Using the Silvester criterion, we will show that the matrix
\begin{equation}\label{eqDistrNonlocalTermP5}
e^{ih}R_1+(e^{ih}R_1)^*=\begin{pmatrix} e^{ih}+e^{-ih}&e^{ih}\alpha_1+e^{-ih}\alpha_2\\
e^{ih}\alpha_2+e^{-ih}\alpha_1 & e^{ih}+e^{-ih}\end{pmatrix}
\end{equation}
is also positively definite. Since $-\pi/2<h<\pi/2$, it follows that $e^{ih}+e^{-ih}>0$.
Thus, we have to prove that $\det(e^{ih}R_1+(e^{ih}R_1)^*)>0$. Let $e^{ih}=\mu+i\nu$,
where $\mu>0$. Using the equality $\nu^2=1-\mu^2$, we obtain
$$
\det(e^{ih}R_1+(e^{ih}R_1)^*)=4\mu^2-[(\alpha_1+\alpha_2)^2+(4\mu^2-4)\alpha_1\alpha_2].
$$
Since $|\alpha_1+\alpha_2|<2$, it follows that, for $h=0$ (i.e., $\mu=1$),
$$
\det(e^{ih}R_1+(e^{ih}R_1)^*)=4-(\alpha_1+\alpha_2)^2>0.
$$
Therefore, there is a number $h_1=h_1(\alpha_1,\alpha_2)>0$ such that
$$
\det(e^{ih}R_1+(e^{ih}R_1)^*)>0,\qquad |h|\le h_1.
$$

It follows from~\eqref{eqDistrNonlocalTermP3'}, from the positive definiteness of the
matrices~\eqref{eqDistrNonlocalTermP4} and~\eqref{eqDistrNonlocalTermP5}, and from the
relation $\omega_1^2+\omega_2^2=1$ that
\begin{equation}\label{eqDistrNonlocalTermP6}
{\rm Re\,}b_{\mathcal R}[u,u]\ge k_2\int\limits_{\theta_1}\Big\{\sum\limits_{i}(\mathcal
U u_{y_i},\mathcal U u_{y_i})_{\mathbb C^2}+(\mathcal U u,\mathcal U u)_{\mathbb
C^2}\Big\}dy= k_2\|u\|^2_{\mathaccent23W^1(\theta)},
\end{equation}
where $k_2>0$ does not depend on $u$.

Inequalities~\eqref{eqDistrNonlocalTermP3} and~\eqref{eqDistrNonlocalTermP6} imply that
$b_{\mathcal R}$ is a closed sectorial form on $L_2(\theta)$, with the domain
$\Dom(b_{\mathcal R})=\mathaccent23W^1(\theta)$ and vertex $k_2>0$
(see~\cite[Chap.~6]{Kato}).

Thus, repeating the reasoning of Example~\ref{ex2.1}, we see that $\mathcal L(\omega)$,
$\omega\in S^1$, is an isomorphism for the above $a,\alpha_1,\alpha_2$, and $h$.
Moreover, since $h$ and $\omega$ belong to the compact sets, it follows that the
inequality
\begin{equation}\label{eqDistrNonlocalTermP7}
\|V\|_{E_a^{2}(\theta)}\le k_3\|\mathcal L(\omega)V\|_{\mathcal
E_a^0(\theta,\gamma)},\qquad \omega\in S^{1},
\end{equation}
holds with a constant $k_3>0$ which does not depend\footnote{Otherwise, denoting
$L_h(\omega)=L(\omega)$, we see that there are sequences $h^{(k)},\omega^{(k)},$ and
$V^{(k)}$, $k=1,2,\dots$, such that $h^{(k)}\to h$, $\omega^{(k)}\to\omega$, $\|\mathcal
L_{h^{(k)}}(\omega^{(k)})V^{(k)}\|_{\mathcal E_a^0(\theta,\gamma)}\to 0$, and
$\|V^{(k)}\|_{E_a^{2}(\theta)}=1$. This leads to contradiction because we have
$1=\|V^{(k)}\|_{E_a^{2}(\theta)}\le c \|\mathcal L_h(\omega)V^{(k)}\|_{\mathcal
E_a^0(\theta,\gamma)}\le c (\|\mathcal L_{h^{(k)}}(\omega^{(k)})V^{(k)}\|_{\mathcal
E_a^0(\theta,\gamma)}+\|(\mathcal L_{h^{(k)}}(\omega^{(k)})V^{(k)}-\mathcal
L_h(\omega)V^{(k)})\|_{\mathcal E_a^0(\theta,\gamma)})\to 0$.} on $V$, $\omega$, and $h$.

Now Theorem~\ref{t8.2} implies that the operator $\mathcal L(p)$ is also an isomorphism
for $p\ge p_0$, where $p_0>0$. Moreover, it follows from~\eqref{eqDistrNonlocalTermP7}
that the inequality
\begin{equation}\label{eqDistrNonlocalTermP8}
\|u\|_{H_a^{2}(\Theta)}\le k_4\|\mathcal L(p)u\|_{\mathcal H_a^0(\Theta,\Gamma)},\qquad
p>p_0,
\end{equation}
holds with a constant $k_4>0$ which does not depend on $u$ and $h$.

3. Similarly to~\eqref{6.18} and~\eqref{6.19}, using Lemma~\ref{l7.1-new},
estimate~\eqref{eqBoundedTrace}, and the interpolation inequalities in Sobolev spaces
(see~\cite[Chap.~1, Sec.~1]{AV}), we obtain
$$
\n B_l^2u\n_{H_a^{3/2}(\Gamma_l)}\le k_5\n u\n_{H_a^2(Q\setminus\overline{\mathcal
K_1^{2\varkappa}})},
$$
$$
\n B_l^2u\n_{H_a^{3/2}(\Gamma_l\setminus\overline{\mathcal K_1^\varkappa})}\le k_6\n
u\n_{H_a^2(Q_\varkappa)},
$$
where $a>1$ and $k_5,k_6>0$ do not depend on $u$. Thus, the operators $B_l^2$ satisfy
Condition~\ref{cond7.4} for $a>1$.

We consider the linear bounded operators
$$
\mathbf L(p),\mathbf L^1(p): H_a^2(Q)\to H_a^0(Q)\times H_a^{3/2}(\Gamma_1)\times
H_a^{3/2}(\Gamma_2)
$$
given by
$$
\mathbf L(p) u=\{-\Delta u+e^{ih}p^2u,\ u|_{\Gamma_l}+B_l^1u+B_l^2u\},\qquad \mathbf
L^1(p) u=\{-\Delta u+e^{ih}p^2u,\ u|_{\Gamma_l}+B_l^1u\}.
$$

The two results below follow from Lemma~\ref{l-t8.1} and Theorem~\ref{t8.2}.
\begin{corollary}\label{cor-tEx8.1}
Let the numbers $a,\alpha_1,\alpha_2$ satisfy relations~\eqref{eqDistrNonlocalTermP2'}.
Then there exist a number $h_1=h_1(\alpha_1,\alpha_2)>0$ and a number $p_0>0$,
independent of $h$, such that the operator $\mathbf L^1(p)$, $p\ge p_0$, $|h|\le h_1$,
has a bounded inverse and
$$
c_1\n\mathbf L^1(p) u\n_{\mathcal H_a^0(Q,\Gamma)}\le \n u \n_{H_a^{2}(Q)}\le
c_2\n\mathbf L^1(p) u\n_{\mathcal H_a^0(Q,\Gamma)},
$$
where $c_1,c_2>0$ do not depend on $u$, $h$, and $p$.
\end{corollary}
\begin{corollary}\label{cor-tEx8.2}
Let $1<a\le 2$, while the numbers $\alpha_1,\alpha_2$, and $h_1$ be the same as in
Corollary~$\ref{cor-tEx8.1}$. Then there is a number $p_1>0$, independent of $h$, such
that the operator $\mathbf L(p)$, $p\ge p_1$, $|h|\le h_1$, has a bounded inverse and
$$
c_1\n\mathbf L(p) u\n_{\mathcal H_a^0(Q,\Gamma)}\le \n u \n_{H_a^{2}(Q)}\le c_2\n\mathbf
L(p) u\n_{\mathcal H_a^0(Q,\Gamma)},
$$
where $c_1,c_2>0$ do not depend on $u$, $h$, and $p$.
\end{corollary}

\begin{remark}\label{remEx8_1}
Let $\alpha_1=\alpha_2$ and $|\alpha_1|<1$. In this case, Corollaries~\ref{cor-tEx8.1}
and~\ref{cor-tEx8.2} are true for any $h_1$ satisfying the relation $0<h_1<\pi/2$.
Indeed, the matrix~\eqref{eqDistrNonlocalTermP5} remains positively definite because
$$
e^{ih}+e^{-ih}=2\mu>0,\qquad \det(e^{ih}R_1+(e^{ih}R_1)^*)=4\mu^2(1-\alpha_1^2)>0,
$$
where $\mu=\Re e^{ih}>0$. Therefore, the form $b_{\mathcal R}$ remains to be a closed
sectorial form on $L_2(\theta)$ with the domain $\Dom(b_{\mathcal
R})=\mathaccent23W^1(\theta)$ and vertex $k_2>0$. Further consideration does not change.
\end{remark}

4. Consider the unbounded operators $\mathbf A, \mathbf A^1: H_a^0(Q)\to H_a^0(Q)$ given
by
$$
\mathbf A u=-\Delta u,\qquad u\in\Dom(\mathbf A)=\{u\in H_a^2(Q):\
u|_{\Gamma_l}+B_l^1u+B_l^2 u=0\},
$$
$$
\mathbf A^1 u=-\Delta u,\qquad u\in\Dom(\mathbf A)=\{u\in H_a^2(Q):\
u|_{\Gamma_l}+B_l^1u=0\}.
$$
Corollary~\ref{cor-tEx8.1} implies the following result (see Fig.~\ref{fig8.1}).
\begin{figure}[ht]
{ \hfill\epsfbox{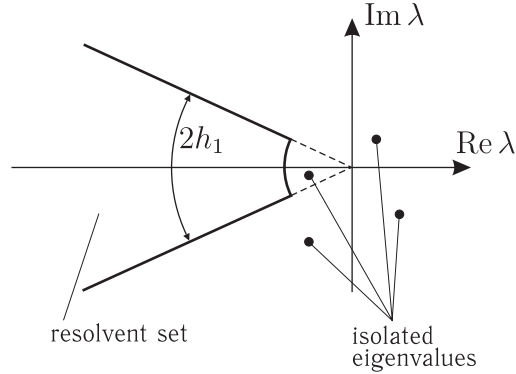}\hfill\ } \caption{Spectra of the operators $\mathbf A^1$ and
$\mathbf A$}\label{fig8.1}
\end{figure}

\begin{corollary}\label{corEx8.1}
Let the conditions of Corollary~$\ref{cor-tEx8.1}$ hold. Then the following assertions
are true{\rm :}
\begin{enumerate}
\item
the spectrum $\sigma(\mathbf A^1)$ is discrete\footnote{This means that the spectrum
consists of a finite or a countable set of isolated eigenvalues of finite
multiplicity.}{\rm ;}
\item
there exist numbers $h_1=h_1(\alpha_1,\alpha_2)>0$ and $\lambda_1>0$ such that
$\sigma(\mathbf A^1)\subset\mathbb C\setminus\Omega_{h_1,\lambda_1}$, where
$$
\Omega_{h_1,\lambda_1}=\{\lambda\in\mathbb C:\ |\arg(\lambda-\pi)|\le h_1,\ |\lambda|\ge
\lambda_1\};
$$
\item the estimate
$$
\|(\mathbf A^1-\lambda\mathbf I)^{-1}\|_{H_a^0(Q)\to H_a^0(Q)}\le c_1/|\lambda|,\qquad
\lambda\in\Omega_{h_1,\lambda_1},
$$
holds with a constant $c_1=c_1(h_1)>0$.
\end{enumerate}
\end{corollary}
\begin{proof}
Set $-\lambda=e^{ih}p^2$ and $\lambda_1=p_0^2$, where $p_0$ is the constant occurring in
Corollary~\ref{cor-tEx8.1}. In this case, assertions~2 and~3 follow from
Corollary~\ref{cor-tEx8.1}. By the same corollary, the operator $(\mathbf
A-\lambda\mathbf I)^{-1}:H_a^0(Q)\to H_a^2(Q)$ is bounded for
$\lambda\in\Omega_{h_1,\lambda_1}$. Combining this fact with the compactness of the
embedding $H_a^2(Q)\subset H_a^0(Q)$, we see that the resolvent $(\mathbf
A-\lambda\mathbf I)^{-1}:H_a^0(Q)\to H_a^0(Q)$ is compact for
$\lambda\in\Omega_{h_1,\lambda_1}$. Therefore, by Theorem~6.29 in~\cite[Chap.~3,
Sec.~6]{Kato}, assertion~1 is true.
\end{proof}

Using Corollary~\ref{cor-tEx8.2} instead of Corollary~\ref{cor-tEx8.1}, we obtain the
following result (see Fig.~\ref{fig8.1}).
\begin{corollary}\label{corEx8.2}
Let the conditions of Corollary~$\ref{cor-tEx8.2}$ hold. Then the following assertions
are true{\rm :}
\begin{enumerate}
\item
the spectrum $\sigma(\mathbf A)$ is discrete{\rm ;}
\item
there exists a number $\lambda_2>0$ such that $\sigma(\mathbf A)\subset\mathbb
C\setminus\Omega_{h_1,\lambda_2}${\rm ;}
\item the estimate
$$
\|(\mathbf A-\lambda\mathbf I)^{-1}\|_{H_a^0(Q)\to H_a^0(Q)}\le c_2/|\lambda|,\qquad
\lambda\in\Omega_{h_1,\lambda_2},
$$
holds with a constant $c_2=c_2(h_1)>0$, where $h_1>0$ is the constant from
Corollary~$\ref{corEx8.1}$.
\end{enumerate}
\end{corollary}

\begin{remark}
Let $\alpha_1=\alpha_2$ and $|\alpha_1|<1$. In this case, Corollaries~\ref{corEx8.1}
and~\ref{corEx8.2} are true for any $h_1$ satisfying the relation $0<h_1<\pi/2$ (cf.
Remark~\ref{remEx8_1}).
\end{remark}

The following questions are unanswered. Do there exist a number $h_1$, $\pi/2\le
h_1<\pi$, and numbers $\lambda_1,\lambda_2>0$ such that
\begin{equation}\label{eqQuestion}
\sigma(\mathbf A^1)\subset\mathbb C\setminus\Omega_{h_1,\lambda_1},\qquad \sigma(\mathbf
A)\subset\mathbb C\setminus\Omega_{h_1,\lambda_2}?
\end{equation}
Can one find, for any $h_1$, $0<  h_1<\pi$, numbers $\lambda_1,\lambda_2>0$ such that
relations~\eqref{eqQuestion} hold (cf. Problem~13.1 in~\cite[Sec.~13]{SkBook})?
\end{example}

\end{document}